\newcommand{\Aa}{\mathfrak{A}}
\newcommand{\AAA}{\mathcal{A}}

\newcommand{\Bb}{\mathfrak{B}}
\newcommand{\BB}{\mathcal{B}}
\newcommand{\Cc}{\mathfrak{C}}
\newcommand{\CC}{\mathcal{C}}
\newcommand{\E}{\mathbb{E}}
\newcommand{\eE}{\mathfrak{e}}
\newcommand{\EE}{\mathcal{E}}

\newcommand{\FF}{\mathcal{F}}
\newcommand{\Hh}{\mathfrak{H}}
\newcommand{\HH}{\mathcal{H}}
\newcommand{\KK}{\mathcal{K}}
\newcommand{\LL}{\mathcal{L}}
\newcommand{\Mm}{\mathfrak{M}}
\newcommand{\MM}{\mathcal{M}}
\newcommand{\N}{\mathbb{N}}

\newcommand{\pP}{\mathfrak{p}}
\newcommand{\PP}{\mathcal{P}}
\newcommand{\Q}{\mathbb{Q}}
\newcommand{\RR}{\mathcal{R}}
\newcommand{\Ss}{\mathfrak{S}}

\newcommand{\SSSS}{\mathscr{S}}

\newcommand{\WW}{\mathcal{W}}
\newcommand{\Xx}{\mathfrak{X}}
\newcommand{\re}{\mathbb{R}}
\newcommand{\com}{\mathbb{C}}
\newcommand{\one}{\mathbf{1}}
\newcommand{\la}{\langle}
\newcommand{\ra}{\rangle}
\newcommand{\inv}{^{-1}}
\newcommand{\realp}{\text{Re} \,}
\newcommand{\LRA}{\quad \Leftrightarrow
\quad}
\newcommand{\sa}[1]{\stackrel{#1}{\rightarrow}}
\newcommand{\longsa}[1]{\stackrel{#1}{\longrightarrow}}
\newcommand{\dss}[2]{{#1}_{_{#2}}}

\newcommand{\hstar}{\, \widehat{\star} \,}
\newcommand{\tstar}{\, \widetilde{\star} \,}
\newcommand{\lla}{\langle \! \langle}
\newcommand{\rra}{\rangle \! \rangle}

\documentclass{memo-l}

\usepackage{amssymb}
\usepackage{enumitem}
\usepackage[all]{xypic}
\usepackage{mathrsfs}

\newtheorem{theorem}{Theorem}[section]
\newtheorem{lemma}[theorem]{Lemma}
\newtheorem{corollary}[theorem]{Corollary}
\newtheorem{proposition}[theorem]{Proposition}
\theoremstyle{definition}
\newtheorem{definition}[theorem]{Definition}
\newtheorem{example}[theorem]{Example}

\theoremstyle{remark}
\newtheorem{remark}[theorem]{Remark}
\numberwithin{section}{chapter}
\numberwithin{equation}{chapter}
\makeindex

\begin{document}

\frontmatter

\title{Unital Dilations of Completely Positive Semigroups: From Combinatorics to Continuity}

%    Remove any unused author tags.

%    author one information
\author{David J.\ Gaebler}
%\curraddr{Department of Mathematics, Hillsdale College}
%\email{dgaebler@hillsdale.edu}

%    \date is required; it is the date received by the editor.
\date{September 13, 2013}

\subjclass[2010]{Primary 46L57}
%    Recognition of the 2010 edition of the Mathematics Subject
%    Classification requires a version of amsbook.cls from July 2009
%    or later.  If "2010" is not recognized, please upgrade.

\keywords{Dilation, completely positive semigroup,
liberation, Sauvageot product, moment polynomial}

%\dedicatory{To Leslie}

\begin{abstract}
Semigroups of completely positive maps arise naturally both in noncommutative stochastic processes and in the dynamics of open quantum systems.  Since its inception in the 1970's, the study of
completely positive semigroups has included among its
central topics the dilation of a completely positive semigroup to an endomorphism semigroup.  In quantum dynamics, this amounts to embedding a given open system inside some closed system, while in noncommutative probability, it corresponds to the construction of a Markov process from its transition probabilities.  In addition to the existence of dilations, one is interested in what properties of the original semigroup (unitality, various kinds of continuity) are preserved.

Several authors have proved the existence of dilations, but in general, the dilation achieved has been non-unital; that is, the unit of the original algebra is embedded as a proper projection in the dilation algebra.  A unique approach due to Jean-Luc Sauvageot overcomes this problem, but leaves unclear the continuity of the dilation semigroup.  The major purpose of this thesis, therefore, is to further develop Sauvageot's theory in order to prove the existence of continuous unital dilations.
This existence is proved in Theorem \ref{thmkahuna}, the central
result of the paper.

The dilation depends on a modification of free probability theory, and in particular on a combinatorial property akin to free independence.  This property is implicit in some of Sauvageot's original calculations, but a secondary goal of this paper is to present it as its own object of study, which we do in chapter \ref{chapliberation}.

The content discussed here is based on
 \cite{Sauvageot}, and the present exposition is
 substantially the same as in the author's Ph.D.\ thesis \cite{GaeblerThesis}.
\end{abstract}

\maketitle

\tableofcontents

%    Include unnumbered chapters (preface, acknowledgments, etc.) here.
\chapter*{Preface: Background and Terminology} \label{chappreface}
Following \cite{Sakai}, we distinguish between W$^*$-algebras, which
are abstractly defined as C$^*$-algebras having a Banach-space predual (necessarily unique, as it turns out), and
von Neumann algebras, which are concretely defined as weakly closed
self-adjoint subalgebras of $B(H)$ for some Hilbert space $H$.
In this convention, every von Neumann algebra is also a W$^*$-algebra
(with predual equal to a quotient of the predual $B(H)_* \simeq L^1(H)$),
whereas every W$^*$-algebra is isomorphic to some von Neumann algebra
(\cite{Sakai} 1.16.7).  We depart somewhat from Sakai in
referring to the weak-* topology on a W$^*$-algebra as
the \textbf{ultraweak topology}, which he calls the $\sigma$-topology
or weak topology, and the topology induced by the seminorms
$x \mapsto \phi(x^* x)$ for positive weak-* continuous functionals
$\phi$ as the \textbf{ultrastrong topology}, which he calls
the strong topology or s-topology.
In the case of a von Neumann algebra, these topologies coincide
with the ultraweak and ultrastrong operator topologies
as usually defined (\cite{Sakai} 1.15.6), and hence also with
the weak and strong operator topologies on bounded subsets
(\cite{Sakai} 1.15.2).  Because of this latter fact, we sometimes drop the ``ultra'' and refer merely to the \textbf{weak} and \textbf{strong} topologies when working on a bounded subset of a W$^*$-algebra.
Among the properties of these topologies that we will need are the following:
\begin{itemize}
    \item Multiplication is separately continuous in both the ultraweak and ultrastrong topologies.  However, it is jointly continuous in neither.  On bounded sets, multiplication is jointly strongly continuous,
        but not jointly weakly continuous.
    \item The adjoint map $x \mapsto x^*$ is ultraweakly continuous,
    but not ultrastrongly nor even strongly continuous.
    \item On bounded subsets, one may relate the weak and
    strong topologies as follows: $x_\nu \to x$ strongly
    iff $x_\nu \to x$ weakly and $x_\nu^* x_\nu \to x^* x$
    weakly.
    \item The \textbf{Kaplansky density theorem}: If $A$ is a W$^*$-algebra and $A_0 \subset A$ an ultraweakly dense
        *-subalgebra, then the unit ball of $A_0$ is
        strong-* dense in the unit ball of $A$.  In the case
        of a von Neumann algebra, the hypothesis of ultraweak
        density may be replaced by WOT-density.
\end{itemize}

A linear map between W$^*$-algebras
which is continuous with respect to their ultraweak topologies
is called \textbf{normal}; if the map in question
is positive, this is equivalent
to the property of preserving
upward-convergent nets (in this case weak and strong convergence
are equivalent) of
positive elements, that is, a positive linear map is normal iff
$\phi(x_\alpha) \uparrow \phi(x)$ whenever $x_\alpha \uparrow x$
(\cite{ConwayOperator} Corollary 46.5).  A C$^*$-isomorphism
between two W$^*$-algebras is automatically normal, but a *-homomorphism or
completely positive map need not be.

We refer to a W$^*$-algebra $\AAA$ as \textbf{separable} if its
predual $\AAA_*$ is a separable Banach space; this can be shown to be equivalent to numerous other conditions, including the separability of
either $\AAA$ or its unit ball in either the ultraweak or ultrastrong topologies, and the existence of a faithful normal representation of $\AAA$ on a separable Hilbert space.  A related but strictly weaker property is that of
\textbf{countable decomposability}, which can be defined as the property
that every mutually orthogonal family of nonzero projections in $\AAA$
is at most countable; this is equivalent to the existence
of a faithful state, the existence of a faithful
normal state, or the strong metrizability of
the unit ball (\cite{Blackadar} III.2.2.27).

Throughout, we use the boldface symbol $\one$ to denote the unit
of an algebra, while $1$ will denote the first natural number.  The phrase ``unital subalgebra'' will always be taken to mean that the subalgebra contains the unit of the larger algebra.

\mainmatter
%    Include main chapters here.
\chapter{Introduction to Completely Positive Semigroups} \label{chapintroduction}

\section{Completely Positive Maps and Semigroups}
In this section we introduce the basic objects of study.

\begin{definition} \label{defncpmap}
Let $A, B$ be C$^*$-algebras and $\phi: A \to B$ a linear map.  We
say that $\phi$ is
\begin{enumerate}
    \item \textbf{positive} if it maps positive elements of $A$
    to positive elements of $B$,
    \item \textbf{$n$-positive} if the map $I_n \otimes \phi:
    M_n(\com) \otimes A \to M_n(\com) \otimes B$ is positive, and
    \item \textbf{completely positive} if $\phi$ is $n$-positive
    for all $n \geq 1$.
\end{enumerate}
\end{definition}

%\begin{remark} \label{remcpmaps}
We record here without proof some of the important properties of positive and completely positive maps.  Throughout, $A, B$ denote
C$^*$-algebras and $\phi: A \to B$ a linear map.
\begin{itemize}
    \item Every positive linear map is a *-map, that is, has
    the property that $\phi(a)^* = \phi(a^*)$ for all $a \in A$.  (\cite{Paulsen} Exercise 2.1)
    \item Any two of
    the following properties implies the third:
    \begin{enumerate}
        \item $\phi(\one) = \one$
        \item $\|\phi\| = 1$
        \item $\phi$ is positive.
    \end{enumerate}
    (\cite{Paulsen} Cor 2.9 and Prop 2.11)

    \item If $\phi$ is 2-positive, then $\phi(a)^* \phi(a) \leq \|\phi(\one)\|\phi(a^* a)$
    for all $a \in A$.  This is known as the \textbf{Schwarz
    inequality for 2-positive maps}.
    In particular, if $\phi$ is unital and
    completely positive then $\phi(a)^* \phi(a)
    \leq \phi(a^* a)$ for all $a \in A$.  (\cite{Paulsen} Proposition 3.3 and Exercise 3.4)
    \item If either $A$ or $B$ is commutative, the map $\phi: A \to B$
    is positive iff it is completely positive.  (\cite{Paulsen} Theorems 3.9 and 3.11)
    \item If $A$ and $B$ are W$^*$-algebras, a completely positive
    map $\phi: A \to B$ is normal iff it is strongly continuous.  (\cite{Blackadar} Proposition III.2.2.2).  Strong continuity is equivalent to ultrastrong because of the boundedness of the map.
    \item If $\phi: A \to B(H)$ is a completely positive map, there
    exists a triple $(K, V, \pi)$, unique up to isomorphism, such that
    \begin{enumerate}
        \item $K$ is a Hilbert space
        \item $V: H \to K$ is a linear map such that $\|\phi\| = \|V\|^2$
        \item $\pi: A \to B(K)$ is a *-homomorphism such that
        $V^* \pi(a) V = \phi(a)$ for all $a \in A$
    \end{enumerate}
    and with the additional minimality property that $\overline{\pi(A) VH} = K$.  The triple $(H, V, \pi)$ is called the \textbf{minimal Stinespring
    dilation} of $\phi$.  If $\phi$ is unital, $V$ is an isometry; if $\phi$ is normal, so is $\pi$.  This is known as
    \textbf{Stinespring's dilation theorem} (\cite{Stinespring},
    \cite{Paulsen} Theorem 4.1,
    \cite{Blackadar} Theorems II.6.9.7 and III.2.2.4).
\end{itemize}
%\end{remark}

\begin{definition} \label{defcpsemigroup}
 Let $\AAA$ be a C$^*$-algebra (resp.\ W$^*$-algebra).
 \begin{enumerate}
    \item A \textbf{cp-semigroup} on $\AAA$ is a family
    $\{\phi_t: t \in [0,\infty)\}$ of (normal) completely positive contractive linear
    maps  $\phi_t: \AAA \to \AAA$ such that $\phi_0 = \dss{\text{id}}{\AAA}$ and
    \[
    \phi_t \circ \phi_s = \phi_{t+s}
    \]
    for all $s,t \geq 0$.

    \item An \textbf{e-semigroup} on $\AAA$ is a cp-semigroup in
    which each $\phi_t$ is a *-endomorphism.

    \item Capital letters (\textbf{CP-semigroup}, \textbf{E-semigroup}) indicate
    that for each $a \in \AAA$, $t \mapsto \phi_t(a)$ is a
    continuous function from $[0,\infty)$ to $\AAA$, where $\AAA$
    is given the norm (resp.\ ultraweak) topology.  We refer to this
    continuity property of the semigroup as \textbf{strong continuity}
    or \textbf{point-norm continuity} in the C$^*$ case, and
    \textbf{point-weak continuity} in the W$^*$ case.

    \item A subscript of 0 (\textbf{cp$_0$-semigroup}, \textbf{CP$_0$-}, \textbf{e$_0$-}, \textbf{E$_0$-})
    indicates that $\AAA$ contains a unit $\one$
    and that $\phi_t(\one) = \one$ for all $t \geq 0$.
 \end{enumerate}
\end{definition}

\begin{remark} \label{remQMP}
The term \textbf{quantum Markov process} or \textbf{quantum Markov
semigroup} is sometimes used in the literature to describe cp$_0$- or CP$_0$-semigroups.
\end{remark}

\begin{remark} \label{remCPdefinition}
In the case where $\AAA$ is a W$^*$-algebra,
the definitions of cp-semigroup and CP-semigroup remain unchanged
when stated in terms of the strong topology rather than the weak
topology.  That is, each map $\phi_t$ is normal iff it is strongly
continuous, as noted above; and, as we shall show in more detail
below, the map $t \mapsto \phi_t(a)$ for fixed $a$ is continuous
with respect to the weak topology iff it is continuous
with respect to the strong topology (that is, point-weak continuity is equivalent to point-strong continuity).
\end{remark}

\begin{definition} \label{definvariantstate}
Let $\phi = \{\phi_t\}$ be a cp-semigroup on $\AAA$.  An
\textbf{invariant state} for $\phi$ is a state $\omega: \AAA \to \com$
with the property
\[
\forall t \geq 0: \qquad \omega \circ \phi_t = \omega.
\]
\end{definition}

\section{Dilation}
In this section we introduce the ways in which cp-semigroups and e-semigroups may be related to each other.

\begin{definition} \label{defembeddingretraction}
Let $A, B$ be C$^*$-algebras.
\begin{enumerate}
    \item A \textbf{conditional expectation} on $A$ is a linear
    map $E: A \to A$ such that $E^2 = E$, $\|E\| = 1$, and
    the range $E(A)$ is a C$^*$-subalgebra.
    \item An \textbf{embedding from $A$ to $B$} is an injective
    (equivalently, isometric) *-homomorphism from $A$ to $B$.
    \item Given an embedding $i: A \to B$, a \textbf{retraction
    with respect to $i$} is a completely positive map $e: B \to A$
    such that $e \circ i = \text{id}_{A}$.
\end{enumerate}
\end{definition}

\begin{remark}
A linear map $E: A \to A$ whose range is a C$^*$-subalgebra
is a conditional expectation iff it
is a completely positive contraction and is a bimodule map over its range, i.e.\ has the property that $E(E(a) x) =  E(a) E(x) = E(aE(x))$ for
all $a,x \in A$; this is known as \textbf{Tomiyama's theorem}
(\cite{Tomiyama}).  As a result, if $i: A \to B$ is an embedding
and $e: B \to A$ a corresponding retraction, then $i \circ e$
is a conditional expectation on $B$ with range $i(A)$.  Hence,
the distinction between a retraction and a conditional expectation is precisely the distinction between \emph{identifying} $A$ as a subalgebra of $B$, and \emph{explicitly writing an inclusion map} from $A$ to $B$.  The difference is a matter of taste; we generally follow the latter approach.
\end{remark}

\begin{definition}
Let $\phi = \{\phi_t\}$ be a cp-semigroup on a C$^*$-algebra $\AAA$.
 An \textbf{e-dilation} of $(A, \phi)$ is a tuple $(\Aa, i, \E, \sigma)$ where $\Aa$ is a C$^*$-algebra, $i: A \to \Aa$ an embedding,
$\E: \Aa \to A$ a retraction with respect to $i$,
and $\sigma = \{\sigma_t\}$ an e-semigroup on $\Aa$, satisfying
\[
\forall t \geq 0: \qquad \phi_t = \E \circ \sigma_t \circ i.
\]
We summarize the relationship in the following diagram:
\[ \xymatrixcolsep{5pc}\xymatrix{
\Aa \ar[r]^{\sigma_t} & \Aa \ar[d]^\E \\
A \ar[u]^i \ar[r]_{\phi_t} & A
} \]
We call $(\Aa, i, \E, \sigma)$
a \textbf{strong e-dilation} if it satisfies $\E \circ \sigma_t
= \phi_t \circ \E$, corresponding to the diagram
\[ \xymatrixcolsep{5pc}\xymatrix{
\Aa  \ar[d]_\E \ar[r]^{\sigma_t} & \Aa \ar[d]^\E \\
A \ar[r]_{\phi_t} & A
} \]
Note that this implies
\[
\phi_t = \phi_t \circ \E \circ i = \E \circ \sigma_t \circ i
\]
so that every strong dilation is a dilation, but the converse does
not always hold.
% AGENDUM: Track down Skeide source for this terminology (cite slides if necessary), and verify that in fact it is stronger.
An e$_0$-dilation of a cp$_0$-semigroup is said to be \textbf{unital}
if $i(\one) = \one$.
\end{definition}

\section{Motivation and Examples}
\begin{example}[Conjugation by Contractions]
Let $H$ be a Hilbert space and $\{T_t\}$ a semigroup of contractions
on $H$.  Then the maps $\phi_t: B(H) \to B(H)$ defined by
\[
\phi_t(X) = T_t^* X T_t
\]
form a cp-semigroup.  It is a cp$_0$-semigroup iff all the $T_t$
are isometries, an e-semigroup iff all the $T_t$ are
coisometries, and hence an e$_0$-semigroup iff all the $T_t$
are unitaries.  If $\{T_t\}$ is strongly continuous, in that
$t \mapsto T_t$ is continuous with respect to the strong operator
topology on $B(H)$, then $\{\phi_t\}$ is a CP-semigroup.
% If {T_t} is uniformly continuous, which I don't know if ever happens nontrivially, an easy two-term estimation shows that {\phi_t} is uniformly continuous.  If {T_t} is strongly continuous, {phi_t} is point-ultraweakly continuous; see the ``Positive and CP Maps'' document.
A theorem of Cooper (\cite{Cooper}) states that, given
a strongly continuous contraction semigroup $\{T_t\}$ on
$H$, there exist a Hilbert space $K$, an isometry
$V: H \to K$, and a strongly continuous group $\{U_t\}$ of
unitaries on $K$ such that
\[
T_t = V^* U_t V.
\]
If the $T_t$ are isometries, one obtains the stronger condition
\[
V T_t = U_t V.
\]
Given the Cooper dilation of the semigroup $\{T_t\}$,
one can then define
\begin{enumerate}
    \item the E$_0$-semigroup $\{\alpha_t\}$ on $B(K)$ by $\alpha_t(Y) = U_t^* Y U_t$
    \item the non-unital embedding $i: B(H) \to B(K)$ by $i(X) = V X V^*$
    \item the retraction $\E: B(K) \to B(H)$ by $\E(Y) = V^* Y V$
\end{enumerate}
Then $(B(K), i, \E, \{\alpha_t\})$ is an E$_0$-dilation of
$(B(H), \{\phi_t\})$.

This example
plays a role in the general theory; for instance, Evans and Lewis
prove their dilation theorem (\cite{EvansLewis}) by relating certain more general
semigroups to those of the form $X \mapsto T_t^* X T_t$,
and then applying Cooper dilation.
\end{example}

%\begin{example}
%Second quantization example.  Remark: both involve mapping dilations to %dilations via some functor.  Check Christensen/Evans on ``Gaussian functor''
%\end{example}

\begin{example}[Open Quantum Systems]
In (one of the axiomatizations of) quantum mechanics, every physical system corresponds to a von Neumann algebra $\AAA$, with states of the system
corresponding to positive elements of $\AAA$ of trace 1.  A physical transformation of the system must map states to states and hence, in
particular, must be a positive map; a continuous-time evolution of the system corresponds therefore to a semigroup of positive maps.  If the system is entangled with an environment, a physical transformation of the composite system  must map composite states to composite states, which implies complete positivity of the restriction to the original system; hence, a continuous-time evolution of such an \textbf{open quantum system} is represented by a semigroup of completely positive maps.  Continuity requirements are also natural to impose in this setting as one of the physical axioms.

Actually, the representation of such a system as a completely positive semigroup is an approximation to a more general \textbf{master equation}, which approximation holds under various simplifying physical assumptions such as those of ``weak coupling'' or a ``singular reservoir.''  Completely positive semigroups arise, for instance, in quantum thermodynamics, where the environment is sometimes regarded as an infinite ``heat bath'' whose self-interactions are much faster than those of the system under study.  For more on these matters see \cite{Haake}, \cite{DaviesMME}, \cite{GKS}, \cite{Lindblad}, \cite{Davies}, \cite{EvansLewis}, and \cite{Attal2}.
% AGENDUM: Skeide also likes Arveson's Dynamical Invariants
In the thermodynamic context one typically assumes the existence of a
normal $\phi$-invariant state $\omega$ on $\AAA$, representing an
equilibrium of the system; correspondingly, one is interested in
dilations $(\Aa, i, \E, \sigma)$ for which there exists a
normal $\sigma$-invariant state $\varpi$ on $\Aa$, which dilates
$\omega$ in the sense that $\varpi \circ i = \omega$.  In the
case of a strong dilation this is automatic, as one can simply
define $\varpi = \omega \circ \E$, and it follows that
\[
\varpi \circ \sigma_t = \omega \circ \E \circ \sigma_t
= \omega \circ \phi_t \circ \E = \omega \circ \E =\varpi.
\]

In this setting, dilation is a way of relating the dynamics
of an open (or ``dissipative'') system to the dynamics of
a closed (or ``non-dissipative'') system containing it.
\end{example}

\begin{example}[Daniell-Kolmogorov Construction] \label{exmarkovdilation}
Let $\AAA$ be a commutative unital C$^*$-algebra, and let $S$
be the maximal ideal space of $\AAA$, so that $\AAA \simeq C(S)$.
Let $\{P_t\}$ be a CP$_0$-semigroup on $\AAA$.  By Riesz representation
we obtain for each $t \geq 0$ and each $x \in S$ a measure $p_{t,x}$
characterized by the property
\[
\forall f \in C(S): \qquad \int_S f(y) \, dp_{t,x}(y) = (P_t f)(x).
\]
Moreover, since $P_t f$ is a continuous function, the family $\{p_{t,x}\}$
varies weak-* continuously in $x$.  The property $P_0 = \text{id}$
implies that $p_{0,x}$ is the point mass at $x$, and the
semigroup property $P_{s+t} = P_s P_t$ implies the \textbf{Chapman-Kolmogorov equation}
\[
p_{t+s,x}(E) = \int_S p_{s,y}(E) dp_{t,x}(y).
\]

Let $\SSSS$ denote the \textbf{path space}
$S^{[0,\infty)}$, and $\Aa =
C(\SSSS)$.  We have the embedding $i: \AAA \to \Aa$ given
by $i(f)(\pP) = f(\pP(0))$.
By the Stone-Weierstrass theorem, the *-subalgebra
$\Aa_0 \subset \Aa$ consisting of finite sums of functions of
the form $f_1^{(t_1)} \cdots f_n^{(t_n)}$, where for a path
$\pP \in \SSSS$ the value of $f_i^{(t_i)}$ depends only
on $\pP(t_i)$, is dense in $\Aa$.  We define a unital
linear map $\E_0: \Aa_0 \to \AAA$ on $\AAA_0$ by
\[
\E_0 [f_1^{(t_1)} \cdots f_n^{(t_n)}]
= f_n P_{t_n-t_{n-1}} \Big( f_{n-1} P_{t_{n-1}- t_{n-2}} \Big( \cdots
P_{t_2-t_1} \Big(f_1 \Big) \Big) \cdots \Big).
\]
Clearly $\E_0 \circ i = \text{id}_\AAA$.
We will show shortly that $\E_0$ is well-defined and contractive,
so that it extends to a unital contractive
(hence positive, hence completely positive)
linear map $\E: \Aa \to \AAA$ which satisfies $\E \circ i = \text{id}_\AAA$
and is therefore a retraction with respect to $i$.

We define for each $t \geq 0$ the continuous maps $\lambda_t:
\SSSS \to \SSSS$ by $(\lambda_t \pP)(s) = \pP(s+t)$, and
the corresponding *-endomorphisms $\sigma_t: \Aa \to \Aa$ by
$\sigma_t f = f \circ \lambda_t$.  It is immediate from
the above that $\E \circ \sigma_t \circ i = P_t$, so that
we have obtained a unital e$_0$-dilation of our CP$_0$-semigroup.

Given any regular Borel probability measure $\mu_0$ on $S$,
we obtain through Riesz representation a regular Borel probability
measure $\mu$ on $\SSSS$ characterized by the property
\[
\forall f \in \Aa: \qquad \int_\SSSS f \, d\mu
= \int_S (\E f) \, d\mu_0.
\]
This then implies that
\[
\forall f \in \AAA: \qquad (P_t f)(x)
= \EE \Big[f(\pP(t)) \Big| \pP(0) = x\Big]
\]
where $\EE$ denotes conditional expectation in the probabilistic
sense, so that we have constructed a Markov process $\{\pP(t)\}$
with specified transition probabilities.
We thus obtain a C$^*$-algebraic version of the classical
\textbf{Daniell-Kolmogorov construction}, at least
in the context of \textbf{Feller processes}
rather than general Markov processes.
% AGENDUM: Come up with better citations for the C*-algebraic version.

We now consider an alternate perspective on the same construction, which
enables us easily to prove that $\E_0$ is well-defined and contractive, and simultaneously offers a preview of the techniques used in
this paper.
For each nonempty finite subset $\gamma \subset [0,\infty)$ let
$\AAA_\gamma$ denote a tensor product of $|\gamma|$ copies of
$C(S)$.  For $\beta \subset \gamma$ we obtain unital embeddings
$\AAA_\beta \to \AAA_\gamma$ as follows: Writing $\gamma$
as a disjoint union $\beta \cup \gamma'$, identify
$\AAA_\gamma$ with $\AAA_\beta \otimes \AAA_{\gamma'}$
and embed via $f \mapsto f \otimes \one$.  This yields
an inductive system and, using the general fact that
$C(X \times Y) \simeq C(X) \otimes C(Y)$ for compact
Hausdorff spaces $X$ and $Y$, we see that
$\displaystyle\lim_{\rightarrow} \AAA_\gamma$ is isomorphic
to $\Aa$.  The domain of $\E_0$ is the union of the images of
all the $\AAA_\gamma$ inside $\Aa$, and the well-definedness and
contractivity of $\E_0$ reduce, by induction, to the well-definedness
and contractivity of the maps $\theta_t: C(S) \otimes C(S)$ given on
simple tensors by $\theta_t (f \otimes g) = (P_t f) g$.  But
such a map $\theta_t$ may
be equivalently defined as
\[
(\theta_t F)(x) = \int_S F(y,x) dp_{t,x}(y)
\]
which obviously yields a well-defined contraction
on $C(S) \otimes C(S)$.

We note that the e$_0$-semigroup $\{\sigma_t\}$ is not continuous,
even when the original semigroup $\{P_t\}$ is; that is,
we obtain only an e$_0$-dilation, not an E$_0$-dilation, of
a CP$_0$-semigroup.  We shall return to this point in chapter
\ref{chapcontinuous}.
\end{example}
%AGENDUM: Do you always get an invariant state in the commutative setting? Kummerer seems to say so...

\begin{remark} \label{remwhy}
The last two examples represent the two major streams of thought which motivate the dilation theory of completely positive semigroups.  On the one hand, in the physics setting such a dilation corresponds to an embedding of an open quantum system inside some closed quantum system.  On the other hand, we have seen that
dilating a CP$_0$-semigroup defined on a \emph{commutative} C$^*$-algebra amounts to construction of a Markov process; hence, we may think of
dilations of general CP$_0$-semigroups as a way
of constructing ``noncommutative Markov processes.''
\end{remark}

\section{Continuity Properties of Semigroups}
In this section we examine in greater detail the continuity properties
of completely positive semigroups, beginning with more general
considerations regarding contraction semigroups on Banach spaces.

\subsection{C$_0$-Semigroups} \label{secC0semigroups}
We recount here some of the essentials of the theory of semigroups on Banach spaces, which can be found in \cite{HillePhillips},
\cite{DunfordSchwarz1}, \cite{BratteliRobinson1}, and \cite{EngelNagel}.

A semigroup $\{T(t)\}_{t \geq 0}$ of bounded linear operators on a Banach space
$\Xx$ is said to be
\begin{enumerate}
    \item \textbf{uniformly continuous} if $t \mapsto T(t)$ is
    continuous with respect to the norm topology on $B(\Xx)$; that
    is, if $\displaystyle\lim_{t \to t_0} \|T(t)-T(t_0)\|_{B(\Xx)} = 0$;
    \item \textbf{strongly continuous} if, for each
    $x \in \Xx$, $t \mapsto T(t) x$ is continuous with respect
    to the norm topology on $\Xx$; that is, if
    $\displaystyle\lim_{t \to t_0} \|T(t)x - T(t_0)x\|_\Xx = 0$ for each $x \in \Xx$;
    \item \textbf{weakly continuous} if, for each $x \in \Xx$,
    $t \mapsto T(t) x$ is continuous with respect to the weak topology
    on $\Xx$; that is, if $\displaystyle\lim_{t \to t_0} \ell \big( T(t) x - T(t_0) x \big) = 0$ for each $x \in \Xx$ and each
    $\ell \in \Xx^*$.
\end{enumerate}
In case $\Xx$ is the dual of some Banach space $\Xx_*$, we
define the semigroup to be
\begin{enumerate}[resume]
    \item \textbf{weak-* continuous} if, for each $x \in \Xx$,
    $t \mapsto T(t) x$ is continuous with respect to the weak-* topology
    on $\Xx$; that is, if $\displaystyle\lim_{t \to t_0} \ell \big( T(t) x - T(t_0) x \big) = 0$ for each $x \in \Xx$ and each
    $\ell \in \Xx_*$.
\end{enumerate}

These modes of continuity can, of course, be defined for other
families $\{T(t)\}$ of operators which do not necessarily form a semigroup; however, when they do, strong and weak continuity
are equivalent (\cite{EngelNagel} Theorem 1.1.6).  Furthermore, uniform continuity is too stringent a hypothesis to be attainable in most applications of interest, so that the bulk of semigroup theory revolves around strongly continuous semigroups, also known as \textbf{C$_0$-semigroups}.
Many important C$_0$-semigroups are \textbf{contractive} (meaning
$\|T(t)\| \leq 1$ for all $t$), including completely positive
semigroups; as the theory is somewhat simpler for contractive
C$_0$-semigroups, we shall focus on this case.

The most important object associated with a contractive C$_0$-semigroup
is its \textbf{generator}, the operator $\LL$ on $\Xx$ defined by
the formula
\[
\LL x = \lim_{t \to 0} t\inv [T(t) x - x].
\]
This is in general a closed densely defined unbounded operator,
and in fact
is bounded iff the semigroup is uniformly continuous.  Furthermore,
the generator satisfies the resolvent growth condition $\|(\lambda \one - \LL)\inv\| \leq \lambda\inv$ for all $\lambda > 0$.  The
\textbf{Hille-Yosida theorem} provides a converse, stating that
every closed densely defined operator satisfying this resolvent growth
condition is the generator of some contractive C$_0$-semigroup.  Intuitively,
this semigroup is given by $T(t) = e^{t \LL}$, but this exponential
functional cannot be defined through the usual power series when
$\LL$ is unbounded; one can, however, write
\[
T(t) x = \lim_{n \to \infty} \left(\one - \frac{t}{n} \LL \right)^{-n} x
\]
which is sometimes called the \textbf{Post-Widder inversion formula}
for C$_0$-semigroups.  We thus have a bijection between
contractive C$_0$-semigroups and closed densely defined operators satisfying
a resolvent growth condition, with explicit formulas for both
directions of the bijection.

A notable consequence of the semigroup property is the equivalence between certain notions of continuity and measurability.  We
define a family $\{T(t)\}$ of operators on $\Xx$,
equivalently viewed as a function $T: [0,\infty) \to B(\Xx)$, to be
\begin{enumerate}
    \item \textbf{uniformly measurable} if $T$ is the a.e.\ norm
    limit of a sequence of countably-valued functions from
    $[0,\infty)$ to $B(\Xx)$;
    \item \textbf{strongly measurable} if, for each $x \in \Xx$,
    $t \mapsto T(t) x$ is the a.e.\ norm limit of a sequence
    of countably-valued functions from $[0,\infty)$ to $\Xx$;
    \item \textbf{weakly measurable} if, for each $x \in \Xx$
    and $\ell \in \Xx^*$, $t \mapsto \ell(T(t) x)$ is
    a measurable function from $[0,\infty)$ to $\com$.
\end{enumerate}
In case $\Xx$ is the dual of another Banach space $\Xx_*$,
we also define $\{T(t)\}$ to be
\begin{enumerate}[resume]
    \item \textbf{weak-* measurable}  if, for each $x \in \Xx$
    and $\ell \in \Xx_*$, $t \mapsto \ell(T(t) x)$ is
    a measurable function from $[0,\infty)$ to $\com$.
\end{enumerate}
One might ask why we do not instead define the different types of measurability using the Borel $\sigma$-algebras generated by the corresponding continuity types; the short answer is that a better integration theory results from the definitions given here (the Bochner integral in the case of uniform measurability, the Pettis integral for the others).

It turns out that weak and strong measurability are equivalent
when $\Xx$ is separable (\cite{HillePhillips} Corollary 2, p.\ 73)
and, when $\{T(t)\}$ is a contraction semigroup, both are
equivalent to strong and weak continuity at times $t > 0$ (\cite{HillePhillips} Theorem 10.2.3).  This latter result is analogous to the fact that measurable solutions to the Cauchy functional equation
$f(x+y) = f(x)f(y)$ on $\re$ are exponentials, and hence are continuous.  However, strong measurability at $t = 0$ is not enough
to infer strong continuity at $t = 0$, but requires the additional
hypothesis that $\bigcup_{t > 0} T(t) \Xx$ be dense in
$\Xx$ (\cite{HillePhillips} Theorem 10.5.5).

A contraction semigroup $\{T(t)\}$ on $\Xx$ induces an
\textbf{adjoint semigroup} $\{T(t)^*\}$ on $\Xx^*$ by the
formula $(T(t)^* f)(x) = f(T(t) x)$.  If $\Xx$ is the dual
of $\Xx_*$ and if each $T(t)$ is weak-* continuous,
one obtains also a \textbf{pre-adjoint semigroup} $\{T(t)_*\}$
through the same formula; since the weak-* topology is
of much more interest than the weak topology for spaces
having a predual, this is usually referred to in the literature as the
adjoint semigroup (and of course is the restriction of the adjoint
semigroup to $\Xx_* \subset \Xx^*$).  Weak-* continuity and
measurability of $\{T(t)\}$ are equivalent to weak continuity and measurability of $\{T(t)_*\}$, so that in particular they
are equivalent to each other at times $t > 0$
if $\Xx_*$ is separable.

The last topic to consider for contraction semigroups
is the passage from separate to joint continuity.  We
summarize the results in the following theorem.

\begin{theorem}[Joint Continuity of C$_0$-Semigroups] \label{thmjointcontinuityC0} \
\begin{enumerate}
    \item Let $\Xx$ be a Banach space and $\{T(t)\}_{t \geq 0}$
    a contractive C$_0$-semigroup.  Then $T(t)(x)$ is
    jointly continuous in $t$ and $x$; that is, the map
    $[0,\infty) \times \Xx \sa{T} \Xx$ is continuous
    with respect to the norm topology on $\Xx$.
    \item Let $\Xx$ be a Banach space with separable predual
    $\Xx_*$, and
    $\{T(t)\}_{t \geq 0}$ a weak-* continuous semigroup
    of weak-* continuous contractions on $\Xx$.
    Then $T(t)(x)$ is jointly weak-* continuous
    in $t$ and $x$ on bounded subsets of $\Xx$.
    That is, the map
    $[0,\infty) \times \Xx_1 \sa{T} \Xx_1$
    is continuous with
    respect to the weak-* topology on $\Xx_1$.

    \item Let $\AAA$ be a W$^*$-algebra and
    $\{\phi_t\}_{t \geq 0}$ a C$_0$-semigroup
    of strongly continuous contractions on $\AAA$.
    Then $\phi_t(a)$ is jointly strongly continuous
    in $t$ and $a$ at nonzero times.  That is,
    the map $(0,\infty) \times \AAA_1 \sa{\phi} \AAA_1$
    is continuous with
    respect to the strong topology on $\AAA_1$.
\end{enumerate}
\end{theorem}

\begin{proof} \
\begin{enumerate}
    \item By the triangle inequality and the contractivity
    of the semigroup,
    \[
    \|T(s)(y)-T(t)(x)\| \leq \|T(s)(y-x)\|
    + \|T(s)x - T(t)x\| \leq \|y-x\| + \|T(s)(x)-T(t)(x)\|
    \]
    which tends to zero as $(s,y) \to (t,x)$.

    \item By Alaoglu's theorem, $\Xx_1$ is weak-* compact, and since $\Xx_*$ is assumed to be separable, another standard result implies that $\Xx_1$ is weakly metrizable (\cite{ConwayFunctional} V.5.1).  Joint weak-*
        continuity at $(t,a)$ with $t > 0$ is therefore a special case of Theorem 4 in \cite{ChernoffMarsden}.  Joint weak-*
        continuity at $(0,a)$ is more complicated to establish, but is a consequence of Corollary 3.3 of \cite{Lawson}.
        A very different proof, using less topology and
more semigroup theory, appears in Lemma A.2
of \cite{SkeideClassification}.  Notably, this proof
does not require separability of the predual; although
stated for W$^*$-algebras, it uses only their Banach-space
structure.

    \item  For strong continuity, we use the same
proof, plus the fact that
$\AAA_1$ is also strongly metrizable (\cite{Blackadar}
III.2.2.27).  Since
$\AAA_1$ is not strongly compact, however, we cannot
infer joint continuity at $(0,a)$.
\end{enumerate}
\end{proof}

\subsection{Completely Positive Semigroups} \label{secCPsemigroupcontinuity}
So far we have considered semigroups of contractions on Banach spaces.  When the Banach space is a W$^*$-algebra, and the contractions are normal completely positive maps, some stronger continuity results hold.  Here we note three .  First, recall that a CP-semigroup was defined by the property
 of point-weak continuity.  It turns out that such a
semigroup is automatically \textbf{point-strongly continuous}.
This is Theorem 3.1 of
\cite{ShalitMarkiewicz}.  Second, if a CP-semigroup on
a W$^*$-algebra is
point-norm continuous (which, confusingly enough, would be
called ``strongly continuous'' in the setting of semigroups on
general Banach spaces), then it is automatically uniformly
continuous.  This is Theorem 1 of \cite{Elliott}.
%(As noted in that paper, several previous
%authors had proved one-sided point-strong continuity, under
%the assumption that two-sided continuity was a trivial corollary,
%but the passage to two-sided continuity turns out to be less trivial
%than expected.)

Our third continuity result which is specific to completely positive
semigroups is an improved statement of joint continuity.

\begin{theorem}[Joint Continuity for CP-Semigroups]  \label{thmjointcontinuityCP} \

Let $\AAA$ be a separable W$^*$-algebra and $\{\phi_t\}_{t \geq 0}$
a CP-semigroup on $\AAA$.
\begin{enumerate}
    \item $\phi_t(a)$ is jointly weakly continuous in $t$
    and $a$; that is, the map
    $[0,\infty) \times \AAA_1 \sa{\phi} \AAA_1$ is
    continuous with respect to the weak topology on $\AAA_1$.

    \item $\phi_t(a)$ is jointly strongly continuous
    in $t$ and $a$; that is, the map
    $[0,\infty) \times \AAA_1 \sa{\phi} \AAA_1$ is
    continuous with respect to the strong topology on $\AAA_1$.
\end{enumerate}
\end{theorem}

\begin{proof}
\
\begin{enumerate}
  \item This follows from Theorem
    \ref{thmjointcontinuityC0}; we mention it here in order to observe that a considerably simpler proof is available in this special case, which appears as Proposition 2.23 of \cite{SeLegue} and as
    Proposition 4.1(2) of \cite{MS02}.  See also Lemma 3.2
    of \cite{AccardiMohariVolterra}.  (Although the statement of that
    lemma does not restrict to bounded subsets of $\AAA$ nor
    assume separability of $\AAA_*$, both seem to be necessary
    to justify the use of sequences rather than nets in the proof.)
   \item This is an improvement on Theorem \ref{thmjointcontinuityC0} because of the joint continuity at time 0, which
       we shall need later.
       Assume that $\AAA \subset B(H)$, with $H$ separable.
   Let $t_n \to t$ be a convergent sequence in $[0,\infty)$
   and $a_n \to a$ an SOT-convergent sequence in $\AAA_1$.
   (We can use sequences rather than nets because
   $\AAA_1$ is SOT-metrizable.)  By the first part of this
   theorem, $\phi_{t_n}(a_n) \to \phi_t(a)$ in WOT.
   Now for any $h \in H$,
   \begin{align*}
   \left\| \phi_{t_n}(a_n) h - \phi_t(a) h \right\|^2
   &= \left\|\phi_{t_n}(a_n) h \right\|^2
   - 2 \realp \left\la \phi_{t_n}(a_n) h,
   \phi_t(a) h \right\ra + \left\| \phi_t(a) h \right\|^2\\   &= \la \phi_{t_n}(a_n)^* \phi_{t_n}(a_n) h, h \ra
   - 2 \realp \left\la \phi_{t_n}(a_n) h,
   \phi_t(a) h \right \ra + \left\| \phi_t(a) h \right\|^2\\
   &\leq \la \phi_{t_n}(a_n^* a_n) h, h \ra
   - 2 \realp \left\la \phi_{t_n}(a_n) h,
   \phi_t(a) h \right\ra + \left\| \phi_t(a) h \right\|^2
   \end{align*}
   where we have used the Schwarz inequality for 2-positive
   maps. Taking the limsup as $n \to \infty$,
   and using the fact that $a_n^* a_n \to a^* a$ in WOT whenever $a_n \to a$ in SOT, we see
   that $\phi_{t_n}(a_n) \to \phi_t(a)$ in SOT.

   This appears as Lemma 4 in \cite{VincentSmith} and
    as Lemma 6.4 in \cite{ShalitE0Dilation}.
\end{enumerate}
\end{proof}

\section{Survey of Extant Results}
The first results concerning the existence of dilations
for cp-semigroups date from the 1970's and pertain to uniformly continuous semigroups.  Recall that a contraction semigroup is uniformly continuous
iff its generator is bounded; \cite{ChristensenEvans}, preceded
in special cases by \cite{GKS} and \cite{Lindblad}, showed that the generator of a uniformly continuous CP-semigroup on a W$^*$-algebra must have the form $a \mapsto \Psi(a) + k^* a + ak$ for some element
$k \in \AAA$ and completely positive map $\Psi: \AAA \to \AAA$.
This structure theorem was used by \cite{EvansLewis} to construct dilations of uniformly continuous
 semigroups on $B(H)$ or, more generally, on injective von Neumann algebras).
% AGENDUM: It looks like, chronologically, Evans/Lewis only proved this for certain W*-algebras (``hyperfine,'' a class which includes B(H)) because at that point the generator structure theorem had only been proved in such cases by Lindblad.  Based on later citations, I think that once Christensen/Evans proved the generator structure theorem for all W*-algebras, the existence of E-dilations follows for all W*-algebras by the same proof.  But this should be verified.
% EDIT: Skeide doesn't think a general uniformly continuous dilation was constructed.
% Note: Evans/Lewis prove for semigroups of ``Lindblad type''; the discussion immediately before definition 15.3 shows that uniformly continuous semigroups are of Lindblad type, and it's not clear how much of a generalization they really get.

Dilations of point-weakly continuous CP-semigroups
were shown to exist in special cases (for instance,
on semigroups having specific forms, on semigroups satisfying additional hypotheses such as the existence of a faithful normal invariant state, in the case of discrete-time semigroups, or using a weaker sense of the word ``dilation'') by
\cite{EmchMinimalDilations},  \cite{AccardiFrigerioLewis}, \cite{VincentSmith}, \cite{KummererMarkovDilations}, and others.  However, progress on the general problem required a new insight.  This insight was the notion of a \textbf{product system of Hilbert spaces}, developed by Arveson in \cite{ArvesonContinuousAnalogues1}, \cite{ArvesonContinuousAnalogues2}, \cite{ArvesonContinuousAnalogues3},
and \cite{ArvesonContinuousAnalogues4}.  Briefly, there is an equivalence of categories between E$_0$-semigroups on $B(H)$ and product systems of Hilbert spaces, so that the problem of constructing E$_0$-dilations reduces in some sense to the problem of building a product system out of a CP$_0$-semigroup.  Variants of this strategy were used in \cite{BhatIndexTheory} and \cite{SeLegue}
to show that every CP$_0$-semigroup on $B(H)$ has an E$_0$-dilation, a result known as \textbf{Bhat's theorem}, and the corresponding result
for separable W$^*$-algebras was established in \cite{ArvesonDynamics}.  %section 8.6
 Later, the more general notion of a \textbf{product system of Hilbert modules} was introduced, leading to new proofs of these theorems in \cite{BhatSkeide} and \cite{MS02}.  More recently, product systems have been used to study families of completely positive maps indexed by semigroups other than $[0,\infty)$, with the existence of dilations depending on an additional hypothesis known as strong commutativity (\cite{ShalitE0Dilation}).
% AGENDUM: Check which of the above results should have a subscript 0

A different approach to dilation theory, standing outside this narrative, was proposed by Jean-Luc Sauvageot in \cite{Sauvageot}, \cite{SauvageotFirstExitTimes}, and \cite{SauvageotDirichletProblem}.  Writing during the nascence of free probability  (shortly after
the publication of \cite{VoiculescuSymmetries}, for instance), Sauvageot developed a modified version of the free product appropriate for use in dilation theory.  Since the Daniell-Kolmogorov construction (Example \ref{exmarkovdilation}) can be built using tensor products, which are the coproduct in the category of commutative unital C$^*$-algebras, and since free products play the corresponding role in the category of unital C$^*$-algebras, this is an attractively functorial way to conceptualize a noncommutative Markov process.  Using his version of the free product,
Sauvageot proved that every cp$_0$-semigroup on a C$^*$-algebra has a unital e$_0$-dilation.  This dilation theorem was then used to solve a Dirichlet problem for C$^*$-algebras, much as classical Brownian motion can be used to solve the classical Dirichlet problem (\cite{Kakutani}).

Sauvageot's theorem is unusual in that it achieves a unital dilation; at some point, all the other dilation strategies mentioned so far rely upon  the non-unital embedding of $B(H)$ into $B(K)$ for Hilbert spaces $H \subset K$.  And although other unital dilation techniques exist (for instance, the quantum stochastic calculus pioneered in \cite{HudsonParthasarathy} and expounded more recently in \cite{GoswamiSinha}), they tend to require restrictions on the algebra or the semigroup or both, in contrast to the generality of Sauvageot's construction.  However, although \cite{Sauvageot} asserts that his dilation technique can be modified to yield continuous dilations on W$^*$-algebras,
 %little detail is given, and later authors indicate some uncertainty about this modification (e.g. \cite{SkeideDilationsProductSystems}).  Skeide later indicated that this may not have been how he meant the reference.
no detail is given as to how this modification would proceed.
Hence, given a CP$_0$-semigroup, it seems that one may be forced to choose either a \emph{unital} e$_0$-dilation or a \emph{continuous} (that is, E$_0$-) dilation.  The present paper will expound Sauvageot's dilation techniques in order to demonstrate the possibility of achieving both objectives together (Theorem \ref{thmkahuna}).

\chapter{Liberation} \label{chapliberation}

\section{Introduction}
Free probability theory was introduced by Voiculescu in \cite{VoiculescuSymmetries} as a tool to address the free group factor problem.  Free probability has since blossomed into its own area of study; its development has been an important success, even though the free group factor problem remains unresolved.  Sauvageot's \emph{ad hoc} modification of free probability, in contrast, does not appear to have inspired further pursuit beyond his first paper.  This could be due in part to the relevant free independence-like property remaining implicit in that paper, appearing only in the midst of the proof of Proposition 1.7.

In the present chapter, Sauvageot's version of free independence, hereinafter referred to as \textbf{liberation}
(meant to suggest something similar to freeness; not to be confused with Voiculescu's use of
the same word in \cite{VoiculescuFisher6}) is studied in its own right.  As yet the only nontrivial liberated system known to the
 author is the one originally used by Sauvageot in application to dilation theory.  However, it is still advantageous to separate this part of the exposition, both (i) to clarify the \emph{combinatorial} aspects of dilation, in contrast to its algebraic and analytic features, and (ii) to suggest possibilities for further investigation of connections with standard free probability theory.

\section{Background: Free Independence and Joint Moments}
We recall some of the basic notions of free
probability, which can be found in
references such as \cite{VoiculescuSymmetries}, \cite{FRV}, and \cite{NicaSpeicher}.

A \textbf{noncommutative probability space} is a pair $(\AAA, \phi)$ where $\AAA$ is a unital complex algebra
and $\phi: \AAA \to \com$ a unital linear functional.
Subalgebras $\{A_i\}_{i \in I}$ of $\AAA$ are
said to be \textbf{freely independent} with respect to $\phi$ if $\phi(a_{i_1} a_{i_2} \dots a_{i_n}) = 0$
whenever
\begin{itemize}
    \item $i_1, \dots, i_n$ are elements of $I$ such that adjacent indices are not equal, i.e. for $k = 1, \dots, n-1$
    one has $i_k \neq i_{k+1}$; this condition is abbreviated as $i_1 \neq i_2 \neq \dots \neq i_n$
    \item $a_{i_k} \in A_{i_k}$ for each $k = 1, \dots, n$
    \item $\phi(a_{i_k}) = 0$ for each $k = 1, \dots, n$.
\end{itemize}
Given noncommutative probability spaces $\{(A_i, \phi_i)\}$, a construction known as the
\textbf{free product} of unital algebras yields, in a universal (i.e.\ minimal) way, a noncommutative probability
space $(\AAA, \phi)$ and injections $f_i: A_i \to \AAA$ satisfying $\phi \circ f_i = \phi_i$,
such that the images $f_i(A_i)$ are freely independent with respect to $\phi$.  Furthermore, this
construction on unital algebras can be ``promoted'' to a construction on unital *-algebras or C$^*$-algebras; in the latter
case it is related to the free product of Hilbert spaces.

One implication of free independence which is essential
for our present purposes is that it determines
the value of $\phi$ on the subalgebra generated by
$\{A_i\}$.  Given $i_1 \neq i_2 \neq \dots \neq i_n$
and elements $a_{i_k} \in A_{i_k}$, one can
compute the \textbf{joint moment}
$\phi(a_{i_1} \dots a_{i_n})$ as follows:
\begin{itemize}
    \item \textbf{Center} each term $a_{i_k}$;
    that is, rewrite it as $\mathring{a}_{i_k}
    + \phi(a_{i_k}) \one$, where we define
    $\mathring{x} = x - \phi(x) \one$.
    \item \textbf{Expand} the product
    $(\mathring{a}_{i_1} + \phi(a_{i_1}) \one)
    \cdots (\mathring{a}_{i_n} + \phi(a_{i_n}) \one)$,
    thus obtaining a sum of $2^n$ words.
    \item \textbf{Simplify} by pulling out scalars:
    rewrite, for instance, $\mathring{a}_{i_1}
    \big( \phi(a_{i_2}) \one \big) \mathring{a}_{i_3}$
    as $\phi(a_{i_2}) \mathring{a}_{i_1} \mathring{a}_{i_3}$.
    \item After simplification, the only
    remaining word of length $n$ is the
    centered word
    $\mathring{a}_{i_1} \dots \mathring{a}_{i_n}$.
    Applying the procedure iteratively to all the smaller words that have been generated, one can rewrite
    the original word as a sum of many centered words, plus a word of length 0, i.e.\ a scalar.  Since $\phi$
    vanishes on centered words and is unital, its
    value at the original word is therefore whatever
    scalar is left when this iteration terminates.
\end{itemize}
% Specifically, this can be found in remark 1.5.6 of
% Free Random Variables, and in example on page 70 of
% Lectures on the Combinatorics of Free Probability.
Using this outline, one can calculate
$\phi(a_{i_1} \dots a_{i_n})$ whenever
$i_1 \neq i_2 \neq \dots \neq i_n$.  Of course,
no generality is lost by this hypothesis, as
neighboring terms belonging to the
same subalgebra can be combined.

\section{Defining Liberation}
We now develop two variations on
free independence, which will be
of use in dilation theory.

\begin{definition} \label{defliberated}
Let $\CC$ be a complex algebra and $\eE: \CC \to \CC$
a linear map.
Given a triple $(A, B, \rho, \psi)$ consisting of subalgebras $A, B \subseteq \CC$
and linear maps $\rho: A \to B$
and $\psi: B \to A$,
we introduce the notation
$\mathring{a} = a - \rho(a)$ for elements
$a \in A$ and $\tilde{b} = b - \psi(b)$
for elements $b \in B$; note that in general $\mathring{a}$
 and $\tilde{b}$ are elements neither of $A$ nor of $B$.
 We say the pair $(A, B)$ is:
\begin{enumerate}
    \item \textbf{right-liberated}
    (with respect to $\eE,\rho,\psi$) if
    $\eE$ is a $B$-bimodule map,
    i.e. $\eE[b_1 x b_2] = b_1 \eE[x] b_2$
    for all $b_1, b_2 \in B$ and $x \in \CC$,
    and for every $n \geq 1$, every
    $a_1, \dots, a_n \in A$, and every
    $b_1, \dots, b_{n-1} \in B$,
    \[
    \eE \Big[ \mathring{a}_1 \tilde{b}_1
    \mathring{a}_2 \tilde{b}_2 \cdots \tilde{b}_{n-1} \mathring{a}_n  \Big]=0;
    \]

    \item \textbf{left-liberated} if
    $\eE$ is an $A$-bimodule map and
    for every $n \geq 1$, every $a_1, \dots, a_{n-1} \in A$,
    and every $b_1, \dots, b_n \in B$,
    \[
    \eE \Big[ \tilde{b}_1 \mathring{a}_1
    \tilde{b}_2 \mathring{a}_2 \cdots \mathring{a}_{n-1} \tilde{b}_n
    \Big] =0.
    \]
\end{enumerate}
\end{definition}

We note that the criteria in these definitions
resemble free independence, in that
the alternating product of centered terms
is centered.  The key difference, however,
 is that the centering takes place with respect
to several different maps---elements of $B$
are centered with respect to $\psi$, elements
of $A$ with respect to $\rho$, and the
alternating product with respect to $\eE$.

In some cases it will be useful to generalize
this definition.

\begin{definition} \label{defliberatingrep}
Let $A, B$ be complex algebras, and $\rho: A \to B$
and $\psi: B \to A$ linear maps.  A \textbf{right-liberating representation} of the quadruple $(A, B, \rho, \psi)$
is a quadruple $(\AAA, f, g, \eE)$ where
\begin{itemize}
    \item $\AAA$ is a complex algebra
    \item $f: A \to \CC$ and $g: B \to \AAA$ are
     homomorphisms
    \item $\eE: \AAA \to \AAA$ is a linear map
\end{itemize}
such that $\eE$ is a $g(B)$-bimodule map and, for every $n \geq 1$, every
     $a_1, \dots, a_n \in A$,
    and every $b_1, \dots b_{n-1} \in B$,
    \[
    \eE \Bigg[ \Big(f(a_1) - g(\rho(a_1)) \Big)
    \Big(g(b_1)- f(\psi(b_1))\Big) \cdots \Big(g(b_{n-1})-
     f(\psi(b_{n-1})) \Big) \Big( f(a_n) - g(\rho(a_n)) \Big)
     \Bigg] = 0.
    \]
A \textbf{left-liberating representation} is such a quadruple
for which $\eE$ is an $f(A)$-bimodule map and,
for every $n \geq 1$, every
     $a_1, \dots, a_{n-1} \in A$,
    and every $b_1, \dots b_n \in B$,
    \[
    \eE \Bigg[ \Big(g(b_1)-f(\psi(b_1))\Big)
     \Big(f(a_1) - g(\rho(a_1)) \Big)
     \cdots \Big( f(a_{n-1}) - g(\rho(a_{n-1})) \Big)
     \Big(g(b_n) - f(\psi(b_n)) \Big) \Bigg] = 0.
    \]

\end{definition}

\begin{remark} \label{remreduceliberatingrep}
If there exist maps $\tilde{\rho}: f(A) \to g(B)$
and $\tilde{\psi}: g(B) \to f(A)$
such that $\tilde{\rho} \circ f = g \circ \rho$
and $\tilde{\psi} \circ g = f \circ \psi$,
then Definition (\ref{defliberatingrep}) reduces
to the statement that $(f(A), g(B))$
is right- or left-liberated with
respect to $\eE$, $\tilde{\rho}$,
$\tilde{\psi}$.  Such maps
$\tilde{\rho}$ and $\tilde{\psi}$ need not exist in general,
but they do in two important special cases:
\begin{enumerate}
    \item If $f$ is injective, one may define
    $\tilde{\rho} = g \circ \rho \circ f\inv$;
    similarly, if $g$ is injective,
     one may define $\tilde{\psi} = f \circ \psi \circ g\inv$.
    \item If $(\eE \circ g)$ is injective
    and $\eE \circ g \circ \rho = \eE \circ f$,
    one may define $\tilde{\rho} = g\circ
    (\eE \circ g)\inv \circ \eE$.  Then
    \[
    \tilde{\rho} \circ f = g \circ (\eE \circ g)\inv
    \circ \eE \circ f = g \circ (\eE \circ g)\inv \circ \eE \circ g \circ \rho = g \circ \rho.
    \]
    Similarly, if $(\eE \circ f)$ is injective
    and $\eE \circ f \circ \psi = \eE \circ g$,
    one may define $\tilde{\psi} = f \circ (\eE \circ f)\inv
    \circ \eE$.
\end{enumerate}
Both of these cases will be used subsequently.
\end{remark}

\section{Liberation and Joint Moments}
Like free independence, liberation is a property
that implies an algorithm.  The idea is the same---by centering, expanding, and simplifying, one can write any word as a centered word plus shorter words---but since the centering takes place with respect to three different maps, the details of the procedure are
more complicated.

Suppose  $(A, B)$ are right-liberated in $\CC$
with respect to $\eE, \rho, \psi$ as above.  Suppose
also that $\CC$ and $B$ are unital (with the same unit).
We continue
to use the notation $\mathring{a} = a - \rho(a)$
and $\tilde{b}=b-\psi(b)$.  We will show that the liberation
property determines $\eE$ on the subalgebra $\la A, B \ra \subset \CC$
generated by $A$ and $B$.  Since $B$ is unital, $\la A, B \ra$
is linearly spanned by words of the form $b_0 a_1 b_1
\cdots a_\ell b_\ell$ for $\ell \geq 0$.  To determine
the value of $\eE$ on such a word, we write
\[
b_0 a_1 b_1 \cdots a_\ell b_\ell
= b_0 \mathring{a}_1 \tilde{b}_1 \mathring{a}_2
\tilde{b}_2 \cdots \tilde{b}_{\ell-1} \mathring{a}_\ell
b_\ell + x,
\]
where $x$ is a sum of words with ``fewer $A,B$ alternations''
(in a sense to be made precise) than before.  One can calculate $x$ by expanding
\[
x = b_0 a_1 b_1 \cdots b_{\ell-1} a_\ell b_\ell
- b_0 \Big[a_1 - \rho(a_1) \Big]
\Big[b_1 - \psi(b_1) \Big]
\cdots
\Big[b_{\ell-1} - \psi(b_{\ell-1}) \Big]
\Big[a_\ell - \rho(a_\ell) \Big] b_\ell.
\]
The liberation property implies that $\eE[b_0
a_1 b_1 \cdots a_\ell b_\ell] = \eE[x]$, which can
then be computed recursively.  The recursion terminates
on elements of $b$, for which the bimodule property
implies $\eE[b] = b \eE[\one]$.

%To make these computations more explicit, we first
%introduce some notational conventions:
%\begin{itemize}
%    \item Subset notation will
%    indicate sub-tuples of an ordered tuple;
%    thus, $(1,3) \subset (1,2,3,4,5)$
%    and $(1,2,3,4,5) \setminus (1,3) = (2,4,5)$.
%    \item The concatenation of ordered tuples
%    will be denoted by $\vee$, so e.g.\
%    $(1,5) \vee (2,4) = (1,5,2,4)$.
%    \item We use $[n]$ to denote the
%    tuple $(1,2, \dots, n)$.
%\end{itemize}
To make these computations more explicit,
 we use the following combinatorial notation
 and terminology:
\begin{itemize}
    \item For a natural number $m \geq 1$,
    $[m]$ denotes the set $\{1,2, \dots, m\}$, while
    $2[m]$ denotes the set $\{2,4,\dots,2m\}$.

    \item Given $\ell$ and $S$ as above, the
    \textbf{alternation number} of $S$,
    which we denote $\text{Alt}(S)$, is
    \[
    \text{Alt}(S) = \frac{1}{2}\sum_{j=1}^{2\ell}
    \Big| \dss{\chi}{S}(j) - \dss{\chi}{S}(j-1) \Big|.
    \]
    If one colors $S \cup \{0,2\ell\}$
    white and $[2\ell-1] \setminus S$ black,
    then $\text{Alt}(S)$ is the number of times
    that the color changes from white to black and
    back again as one counts from $0$ to $2\ell$.
    Note that the maximum value is $\ell$, achieved
    iff $S = 2[\ell-1]$.

    \item Given $\ell$ and $S$ as above, the
    \textbf{consecutive in-subsets} of $S$,
    which we denote $T_0, \dots, T_{\text{Alt}(S)}$,
    are the maximal subsets of $S \cup \{0,2\ell\}$
    consisting of consecutive elements, while
    the \textbf{consecutive out-subsets},
    denoted $U_1, \dots, U_{\text{Alt}(S)}$, are
    the maximal consecutive subsets of
    $[2\ell-1] \setminus S$.  For
    example, if $\ell=5$ and
    $S = \{1,3,4,8\}$ then
    \begin{align*}
    T_0 &= \{0,1\}, \qquad &T_1 &= \{3,4\},
   \qquad &T_2 &= \{8\}, \qquad
    &T_3 &= \{10\},\\
    && U_1 &= \{2\}, \qquad &U_2 &= \{5,6,7\}, \qquad
    &U_3 &= \{9\}.
    \end{align*}
\end{itemize}
We then introduce the following algebraic definitions:
\begin{itemize}
    \item[(i)] Given complex algebras $A,B$,
    let
    \[
    \WW_\ell = \{(b_0, a_1, b_1, \dots, a_\ell, b_\ell)
    \mid a_1, \dots, a_\ell \in A; \ b_0, \dots, b_\ell \in B\}, \qquad \ell \geq 0
    \]
    denote the alternating tuples of length $2\ell+1$ which start and
    end with an element of $B$, and $\WW_\# = \bigcup_{\ell=0}^\infty \WW_\ell$.

    \item[(ii)] Given $A, B$ as above,
    as well as linear maps $\rho: A \to B$
    and $\psi: B \to A$, a natural
    number $\ell \geq 1$, a subset
    $S \subseteq [2\ell-1]$, and a tuple
    $w \in \WW_\ell$, let
    \[
    x_j = \begin{cases} \rho \big( a_{(j+1)/2}\big) &
    j \text{ odd}\\ b_{j/2} & j \text{ even}
    \end{cases}
    \]
    for $j \in S \cup \{0,2\ell\}$, and
    \[
    y_k = \begin{cases} a_{(k+1)/2} & k \text{ odd}\\
    \psi \big( b_{k/2} \big) & k \text{ even}
    \end{cases} \]
    for $k \in [2\ell-1] \setminus S$.  Then
    the \textbf{collapse of $\mathbf{w}$ determined by $S$}
    is the tuple in $\WW_{\text{Alt}(S)}$ given by
    \[
    \text{Col}(w; S) = \left( \prod_{j \in T_0}
    x_j, \prod_{k \in U_1} y_k, \prod_{j \in T_1}
    x_j, \prod_{k \in U_2} y_k, \dots, \prod_{k
    \in U_{\text{Alt}(S)}} y_k, \prod_{j \in T_{\text{Alt}(S)}} x_j \right).
    \]

    \item[(iii)] Finally, we define the \textbf{moment function}
    $\MM: \WW_\# \to B$ recursively by $\MM(x) = x$ for
    $x \in \WW_0 = B$, and
    \[
    \MM(x) =
    \sum_{\underset{S \neq 2[\ell-1]}{S \subseteq [2\ell-1]}}
    (-1)^{\ell+|S|} \MM(\text{Col}(x; S)), \qquad
     \qquad x \in \WW_\ell, \quad \ell \geq 1.
    \]

\end{itemize}
The recursion for $\MM$ is well defined
    because the sum is over sets with alternation
    number strictly less than $\ell$.

    The point of these definitions is that, in
expanding the expression
$b_0 a_1 b_1 \cdots a_\ell b_\ell
- b_0 \Big[a_1 - \rho(a_1) \Big]
\Big[b_1 - \psi(b_1) \Big] \cdots
\Big[a_\ell-\rho(a_\ell)\Big] b_\ell$, each term
in the expansion corresponds to a subset $S \subset [2\ell-1]$
indicating from which of the bracketed factors one has chosen
an element of $B$ (either $\rho(a_i)$ or $b_i$).  In the
resulting term of the expansion, one then multiplies
together consecutive elements of $A$ and $B$ to obtain
a word with fewer $A,B$ alternations than the original
word $b_0 a_1 b_1 \cdots a_\ell b_\ell$.

\begin{remark}[Complexity of the moment function] \label{remnumberterms}
Note that evaluating $\MM$
    on a word of length $\ell$ returns a sum of $2^{2\ell-1}$ evaluations on words of length up to $\ell-1$.   This implies that the number
    of terms in the evaluation of $\MM$ on words of length $\ell$ is
    bounded above by the sequence $\{s_\ell\}$ determined by $s_0 = 1$
    and $s_{\ell+1} = 2^{2\ell+1} s_\ell$, which has
    the closed form $s_\ell = 2^{\ell^2}$.
    Of course the actual number of terms is considerably less, due both to cancellation and
    to the fact that this estimate treats all words of length less than $\ell$
    as if they had length $\ell-1$.
\end{remark}

\begin{theorem} \label{thmrightlibmoments}
Let $(A, B)$ be right-liberated in $\AAA$
with respect to $\eE, \rho, \psi$, and suppose
$\AAA$ and $B$ are unital.  Define the product
    function $\Pi: \WW_\# \to \AAA$ by
    $\Pi(b_0, a_1, b_1, \ldots, a_\ell, b_\ell)
    = b_0 a_1 b_1 \cdots a_\ell b_\ell$.  Then
for any $x \in \WW_\#$,
\[
\eE\left[ \Pi(x) \right] = \MM(x) \eE[\one].
\]
\end{theorem}

\begin{proof}
This follows from the above discussion of
$\MM$ as a formalized way of expanding
certain expressions relating to the right-liberation property.
\end{proof}

\begin{corollary} \label{corrightlibgenerated}
Let $(A, B)$ be right-liberated in
$\AAA$ with respect to $\eE, \rho, \psi$.   Then
\[
\eE \Big[ \la A, B \ra \Big] = \eE \big[ B \big].
\]
\end{corollary}
We note that this holds even without assuming unitality of $\AAA$ and $B$, although in that case it follows from the proof of Theorem (\ref{thmrightlibmoments}) rather than the theorem itself.

The obvious generalizations of Theorem
(\ref{thmrightlibmoments}) and Corollary (\ref{corrightlibgenerated}) to right-liberating representations are true as well, and are verified
inductively in the same manner.  We record
them here without proof.

\begin{theorem} \label{thmrightlibrepmoments}
Let $(\AAA, f,g, \eE)$
be a right-liberating representation
of $(A, B, \rho, \psi)$.  For
$x = (b_0, a_1, b_1, \dots, a_\ell, b_\ell)
\in \WW_\#$, let $(f \times g)(x)$
denote the element \\
 ${g(b_0) f(a_1) g(b_1)
\dots f(a_\ell) g(b_\ell) \in \AAA}$.  Suppose $\AAA, B, g$
are unital.

Then for any $x \in \WW_\#$,
\[
\eE[ (f \times g)(x)] = g(\MM(x)) \eE[\one].
\]
\end{theorem}

\begin{corollary} \label{corrightlibrepgenerated}
Let $(\AAA, f,g, \eE)$
be a right-liberating representation
of $(A, B, \rho, \psi)$.   Let
$\la A, B \ra$ denote the subalgebra of $\AAA$
generated by $f(A)$ and $g(B)$.  Then
\[
\eE \Big[ \la A, B \ra \Big] = \eE \big[ g(B) \big].
\]
\end{corollary}

Later we shall be interested in the continuity
properties of joint moments.  We record here the following simple observation:

\begin{proposition}  \label{propnormalmoments}
Let $\AAA$ be a W$^*$-algebra, $A$ and $B$ subalgebras,
and $\rho: A \to B$ and $\psi: B \to A$ normal linear maps.
% The liberation is actually unnecessary; it's
% just long-winded to say that $\AAA$ is a $w^*$-algebra,
% A,B are $W^*$-subalgebras, $\rho: A \to B$
% is a normal linear map, $\eE$ is a normal
% pre-expectation on $\AAA$, and $\nu$ is a
% normal state on $\AAA$.
\begin{enumerate}
    \item For any $x \in \WW_\#$, $\MM(x)$
    is normal in each entry of $x$.  That is,
    given $\ell \geq 0$, $x \in \WW_\ell$,
    and $1 \leq j \leq 2\ell+1$, let $x_k$
    be fixed for all $1 \leq k \leq 2\ell+1$ with $k \neq j$;
    then $\MM(x)$, viewed as a function
    of $x_j$, is a normal linear map from $A$ or $B$
    (depending on the parity of $j$) to $\AAA$.

    \item If $\rho$ is strongly continuous
    on the unit ball $A_1$, and $\psi$ strongly
    continuous on $B_1$, then
    $\MM(x)$ is jointly strongly continuous in the
    entries of $x$ on bounded subsets.  That is,
    the corresponding map
    $A_1 \times B_1 \times \dots \times A_1
    \to \AAA$ is strongly continuous.
\end{enumerate}
\end{proposition}
The proof is a straightforward induction
on $\ell$.

Later we will need to consider moments with respect to several
maps.  When need arises, we use $\MM(x; \rho; \psi)$ in place of $\MM(x)$ for specificity.
\begin{proposition} \label{propmomentwrtsemigroup}
Let $A$ be a $W^*$-algebra,
$\psi: A \to A$ a strongly continuous linear map, and $\{\rho_t\}_{t \geq 0}$ a CP-semigroup on $A$.  Then for each fixed
$x \in \WW_\#$, $\MM(x; \rho_t; \psi)$
and $\MM(x; \psi; \rho_t)$ are strongly continuous in $t$.
\end{proposition}

Here we are implicitly using $B=A$ in the definition of $\WW_\#$.  For the proof, recall that
$t \mapsto \phi_t(a)$ is strongly continuous for
fixed $a \in A$, as discussed in section \ref{secCPsemigroupcontinuity}.  This fact plus the joint
strong continuity of multiplication on $A_1$ yields
a straightforward induction.

%Strong right-liberation is useful in inducing conditional expectations, which we will
%need when it comes time to iterate the Sauvageot
%product.

%\begin{theorem}
%Let $(A, B, \rho, R)$ be strongly
%right-liberated in $\AAA$ with respect to
%$\eE, \nu$.  Let $A_0 = R(A)$
%and $\theta_0$ be a Reynolds
% operator on $\la A_0, B \ra$ with range $B$
% such that $\theta_0 \circ R = \rho$
%and $\eE \circ \theta_0 = \eE$.  Then
%there exists at most one Reynolds operator
%$\theta: \la A, B \ra \to B$ which extends
%$\theta_0$, is a left $\la A_0, B \ra$-module
% map, and has the property $\eE \circ \theta = \eE$.
%\end{theorem}

\subsection{Left Liberation and Joint Moments}
The calculation of joint moments given the property of left liberation
is essentially the same as for right liberation.  Informally,
given a moment calculation based on right liberation, one
may obtain a corresponding moment calculation for left
liberation by interchanging the roles of $A$ and $B$,
and of $\rho$ and $\psi$; for example, given that
\begin{align*}
\eE[b_0 a_1 b_1 a_2 b_2] &= b_0 \Bigg( \rho(a_1) b_1
\rho(a_2) + \rho \Big( a_1 \psi(b_1) a_2 \Big)
- \rho(a_1) \rho \Big( \psi(b_1) a_2 \Big)\\
&\qquad - \rho \Big( a_1 \psi(b_1) \Big)
+ \rho(a_1) \rho \Big( \psi(b_1) \Big) \rho(a_2)
\Bigg) \eE[\one]
\end{align*}
when $(A,B)$ are right-liberated and $\AAA, B$ unital, one can infer that
\begin{align*}
\eE[a_0 b_1 a_1 b_2 a_2] &= a_0 \Bigg( \psi(b_1) a_1
\psi(b_2) + \psi \Big( b_1 \rho(a_1) b_2 \Big)
- \psi(b_1) \psi \Big( \rho(a_1) b_2 \Big)\\
&\qquad - \psi \Big( b_1 \rho(a_1) \Big)
+ \psi(b_1) \psi \Big( \rho(a_1) \Big) \psi(b_2)
\Bigg) \eE[\one]
\end{align*}
when $(A,B)$ are left-liberated and $\AAA, A$ unital.

More precisely, we modify our algebraic definitions above
as follows:
\begin{itemize}
    \item[(i)] Given $A, B$ we define
    for each $\ell \geq 0$ the set
    \[
    \widehat{\WW}_\ell = \{ (a_0, b_1, a_1, \dots,
    b_\ell, a_\ell) \mid
    a_0, \dots, a_\ell \in A; \
    b_1, \dots, b_\ell \in B \}
    \]
    and the corresponding set
    $\widehat{\WW}_\# = \bigcup_{\ell =0}^\infty \widehat{\WW}_\ell$.

    \item[(ii)] Given $A, B, \rho, \psi,
    \ell \geq 1$, $S \subseteq [2\ell-1]$, and
    $w \in \widehat{\WW}_\ell$, let
    \[
    \hat{x}_j = \begin{cases} \psi \big( b_{(j+1)/2}\big) &
    j \text{ odd}\\ a_{j/2} & j \text{ even}
    \end{cases}
    \]
    for $j \in S \cup \{0,2\ell\}$, and
    \[
    \hat{y}_k = \begin{cases} b_{(k+1)/2} & k \text{ odd}\\
    \rho \big( a_{k/2} \big) & k \text{ even}
    \end{cases} \]
    for $k \in [2\ell-1] \setminus S$.  Then
    we let
    \[
    \widehat{\text{Col}}(w; S) = \left( \prod_{j \in T_0}
    \hat{x}_j, \prod_{k \in U_1} \hat{y}_k, \prod_{j \in T_1}
    \hat{x}_j, \prod_{k \in U_2} \hat{y}_k, \dots, \prod_{k
    \in U_{\text{Alt}(S)}} \hat{y}_k, \prod_{j \in T_{\text{Alt}(S)}} \hat{x}_j \right).
    \]

    \item[(iii)] Finally, we define
    $\widehat{\MM}: \widehat{\WW}_\# \to A$ recursively by $\widehat{\MM}(x) = x$ for
    $x \in \widehat{\WW}_0 = A$, and
    \[
    \widehat{\MM}(x) =
    \sum_{\underset{S \neq 2[\ell-1]}{S \subseteq [2\ell-1]}}
    (-1)^{1+|S|} \widehat{\MM}(\widehat{\text{Col}}(x; S)), \qquad
     \qquad x \in \widehat{\WW}_\ell, \quad \ell \geq 1.
    \]
\end{itemize}
% Thought about some kind of ``reversal function'',
% denoted by \widehat, from
% \WW_\ell to \widehat{\WW}_\ell, with the theorem
% being that \widehat{\MM}(\widehat{x}) = \widehat{\MM(x)}.
% However, the reversal is carried out formally, at the
% level of strings of symbols, rather than as a map from
% A to B and vice versa, so that doesn't seem workable.

We arrive at the obvious analogues of our
results for right liberation, presented here
without proof.

\begin{theorem} \label{thmleftlibmoments}
Let $(A, B)$ be left-liberated in $\AAA$
with respect to $\eE, \rho, \psi$, and suppose $\AAA$
   and $A$ are unital.  Define the product function
    $\widehat{\Pi}: \widehat{\WW}_\# \to \AAA$
    by $\widehat{\Pi}(a_0, b_1, a_1, \dots, b_\ell, a_\ell)
    = a_0 b_1 a_1 \cdots b_\ell a_\ell$.  Then
for any $x \in \widehat{\WW}_\#$,
\[
\eE\left[ \widehat{\Pi}(x) \right] = \widehat{\MM}(x) \eE[\one].
\]
\end{theorem}

\begin{corollary} \label{corleftlibgenerated}
Let $(A, B)$ be left-liberated in
$\AAA$ with respect to $\eE, \rho, \psi$.  Then
\[
\eE \Big[ \la A, B \ra \Big] = \eE \big[ A \big].
\]
\end{corollary}

\begin{theorem} \label{thmleftlibrepmoments}
Let $(\AAA, f,g, \eE)$
be a left-liberating representation
of $(A, B, \rho, \psi)$.  For
$x = (a_0, b_1, a_1, \dots, b_\ell, a_\ell)
\in \widehat{\WW}_\#$, let $\widehat{(f \times g)}(x)$
denote the element \\
 ${f(a_0) g(b_0) f(a_1)
\dots g(b_\ell) f(a_\ell)\in \AAA}$.  Suppose also
that $\AAA, A, f$ are unital.

Then for any $x \in \widehat{\WW}_\#$,
\[
\eE[ \widehat{(f \times g)}(x)] = f(\widehat{\MM}(x)) \eE[\one].
\]
\end{theorem}

\begin{corollary} \label{corleftlibrepgenerated}
Let $(\AAA, f,g, \eE)$
be a left-liberating representation
of $(A, B, \rho, \psi)$.   Let
$\la A, B \ra$ denote the subalgebra of $\AAA$
generated by $f(A)$ and $g(B)$.  Then
\[
\eE \Big[ \la A, B \ra \Big] = \eE \big[f(A) \big].
\]
\end{corollary}

\begin{proposition}  \label{propnormalleftmoments}
Let $\AAA$ be a W$^*$-algebra, $A$ and $B$ subalgebras,
and $\rho: A \to B$ and $\psi: B \to A$ normal linear
maps.
% The liberation is actually unnecessary; it's
% just long-winded to say that $\AAA$ is a $w^*$-algebra,
% A,B are $W^*$-subalgebras, $\rho: A \to B$
% is a normal linear map, $\eE$ is a normal
% pre-expectation on $\AAA$, and $\nu$ is a
% normal state on $\AAA$.
\begin{enumerate}
    \item For any $x \in \widehat{\WW}_\#$, $\widehat{\MM}(x)$
    is normal in each entry of $x$.
    \item If $\rho$ is strongly continuous
    on the unit ball $A_1$, and $\psi$ strongly
    continuous on $B_1$, then
    $\widehat{\MM}(x)$ is jointly strongly continuous in the
    entries of $x$ on bounded subsets.
\end{enumerate}
\end{proposition}

\begin{proposition} \label{propmomentwrtsemigroup}
Let $A$ be a $W^*$-algebra,
$\psi: A \to A$ a strongly continuous linear map, and $\{\rho_t\}_{t \geq 0}$ a CP-semigroup on $A$.  Then for each fixed
$x \in \widehat{\WW}_\#$, $\widehat{\MM}(x; \rho_t; \psi)$
and $\widehat{\MM}(x; \psi; \rho_t)$ are strongly continuous in $t$.
\end{proposition}

\subsection{Tables and the Scalar Case}

We include in Appendix \ref{appendixmomenttables} values of the
moment function $\MM$ for small $\ell$.

An important special case of right-liberation occurs
when the map $\psi: B \to A$ is scalar-valued (in particular,
the applications to Sauvageot's dilation theory are
built on the case where $\psi$ is a state on $B$).
This allows considerable simplification of the expressions for joint moments, both because $\psi(B) \subset A \cap A'$
and because $\rho$ is a $\psi$-bimodule map.  Due to its
importance in applications, this special case also has
a table in Appendix \ref{appendixmomenttables}. 
% Thesis Chapter 3 by Dave Gaebler

\chapter{The Sauvageot Product} \label{chapsauvageotproduct}

\section{Introduction}
In this chapter we develop a modification of the unital free product of C$^*$-algebras, adapted for use in dilation theory.  As mentioned in example (\ref{exmarkovdilation}), the classical Daniell-Kolmogorov construction can be reduced to the construction of maps $\theta_t: C(S) \otimes C(S) \to C(S)$
given on simple tensors by $\theta_t(f \otimes g) = (P_t f) g$.  We shall return to the details of this reduction in chapter \ref{chapiteratedproducts}; at present we only describe enough of its features to see what we shall need for the appropriate noncommutative analogue.

Among the many embeddings of $C(S)$ into $C(S) \otimes C(S)$ we distinguish
two, the ``left'' embedding $f \mapsto f \otimes \one$ and the ``right''
embedding $f \mapsto \one \otimes f$.  The map $\theta_t$ is a retraction with respect to the right embedding, and its composition with the left embedding is $P_t$.  That is, by constructing $\theta_t$ we factor $P_t$ into an embedding followed by a retraction (with respect to a different embedding), as
depicted in the following diagram:
\[ \xymatrix{
    C(S) \otimes C(S) \ar@/^1pc/[rd]^{\theta_t}\\
    C(S) \ar[u] \ar[r]_{P_t} & C(S) \ar@{.>}[lu]
    }  \quad \theta_t(f \otimes g) = P_t(f) g
    \]
More generally, the inductive process will work
with tensor powers $C(S)^{\otimes \gamma}$ for finite
sets $\gamma \subset [0,\infty)$, building for each one
a retraction $\epsilon_\gamma: C(S)^{\otimes \gamma} \to C(S)$.
Given $\gamma' = \gamma \cup \{t_k\}$, where $\tau = t_k - \underset{t \in \gamma}{\min} \, t > 0$, we will seek to define $\epsilon_{\gamma'}$ such that
\[ \xymatrix{
   C(S)^{\otimes \gamma'} \ar@/^1pc/[rd]^{\epsilon_{\gamma'}}\\
    C(S)^{\otimes \gamma} \ar[u] \ar[r]_{P_\tau} & C(S) \ar@{.>}[lu]
    }  \quad \epsilon_{\gamma'}(f \otimes g) = P_\tau (\epsilon_\gamma(f)) g
    \]
We note in passing that Stinespring dilation produces a very similar
diagram: Given a unital completely map $\phi: A \to B(H)$ with
minimal Stinespring triple $(K, V, \pi)$, we obtain
\[ \xymatrix{
    B(K) \ar@/^1pc/[rd]^{\theta}\\
    A \ar[u] \ar[r]_{\phi} & B(H) \ar@{.>}[lu]
    }  \quad \theta(T) = V^* T V
    \]
Crucially, however, the right embedding in this case is the non-unital
map $X \mapsto V X V^*$, in contrast to the unital embedding in the
commutative example.  Hence, we take the tensor product as our model in what follows.

In addition to constructing tensor products $C(X) \otimes C(Y) \simeq C(X \times Y)$ of commutative unital C$^*$-algebras, one can also form tensor products of maps between them, and the resulting maps satisfy certain functorial properties. We summarize the properties of the tensor product which we shall seek to replicate in this chapter:
\begin{enumerate}
    \item Given unital C$^*$-algebras $A,B$ and a unital completely
    positive map $A \sa{\phi} B$, we construct a unital C$^*$-algebra
    $A \star B$ with unital embeddings of $A$ and $B$, the images
    of which generate $A\star B$.
    \item We also construct a retraction $A \star B \to B$ which
    factors $\phi$ in the sense of the above diagrams.
    \item Given unital completely positive maps $A \sa{\phi} B$
    and $C \sa{\psi} D$, and given unital *-homomorphisms
    $A \sa{f} C$ and $B \sa{g} D$ such that the square
    \[ \xymatrix{
    C \ar[r]^\psi & D \\
    A \ar[r]_\phi \ar[u]^f & B \ar[u]_g
    }\]
    commutes, we construct a (necessarily unique) unital *-homomorphism\\
    $f \star g: A \star B \to C \star D$ such that the squares
    \[ \xymatrix{
    A \star B \ar[r]^{f \star g} &C \star D \\
    A \ar[u] \ar[r]_f & C \ar[u]
    } \qquad \qquad \xymatrix{
    A \star B \ar[r]^{f \star g} &C \star D \\
    B \ar[u] \ar[r]_g & D \ar[u]
    }\]
    commute.
\end{enumerate}

We now begin our development of a construction satisfying these
requirements.

\section{Sauvageot Products of Hilbert Spaces and Bounded Operators}
Just as the free product of unital C$^*$-algebras can be constructed
from a free product of Hilbert spaces (\cite{VoiculescuSymmetries}),
our product construction on C$^*$-algebras will rely on an underlying
construction on Hilbert space.  Some notational preliminaries:
For a Hilbert space $H$, we use $H^+$ to denote $H \oplus \com$,
and if a unit vector has been distinguished, $H^-$ to denote the
complement of its span.  A distinguished unit vector (such as
$1 \in \com$ as an element of the direct sum $H\oplus \com$)
is generally denoted by $\Omega$.  We also follow the convention, most common in physics and in Hilbert C$^*$-modules, that inner
products are linear in the second variable.

\begin{definition} \label{defsauvageotproducthilbertspaces}
Let $\HH$ and $\LL$ be Hilbert spaces.  The \textbf{Sauvageot
product} $\HH \star \LL$ is the space
\[
\HH \star \LL = \HH^+ \oplus \bigoplus_{n=0}^\infty \Big[
\big(\LL^{+ \otimes n} \otimes \LL \big) \oplus
\big(\HH \otimes \LL^{+ \otimes n} \otimes \LL \big) \Big]
\]
with the convention $\LL^{+ \otimes 0} \otimes \LL = \LL$.
\end{definition}

Though defined as a direct sum, the Sauvageot product of Hilbert
spaces may also be viewed as an infinite tensor product, as
expressed in the following proposition.

\begin{proposition} \label{prop3UE}
Let $\HH$ and $\LL$ be Hilbert spaces, and $\KK = \HH^+ \oplus \LL$.  Denote by $\LL^{+ \otimes \N}$ the infinite tensor power of $\LL^+$ with respect to $\Omega$.  Then there are unitary
equivalences between $\HH \star \LL$
and both $\HH^+ \otimes \LL^{+ \otimes \N}$ and
$\KK \otimes \LL^{+ \otimes \N}$, under which
\begin{itemize}
    \item the subspace $\HH \otimes \LL^{+ \otimes \N}$ of $\HH^+ \otimes \LL^{+ \otimes \N}$
    is identified with the subspace\\ $\displaystyle{\bigoplus_{n=0}^\infty \HH \otimes \LL^{+ \otimes n} \otimes \LL}$
    of $\HH \star \LL$
    \item the subspace $\com \otimes \LL^{+ \otimes \N}$ of $\HH^+ \otimes \LL^{+ \otimes \N}$
    is identified with the subspace\\ $\displaystyle\bigoplus_{n=0}^\infty \LL^{+ \otimes n} \otimes \LL$
    of $\HH \star \LL$
    \item the subspace $\HH \otimes \LL^{+ \otimes \N}$ of $\KK \otimes \LL^{+ \otimes \N}$
    is identified with the subspace\\ $\HH \
    \oplus \displaystyle\bigoplus_{n=0}^\infty \HH \otimes \LL^{+ \otimes n} \otimes \LL$
    of $\HH \star \LL$
    \item the subspace $\com \otimes \LL^{+ \otimes \N}$ of $\KK \otimes \LL^{+ \otimes \N}$
    is identified with the subspace\\ $\LL \oplus \displaystyle\bigoplus_{n=1}^\infty \com \otimes \LL^{+ \otimes (n-1)}
    \otimes \LL$ of $\HH \star \LL$
    \item the subspace $\LL \otimes \LL^{+ \otimes \N}$ of $\KK \otimes \LL^{+ \otimes \N}$
    is identified with the subspace\\ $\displaystyle\bigoplus_{n=1}^\infty \LL \otimes \LL^{+ \otimes (n-1)} \otimes \LL$
    of $\HH \star \LL$.
\end{itemize}
\end{proposition}

\begin{proof}
We will use the unitary equivalences $\LL^{+ \otimes \N} \simeq \LL^+ \otimes \LL^{+ \otimes \N}$, which is evident, and $\LL^{+ \otimes \N} \simeq \com \oplus \bigoplus_{n=0}^\infty \LL^{+ \otimes n} \otimes \LL$,
which follows from the definition of infinite tensor powers.  Repeated application of these equivalences plus
the associative, commutative, and distributive laws
\[
H \otimes (K_1 \otimes K_2) \simeq (H \otimes K_1) \otimes K_2, \quad
H \otimes K \simeq K \otimes H, \quad
H \otimes (K_1 \oplus K_2) \simeq (H \otimes K_1) \oplus (H \otimes K_2)
\]
and the identity $H \otimes \com \simeq H$
yield
\begin{align*}
\KK \otimes \LL^{+ \otimes \N} &= (\HH^+ \oplus \LL) \otimes \LL^{+ \otimes \N} \simeq (\HH \oplus \LL^+)\otimes \LL^{+ \otimes \N}
\simeq (\HH \otimes \LL^{+ \otimes \N}) \oplus (\LL^+ \otimes
\LL^{+ \otimes \N})\\
&\simeq (\HH \otimes \LL^{+ \otimes \N}) \oplus
\LL^{+ \otimes \N}
\simeq (\HH \otimes \LL^{+ \otimes \N}) \oplus
(\com \otimes \LL^{+ \otimes \N}) \simeq \HH^+ \otimes \LL^{+ \otimes \N}
\end{align*}
and
\begin{align*}
\HH \star \LL &= \HH^+ \oplus \bigoplus_{n=0}^\infty
\left(\HH \otimes \LL^{+ \otimes n} \otimes \LL
\oplus \LL^{+ \otimes n} \otimes \LL \right)\\
&\simeq (\HH^+ \otimes \com) \oplus \bigoplus_{n=0}^\infty
\left(\HH \otimes \LL^{+ \otimes n} \otimes \LL
\oplus \com \otimes \LL^{+ \otimes n} \otimes \LL \right)\\
&\simeq (\HH^+ \otimes \com) \oplus \bigoplus_{n=0}^\infty
(\HH^+ \otimes \LL^{+ \otimes n} \otimes \LL)
\simeq \HH^+ \otimes \left(\com \oplus \bigoplus_{n=0}^\infty
\LL^{+ \otimes n} \otimes \LL\right) \simeq
\HH^+ \otimes \LL^{+ \otimes \N}.
\end{align*}
The specific identifications arise by following subspaces through these equivalences.
\end{proof}

As a simple corollary, we obtain the folllowing identifications:
\begin{proposition} \label{propsauvageotproductsHS}
Let $\HH, \HH_1, \HH_2, \LL$ be Hilbert spaces.
\begin{enumerate}
    \item $(\HH_1 \oplus \HH_2)^+ \star
    \LL \simeq (\HH_1 \star \LL) \oplus
    (\HH_2 \star \LL)$
    \item $\HH \star \{0\} \simeq \HH^+$
    \item $\{0\} \star \LL \simeq \LL^{+ \otimes \N}$
\end{enumerate}
\end{proposition}

%AGENDUM: What happens when direct sum L's?

Our next goal is to define the product of maps between Hilbert
spaces.

\begin{definition} \label{defSPops}
Let $\HH_1, \HH_2, \LL_1, \LL_2$ be Hilbert spaces,
$\KK_1 = \HH_1^+ \oplus \LL_1$ and $\KK_2 = \HH_2^+
\oplus \LL_2$, $S: \HH_1^+ \to \HH_2^+$ and $T: \KK_1 \to \KK_2$ bounded maps, and $V: \LL_1 \to \LL_2$ a contraction.
Define the bounded maps $S \star V: \HH_1 \star \LL_1
\to \HH_2 \star \LL_2$ and $T \hstar V:
\HH_1\star \LL_1 \to \HH_2 \star \LL_2$ as follows: Let $V^+
= V \oplus \dss{\text{id}}{\com}: \LL_1^+ \to \LL_2^+$
and let $V^{+ \otimes \N}: \LL_1^{+ \otimes \N} \to \LL_2^{+ \otimes \N}$
be the limit
of the contractions $V^{+\otimes n}: \LL_1^{+ \otimes n}
\to \LL_2^{+ \otimes n}$.  Then
$S \star V: \HH_1 \star \LL_1 \to \HH_2 \star \LL_2$ is the
operator $S \otimes V^{+ \otimes \N}:
\HH_1^+ \otimes \LL_1^{+ \otimes \N}
\to \HH_2^+ \otimes \LL_2^{+ \otimes \N}$ composed
with the unitary equivalences of Proposition \ref{prop3UE};
similarly, $T \hstar V$ is the operator $T \otimes V^{+ \otimes \N}:
\KK_1 \otimes \LL_1^{+ \otimes \N} \to
\KK_2 \otimes \LL_2^{+ \otimes \N}$ composed with the appropriate
unitary equivalences.
\end{definition}

By following the sequence of equivalences in the proof
of Proposition \ref{prop3UE}, we can calculate
how product maps act on the various summands
of $\HH_1 \star \LL_1$.

\begin{proposition} \label{propSPopsummands}
Let $\HH_1, \HH_2, \LL_1, \LL_2$ be Hilbert spaces,
and for $i = 1,2$ let $\KK_i = \HH_i^+ \oplus \LL_i$.
Let $\HH_1^+ \sa{S} \HH_2^+$ and $\KK_1 \sa{T} \KK_2$ be
bounded operators and $\LL_1 \sa{V} \LL_2$ a contraction.
For each $n \geq 0$ let $V^{(n)}$ denote
$V^{+\otimes n} \otimes V: \LL_1^{+ \otimes n}
\otimes \LL_1 \to \LL_2^{+ \otimes n} \otimes \LL_2$.

Let $h \in \HH_1^+$, $h_0 \in \HH_1$,
$k \in \KK_1$, $\ell \in \LL_1^+$, and $\xi \in \LL_1^{+ \otimes n} \otimes \LL_1$ for some $n \geq 0$, and suppose that
\begin{align*}
S \Omega_1 &= \alpha \Omega_2 + y, \qquad \qquad \alpha \in \com, y \in \HH_2\\
S h_0 &= \beta \Omega_2 + z, \qquad \qquad \beta \in \com, z \in \HH_2\\
T h_0 &= \eta + w, \qquad \qquad \eta \in \HH_2, w \in \LL_2^+\\
T\ell &= \zeta+u, \qquad \qquad \zeta \in \HH_2, u \in \LL_2^+.
\end{align*}
Then
\begin{align*}
(S \star V) h &= Sh\\
(S \star V)\xi &= \alpha V^{(n)} \xi + (y \otimes V^{(n)} \xi)\\
(S \star V)(h_0 \otimes \xi) &= \beta V^{(n)}\xi
+ (z \otimes V^{(n)}\xi)\\
(T \hstar V) k &= Tk \\
(T \hstar V)(h_0 \otimes \xi) &= (\eta \otimes V^{(n)}\xi) + (w \otimes V^{(n)}\xi)\\
(T \hstar V)(\ell \otimes \xi) &= (\zeta \otimes V^{(n)}\xi) + (u \otimes V^{(n)}\xi).
\end{align*}
\end{proposition}

Next, we develop some of the essential properties of
this construction.

\begin{proposition} \label{propSPopproperties}
Let $\HH_1^+ \sa{S} \HH_2^+ \sa{S'} \HH_3^+$
and $\KK_1 \sa{T} \KK_2 \sa{T'} \KK_3$ be bounded
maps, and $\LL_1 \sa{V} \LL_2 \sa{V'} \LL_3$ contractions.
\begin{enumerate}
    \item $(S' \star V') \circ (S \star V) =
    (S' \circ S) \star (V' \circ V)$ and
    $(T' \hstar V') \circ (T \hstar V)
    = (T' \circ T) \hstar (V' \circ V)$.
    \item If $S$ is the identity map on $\HH_1 = \HH_2$, and
    $V$ the identity map on $\LL_1 = \LL_2$,
    then $S \star V$ is the identity map on
    $\HH_1 \star \LL_1 = \HH_2 \star \LL_2$.  Similarly, if $T$
    and $V$ are the appropriate identity maps, then so is
    $T \hstar V$.
    \item $\|S \star V\| \leq \|S\| \|V\|$ and
    $\|T \hstar V\| \leq \|T\| \|V\|$.
    \item If $S$ and $V$ are isometries (resp.\ unitaries),
    so is $S \star V$; if $T$ and $V$ are isometries (resp.\ unitaries),    so is $T \hstar V$.
    \item $(S \star V)^* = S^* \star V^*$ and
    $(T \hstar V)^* = T^* \hstar V^*$.
    \item If $S$ decomposes as a direct sum $\dss{S}{L}
    \oplus \dss{S}{R}: \HH_1 \oplus \com \to \HH_2 \oplus \com$,
    then $S \star V$ maps the summands
    of $\HH_1 \star \LL_1$ into the corresponding summands
    of $\HH_2 \star \LL_2$.  That is, if $P_1$ is the projection
    from $\HH_1 \star \LL_1$ onto any of
    $\HH_1^+$, $\LL_1^{+ \otimes n} \otimes \LL$,
    or $\HH_1 \otimes \LL_1^{+ \otimes n} \otimes \LL$,
    and $P_2$ the projection from $\HH_2 \star \LL_2$ onto
    its corresponding subspace, then
    \begin{equation} \label{eqnSPopsintertwines}
    (S \star V) P_1 = P_2 (S \star V).
    \end{equation}
    Similarly, if $T$ decomposes as a direct sum
    $\dss{T}{L} \oplus \dss{T}{R}:\HH_1 \oplus \LL_1
    \to \HH_2 \oplus\LL_2$ and $P_1, P_2$ are as before, then
    \begin{equation} \label{eqnSPopsintertwines2}
    (T \hstar V) P_1 = P_2 (T \hstar V).
    \end{equation}

    \item If $S$ decomposes as the direct sum $\dss{S}{L}
    \oplus \dss{\text{id}}{\com}: \HH_1 \oplus \com \to \HH_2 \oplus \com$
    and $T = S\oplus V$, then
    $S \star V = T \hstar V$.
\end{enumerate}
\end{proposition}

\begin{proof}
The first five assertions are simple
consequences of the corresponding properties of
tensor products of operators.  The sixth follows
as a special case of Proposition \ref{propSPopsummands}
with $y = \beta = 0$ or $w = \zeta = 0$, and the
seventh is also an easy corollary of Proposition \ref{propSPopsummands}.
\end{proof}

\begin{remark} \label{rembifunctor}
The first two properties say that $- \star -$
and $- \hstar -$ are bifunctors from (Hilbert spaces, bounded maps)
$\times$(Hilbert spaces, contractions) to (Hilbert spaces, bounded maps), and the third and fourth imply that they restrict to bifunctors
from (Hilbert spaces, contractions)$^2$ to (Hilbert spaces, contractions),
from (Hilbert spaces, isometries)$^2$ to (Hilbert spaces, isometries), and from (Hilbert spaces, unitaries)$^2$ to (Hilbert spaces, unitaries).
\end{remark}

\begin{remark} \label{remSPsubspaces}
Together, the fourth and seventh parts of Proposition
\ref{propSPopproperties} imply that, given isometries
$W: \HH_1 \to \HH_2$ and $V: \LL_1 \to \LL_2$, one obtains
an isometry $\HH_1 \star \LL_1 \to \HH_2 \star \LL_2$ which
may be constructed either as $(W \oplus \dss{\text{id}}{\com})
\star V$ or as $(W \oplus \dss{\text{id}}{\com} \oplus V)
\hstar V$.  The sixth part then implies that this
induced isometry maps each summand of $\HH_1 \star \LL_1$
into the corresponding summand of $\HH_2 \star \LL_2$.
Hence, if $\HH_1 \subset \HH_2$ and $\LL_1 \subset \LL_2$,
we may regard $\HH_1 \star \LL_1$ as a subspace of
$\HH_2 \star \LL_2$.
\end{remark}

An important special case of Proposition
\ref{propSPopsummands} occurs when $\HH_1 = \HH_2$,
$\LL_1 = \LL_2$, and $V$ is the identity map.  In
 this case we obtain unital representations
of both $B(\HH^+)$ and $B(\KK)$ on $\HH \star \LL$,
given by $S \mapsto S \star I$ and $T \mapsto T \hstar I$.

\begin{proposition} \label{prop3UEreps}
Let $\HH$ and $\LL$ be Hilbert spaces and $\KK = \HH^+ \oplus \LL$.
Let $\Phi: B(\HH^+) \to B(\HH \star \LL)$ and
$\Psi: B(\KK) \to B(\HH \star \LL)$ be the representations induced
by the unitary equivalences of Proposition \ref{prop3UE}.  Let
$b \in B(\HH^+)$, $a \in B(\KK)$, $h \in \HH^+$, $h_0 \in \HH$,
$k \in \KK$, $\ell \in \LL^+$, and $\xi \in \LL^{+ \otimes n} \otimes \LL$ for some $n \geq 0$, and suppose that
\begin{align*}
b \Omega &= \alpha \Omega + y, \qquad \qquad \alpha \in \com, y \in \HH\\
b h_0 &= \beta \Omega + z, \qquad \qquad \beta \in \com, z \in \HH\\
a h_0 &= \eta + w, \qquad \qquad \eta \in \HH, w \in \LL^+\\
a \ell &= \zeta+u, \qquad \qquad \zeta \in \HH, u \in \LL^+.
\end{align*}
Then
\begin{align*}
\Phi(b) h &= h\\
\Phi(b) \xi &= \alpha \xi + (y \otimes \xi)\\
\Phi(b)(h_0 \otimes \xi) &= \beta \xi + (z \otimes \xi)\\
\Psi(a) k &= a k \\
\Psi(a)(h_0 \otimes \xi) &= (\eta \otimes \xi) + (w \otimes \xi)\\
\Psi(a)(\eta \otimes \xi) &= (\zeta \otimes \xi) + (u \otimes \xi).
\end{align*}
\end{proposition}

The following is an immediate consequence:

\begin{corollary} \label{corinvariant}
In the situation of Proposition \ref{prop3UEreps},
the subspaces $\HH^+$ and\\
$(\LL^{+ \otimes n} \otimes \LL)
\oplus (\HH \otimes \LL^{+ \otimes n} \otimes \LL)$ of
$\HH \star \LL$ are $\Phi$-invariant, while the subspaces\\
$\HH^+ \oplus \LL$ and $(\HH \otimes \LL^{+ \otimes n} \otimes \LL)
\oplus (\LL^{+ \otimes (n+1)} \otimes \LL)$ are $\Psi$-invariant.
\end{corollary}

We visualize this corollary using a stairstep diagram:
\[
\begin{matrix}
\HH^+ \\
\LL & (\HH \otimes \LL) \\
& (\LL^+ \otimes \LL) & (\HH \otimes \LL^+ \otimes \LL) \\
& & (\LL^{+ \otimes 2} \otimes \LL) & (\HH \otimes \LL^{+ \otimes 2}
\otimes \LL) \\
& & & \ddots
\end{matrix}
\]
The rows here are $\Phi$-invariant, while the columns are $\Psi$-invariant.  Equivalently, $\Phi$ and $\Psi$
have staggered block-diagonal decompositions:
\[
\Phi = \begin{bmatrix} * & 0  & 0 & 0 & 0 & \dots \\
0 & * & * & 0 & 0 & \dots\\
0 & * & * & 0 & 0 & \dots\\
0 & 0 & 0 & * & * & \dots\\
0 & 0 & 0 & * & * & \dots\\
\vdots & \vdots & \vdots & \vdots & \vdots & \ddots
\end{bmatrix}, \qquad \qquad \qquad
\Psi = \begin{bmatrix} * & *  & 0 & 0 & 0 & 0 & \dots \\
* & * & 0 & 0 & 0 & 0 & \dots\\
0 & 0 & * & * & 0 & 0 & \dots\\
0 & 0 & * & * & 0 & 0 & \dots\\
0 & 0 & 0 & 0 & * & * & \dots\\
0 & 0 & 0 & 0 & * & * & \dots\\
\vdots & \vdots & \vdots & \vdots & \vdots & \vdots & \ddots
\end{bmatrix}
\]
We see that a sufficiently long word $\Phi(b_0) \Psi(a_1)
\Phi(b_1) \Psi(a_2) \cdots$ applied to a vector in one of
these subspaces could have a nonzero component in any other subspace.  Keeping track of such components will become important later on.

\begin{remark} \label{remsauvageotvsfree}
For simplicity of definition, we have thus far begun with Hilbert spaces $\HH$ and $\LL$, and defined the space $\KK = \HH^+ \oplus \LL$ in terms of them.  In application, however, we will begin with an inclusion
$H \subset K$ (obtained from Stinespring dilation), select
 a unit vector $\Omega \in H$, and form the Sauvageot product
 $H^- \star (K \ominus H)$.
As noted above, $B(H^- \star (K \ominus H))$ contains unital
copies of both $B(H)$ and $B(K)$.  In this alone, however, it
is no different from $B(H \otimes K)$ or $B(H * K)$, where
$H * K$ denotes the free product of Hilbert spaces in the
sense of \cite{VoiculescuSymmetries}.  The crucial difference
is that, when both are represented on $H^- \star (K \ominus H)$,
the copy of $B(H)$ is a corner of the copy of $B(K)$; if
$H \subset K$ is a Stinespring dilation, the compression will
implement a given unital completely positive map.
\end{remark}

\section{Sauvageot Products of C$^*$-Algebras and W$^*$-Algebras}
We now begin our construction of the Sauvageot product of
unital C$^*$-algebras with respect to a unital completely
positive map; the construction requires the choice of
a state on one of the C$^*$-algebras, prompting the following definition.

\begin{definition} \label{defCPtuple}
A \textbf{CPC$^*$-tuple} (resp.\ \textbf{CPW$^*$-tuple})
is a quadruple $(A, B, \phi, \omega)$ where
$A$ and $B$ are unital C$^*$-algebras (resp.\ W$^*$-algebras),
$\phi: A \to B$ a unital (normal) completely positive map, and $\omega$ a (normal) state on $B$.
The term \textbf{CP-tuple} will refer to CPC$^*$- and
CPW$^*$-tuples together.  A CP-tuple
is said to be \textbf{faithful} if $\omega$ is a faithful state.
\end{definition}

\begin{definition} \label{defrepCPtuple}
A \textbf{representation} of a CPC$^*$-tuple $(A, B, \phi, \omega)$
is a sextuple $(H, \Omega, \dss{\pi}{R}, K, V, \dss{\pi}{L})$
where
\begin{enumerate}
    \item $H$ is a Hilbert space
    \item $\Omega \in H$ is a unit vector
    \item $\dss{\pi}{R}: B \to B(H)$ is a unital *-homomorphism
    such that $\la \Omega, \dss{\pi}{R}(\cdot) \Omega
    \ra = \omega(\cdot)$
    \item $K$ is a Hilbert space
    \item $V: H \to K$ is an isometry
    \item $\dss{\pi}{L}: A \to B(K)$ is a unital *-homomorphism
    such that $V^* \dss{\pi}{L}(\cdot) V =
    \dss{\pi}{R}(\phi(\cdot))$.
\end{enumerate}
For a CPW$^*$-tuple, we also require that $\dss{\pi}{R}$
and $\dss{\pi}{L}$ be normal.  A representation is
\textbf{right-faithful} if $\dss{\pi}{R}$ is injective (which is
automatically the case for a representation of a faithful
CP-tuple), and \textbf{left-faithful} if $\dss{\pi}{L}$ is
injective.  We also refer to $(H, \Omega, \dss{\pi}{R})$
satisfying the first three criteria as a \textbf{representation} of $(A, \omega)$.
\end{definition}

From now until Definition \ref{defSPCPtuple}
we fix a CP-tuple $(A, B, \phi, \omega)$
and a right-faithful representation
$(H, \Omega, \dss{\pi}{R}, K, V, \dss{\pi}{L})$.
We introduce the following additional notation:
\begin{itemize}
    \item $L=K \ominus VH$
    \item $\Hh = H^- \star L$
    \item $\dss{\psi}{L}: A \to B(\Hh)$ and $\dss{\psi}{R}: B \to B(\Hh)$ are the compositions of $\dss{\pi}{L}$ and $\dss{\pi}{R}$ with the representations of Proposition \ref{prop3UEreps}
    \item $A \star B$ is the C$^*$-subalgebra (resp.\ von Neumann
    subalgebra) of $B(\Hh)$
    generated by the images of $\dss{\psi}{L}$
    and $\dss{\psi}{R}$
    \item $H' = H \ominus \overline{\dss{\pi}{R}(B) \Omega}$, which
    could be zero
    \item $q_n$ for $n \geq 0$ is the projection from
    $\Hh$ onto the subspace $H' \otimes L^{+ \otimes n}
    \otimes L$ of $H \otimes L^{+ \otimes n} \otimes L$
    \item $\Cc: B(\Hh) \to B(\Hh)$ is the non-unital conditional
    expectation
    \[
    \Cc(T) = \dss{P}{H} T \dss{P}{H} + \sum_{n=0}^\infty q_n T q_n.
    \]
\end{itemize}

\begin{proposition} \label{propcovariantcorner}
\[
\Cc \circ \dss{\psi}{L} = \Cc \circ \dss{\psi}{R}
\circ \phi.
\]
\end{proposition}

\begin{proof}
%Let $\dss{C}{H}$ denote the non-unital conditional expectation
%$\dss{C}{H}(T) = \dss{P}{H} T \dss{P}{H}$ on $B(\Hh)$,
%and for $n \geq 0$ let $C_n(T) = q_n T q_n$.  Now
%for any $a \in A$ and $h \in H$, we have by Proposition \ref{prop3UEreps} that $\dss{\psi}{L}(a) h \in K$, so that
%\[
%\dss{P}{H} \dss{\psi}{L}(a) \dss{P}{H} h
%= \dss{P}{K} \dss{\pi}{L}(a) \dss{P}{K} h
%= \dss{\pi}{R}(\phi(a)) h
%= \dss{P}{H} \dss{\psi}{R}(\phi(a)) \dss{P}{H} h
%\]
%so that $\dss{C}{H} \circ \dss{\psi}{L} = \dss{C}{H}
%\circ \dss{\psi}{R} \circ \phi$.  Similarly, for any $n \geq 0$,
%$h' \in H'$, and $\xi \in H' \otimes L^{+ \otimes n} \otimes L$,
%we have by Proposition \ref{prop3UEreps} that
%\begin{align*}
%\dss{\psi}{L}(a) [h' \otimes \xi]
%&= (V V^* \dss{\pi}{L}(a) h') \otimes \xi
%+ [(I - V V^*) \dss{\pi}{L}(a) h'] \otimes \xi\\
%&= [\dss{\pi}{R}(\phi(a))h'] \otimes \xi
%+ [(I - VV^*) \dss{\pi}{L}(a) h'] \otimes \xi
%\end{align*}
%and, as the second term is contained in $L \otimes L^{+ \otimes n} \otimes L$ which is orthogonal to $H' \otimes L^{+ \otimes n} \otimes L$, it follows that
%\[
%q_n \dss{\psi}{L}(a) [h' \otimes \xi] = [\dss{\pi}{R}(\phi(a))h']
%\otimes \xi = [\dss{\psi}{R}(\phi(a))](h' \otimes \xi)
%\]
%and therefore that
%\[
%C_n \circ \dss{\psi}{L}(a)
%= q_n \dss{\psi}{L}(a) q_n = q_n \dss{\psi}{R}(\phi(a)) q_n
%= C_n \circ \dss{\psi}{R}(\phi(a)).
%\]
%Summing this equation over $n$ and adding the corresponding
%equation for $\dss{C}{H}$ yields the result.
For $a \in A$ and $h \in H$, let
$\dss{\pi}{L}(a) h = x + \ell$ with $x \in h$ and $\ell
\in L$; then $x = \dss{P}{H} \dss{\pi}{L}(a) \dss{P}{H} h
= \dss{\pi}{R}(\phi(a)) h$.  It follows from
Proposition \ref{prop3UEreps} that $\dss{\psi}{L}(a) h
= x +\ell$, so that
\[
\dss{P}{H} \dss{\psi}{L}(a) \dss{P}{H}
= x = \dss{P}{H} \dss{\pi}{R}(\phi(a)) \dss{P}{H}.
\]
Similarly, $q_n \dss{\psi}{L}(a) q_n
= q_n \dss{\pi}{R}(\phi(a)) q_n$ for all $n \geq 0$.
Summing over $n$ yields the result.
\end{proof}

For the next lemma and proposition we use
$E_n$ to denote the subspace $L^{+ \otimes n} \otimes L$
of $\Hh$.

\begin{lemma} \label{lemhorriblyunmotivated}
Let $\zeta \in E_n$.
\begin{enumerate}
        \item Let $a \in A$ and $b \in B$ with $\omega(b) = 0$.  Then
        \[
        \big[ \dss{\psi}{L}(a) - \dss{\psi}{R}(\phi(a))  \big] \dss{\psi}{R}(b)  \zeta
        = [ \dss{P}{L^+} \dss{\pi}{L}(a) \dss{\pi}{R}(b) \Omega] \otimes \zeta
        - \omega(\phi(a) b) \zeta \in E_n \oplus E_{n+1}.
        \]

        \item More generally, given $a_1, \dots, a_k \in A$
        and $b_1, \dots, b_k \in B$ such that $\omega(b_i) = 0$ for
        $i = 1, \dots, k$, if we define
        \[
        \zeta_k =  \left( \prod_{i=1}^k \big[ \dss{\psi}{L}(a_i) - \dss{\psi}{R}(\phi(a_i)) \big] \dss{\psi}{R}(b_i)
        \right) \zeta,
        \]
        then
        \[
        \zeta_k \in \bigoplus_{i=0}^k E_{n+i} \text{ with }
        P_{E_n} (\zeta_k)  = (-1)^k \prod_{i=1}^k \omega(\phi(a_i) b_i) \zeta.
        \]
    \end{enumerate}
\end{lemma}

\begin{proof}
The stipulation that $\omega(b) = 0$ implies that $\dss{\pi}{R}(b) \Omega \in H^-$,
so that
\[
\xi := \dss{\psi}{R}(b) \zeta = (\dss{\pi}{R}(b) \Omega) \otimes \zeta
\in H^- \otimes E_n
\]
where we have used the calculations in Proposition (\ref{prop3UEreps}). We have now to apply two different operators to $\xi$ and subtract the results.  First, when we apply $\dss{\psi}{R}(\phi(a))$ we get
    \[
    \dss{\psi}{R}(\phi(a)) \xi = (\dss{\pi}{R}(\phi(a) b) \Omega) \otimes \zeta
    = [\omega(\phi(a) b) \Omega] \otimes \eta + [\dss{P}{H^-} \dss{\pi}{R}(\phi(a) b) \Omega] \otimes \zeta
    \]
    by virtue of the fact that $\dss{\pi}{R}(\beta) \Omega = \omega(\beta) \Omega + \dss{P}{H^-} (\dss{\pi}{R}(\beta) \Omega)$
    for all $b \in B$.    Secondly, we apply $\dss{\psi}{L}(a)$ as follows:
    \begin{align*}
    \dss{\psi}{L}(a) \xi &= (\dss{\pi}{L}(a) \dss{\pi}{R}(b) \Omega) \otimes \zeta \\
    &= \Big[ \left(\dss{P}{H^-} \dss{\pi}{L}(a) \dss{\pi}{R}(b) \Omega \right)
    \oplus \left(\dss{P}{L^+} \dss{\pi}{L}(a) \dss{\pi}{R}(b) \Omega \right) \Big] \otimes \zeta\\
    &= \Big[ \left(\dss{P}{H^-} P_H \dss{\pi}{L}(a) \dss{\pi}{R}(b) \Omega \right)
    \oplus \left(\dss{P}{L^+} \dss{\pi}{L}(a) \dss{\pi}{R}(b) \Omega \right) \Big] \otimes \zeta\\
    &= \Big[ \left(\dss{P}{H^-} \dss{\pi}{R}(\phi(a)b) \Omega \right) \otimes \zeta\Big]
    \oplus \Big[\left(\dss{P}{L^+} \dss{\pi}{L}(a) \dss{\pi}{R}(b) \Omega \right) \otimes \zeta\Big].
    \end{align*}
    Subtracting  yields the desired result.
    The second assertion of the lemma follows by induction.
\end{proof}

We now connect the current material to chapter \ref{chapliberation}.

\begin{proposition} \label{propcornerliberation}
%let $\rho$ denote the completely positive map $\dss{\psi}{R} \circ (\Cc \circ %\dss{\psi}{R})\inv \circ \Cc$ from $\dss{\psi}{L}(A)$ to $\dss{\psi}{R}(B)$.
$(B(\Hh), \dss{\psi}{L}, \dss{\psi}{R}, \Cc)$ is a right-liberating
representation of $(A, B, \phi, \omega)$.
\end{proposition}

\begin{proof}
    %\item For any $X \in B(\Hh)$,
%    \begin{align*}
%    \omega(\Cc(X)) &= \left\la \Omega, \left(\dss{P}{H} X \dss{P}{H}
%    + \sum_{n=0}^\infty q_n X q_n\right) \Omega \right\ra
%    = \la \Omega, \dss{P}{H} X \dss{P}{H} \Omega  \ra\\
%    &= \la \dss{P}{H} \Omega, X \dss{P}{H} \Omega \ra
%    = \la \Omega, X \Omega \ra = \omega(X).
%    \end{align*}  AGENDUM: Seems unnecessary.

   Since $H$ and each $H' \otimes E_n$ are
    $\dss{\psi}{R}$-invariant subspaces by Proposition
    \ref{prop3UEreps}, their projections all commute
    with $\dss{\psi}{R}$, so that $\Cc$ is a
    $\dss{\psi}{R}(B)$-bimodule map.  Now
    let $\xi \in H$ and let $a_1, \dots, a_n \in A$,
    $b_1, \dots, b_n \in B$ such that $\omega(b_2) = \dots = \omega(b_n) = 0$.
    Define the operators
    \[
    T_k = \Big(\dss{\psi}{L}(a_k) - \dss{\psi}{R}(\phi(a_k)) \Big)\dss{\psi}{R}(b_k)
    \]
    on $\Hh$, and the vectors
    \[
    \zeta_k = T_k \dots T_1 \xi \in \Hh.
    \]
    We will show by induction that $\zeta_k \in \displaystyle\bigoplus_{j=0}^{k-1} E_j$,
    which is contained in the kernel of $\dss{P}{H}$; it will follow that
    $\dss{P}{H} T_k \dots T_1 \dss{P}{H} = 0$.  For the base case $k=1$, we have
    $\dss{\psi}{R}(b_1) \xi = \dss{\pi}{R}(b_1) \xi,$ so that
    $\dss{\psi}{R}(\phi(a_1)) \dss{\psi}{R}(b_1) \xi = \dss{\pi}{R}(\phi(a_1) b_1) \xi$.
    We also have
    \begin{align*}
    \dss{\psi}{L}(a_1) \dss{\psi}{R}(b_1) \xi &= \dss{\pi}{L}(a_1) \dss{\pi}{R}(b_1) \xi \\
    &= \Big( \dss{P}{L} \dss{\pi}{L}(a_1) \dss{\pi}{R}(b_1) \xi \oplus
    \dss{P}{H} \dss{\pi}{L}(a_1) \dss{\pi}{R}(b_1) \xi \Big) \\
    &= \Big( \dss{P}{L} \dss{\pi}{L}(a_1) \dss{\pi}{R}(b_1) \xi \oplus
    \dss{P}{H} \dss{\pi}{L}(a_1) \dss{P}{H} \dss{\pi}{R}(b_1) \xi \Big) \\
    &= \Big( \dss{P}{L} \dss{\pi}{L}(a_1) \dss{\pi}{R}(b_1) \xi \oplus
    \dss{\pi}{R}(\phi(a_1)) \dss{\pi}{R}(b_1) \xi \Big)
    \end{align*}
    and subtracting yields
    \[
    \zeta_1 = \dss{P}{L} \dss{\pi}{L}(a_1) \dss{\pi}{R}(b_1) \xi \in E_0
    \]
    as desired.  The inductive step is immediate from Lemma
    \ref{lemhorriblyunmotivated}.

    Similarly, for $\xi \in H'$ and $\eta \in E_n$,
    $\dss{\psi}{R}(b_1)(\xi \otimes \eta) = (\dss{\pi}{R}(b_1) \xi) \otimes \eta$ so that
    \[
    \dss{\psi}{R}(\phi(a_1)) \dss{\psi}{R}(b_1) (\xi \otimes \eta)
    = [\dss{\pi}{R}(\phi(a_1) b_1) \xi] \otimes \eta.
    \]
    Then
    \begin{align*}
    \dss{\psi}{L}(a_1) \dss{\psi}{R}(b_1) (\xi \otimes \eta)
    &= \dss{\psi}{L}(a_1)[(\dss{\pi}{R}(b_1) \xi) \otimes \eta] \\
    &= [\dss{\pi}{L}(a_1) \dss{\pi}{R}(b_1) \xi] \otimes \eta \\
    &= \Big[ \dss{P}{L} \dss{\pi}{L}(a_1) \dss{\pi}{R}(b_1) \xi \oplus
    \dss{P}{H} \dss{\pi}{L}(a_1) \dss{\pi}{R}(b_1) \xi \Big] \otimes \eta \\
    &= \Big[ \dss{P}{L} \dss{\pi}{L}(a_1) \dss{\pi}{R}(b_1) \xi \oplus
    \dss{P}{H} \dss{\pi}{L}(a_1) \dss{P}{H} \dss{\pi}{R}(b_1) \xi \Big] \otimes \eta \\
    &= \Big[ \dss{P}{L} \dss{\pi}{L}(a_1) \dss{\pi}{R}(b_1) \xi \oplus
    \dss{\pi}{R}(\phi(a_1)) \dss{\pi}{R}(b_1) \xi \Big] \otimes \eta
    \end{align*}
    and subtracting yields
    \[
    \zeta_1 = [\dss{P}{L} \dss{\pi}{L}(a_1) \dss{\pi}{R}(b_1) \xi] \otimes \eta \in E_{n+1}.
    \]
    It follows by induction that $\zeta_k \in \displaystyle\bigoplus_{j=0}^{k-1} E_{j+n}$, so that $q_n \zeta_k = 0$; hence
    $q_n T_k \dots T_1 q_n = 0$.  Summing over $n$,
    we have $\Cc(T_k \dots T_1) = 0$.
\end{proof}

\begin{corollary} \label{corcornermapsinto}
\[
\Cc(A \star B) = \Cc(\dss{\psi}{R}(B)).
\]
\end{corollary}

\begin{proof}
This is an immediate consequence of Proposition
\ref{propcornerliberation} together with
Corollary \ref{corrightlibrepgenerated} and the
norm continuity and normality of $\Cc$.
\end{proof}

Before making our next definition, we note that the
right-faithfulness of our representation guarantees that
$b \mapsto \dss{P}{H} \dss{\psi}{R}(b) \dss{P}{H}$ is
injective, so that $\Cc \circ \dss{\psi}{R}$ is
injective as well.

\begin{definition} \label{defSPretraction}
The \textbf{Sauvageot retraction} for the given tuple
and representation is the map $\theta: A \star B \to B$
given by
\[
\theta = (\Cc \circ \dss{\psi}{R})\inv \circ \Cc.
\]
\end{definition}

The Sauvageot retraction is well-defined by Corollary
 \ref{corcornermapsinto}, and is
 a retraction with respect to $\dss{\psi}{R}$; furthermore,
 as a consequence of Proposition \ref{propcovariantcorner},
 it factors $\phi$ in the sense that
 \begin{equation} \label{eqnthetafactorsphi}
 \theta \circ \dss{\psi}{L} = \phi.
 \end{equation}

 Furthermore, the following is an immediate consequence
 of Proposition \ref{propcornerliberation}:

\begin{corollary} \label{corthetaliberates}
$(A \star B, \dss{\psi}{L}, \dss{\psi}{R},
\dss{\psi}{R} \circ \theta)$ is a right-liberating
representation of $(A, B, \phi, \omega)$.
\end{corollary}

\begin{definition} \label{defSPCPtuple}
Given a CP-tuple $(A, B, \phi, \omega)$ and a right-liberating
representation $(H, \Omega, \dss{\pi}{R}, K, V, \dss{\pi}{L})$,
the \textbf{Sauvageot product of the tuple realized by
the representation} is the tuple $(A \star B, \dss{\psi}{L},
\dss{\psi}{R}, \theta)$ of objects constructed as above.
\end{definition}

\section{Induced Morphisms and Uniqueness}
We pause now to consider an analogy with other product constructions.  In building either the (minimal) tensor product or the free product of C$^*$-algebras (resp.\ W$^*$-algebras) $A$ and $B$, one can proceed as follows:
\begin{enumerate}
    \item Represent $A$ and $B$ on Hilbert spaces $H$ and $K$
    \item Form the product Hilbert space $H \otimes K$ or $H * K$
    \item Lift the representations of $A$ and $B$ to
    representations of each on this product Hilbert space
    \item Take the C$^*$-subalgebra (resp.\ von Neumann subalgebra)
    generated by the images
    of these representations.
\end{enumerate}
In both cases, one can show that the resulting C$^*$-algebra
(resp.\ W$^*$-algebra)
$A \otimes B$ or $A * B$ is, up to isomorphism, independent of
the choice of the representations of $A$ and $B$ provided
both are faithful.

We have followed the same outline in constructing $A \star B$,
and come now to the question of the independence of this object
from the representations used to produce it.  It turns out that
we need some more complicated hypotheses on the representation,
resulting from the fact that a representation of a CP-tuple
is more complicated than a representation of a C$^*$-algebra
or W$^*$-algebra,
as well as the fact that the product $A \star B$ comes with
the additional information of a retraction onto $B$.

\begin{definition} \label{defdecompfaithful}
Let $(A, B, \phi, \omega)$ be a CP-tuple,
$(H, \Omega, \dss{\pi}{R}, K, V, \dss{\pi}{L})$ a
representation, and $L = K \ominus VH$.
\begin{itemize}
    \item A \textbf{decomposition} of the representation is
    a pair $(L', L'')$ of $\dss{\pi}{L}$-invariant
    subspaces $L' \subset L$ and $L'' \subset L^+$, with the properties
    \begin{align*}
    L &\subset L' + \overline{\dss{\pi}{L}(A) VH}\\
    L^+ &\subset L'' + \overline{\dss{\pi}{L}(A) V H^-}.
    \end{align*}
    \item A decomposition is \textbf{faithful} if the subrepresentation
    $\dss{\pi}{L} \big|_{L'}$ is faithful.
    \item A representation for which there exists a
    faithful decomposition is \textbf{faithfully
    decomposable}.
    \item A representation is \textbf{faithful} if it
    is right-faithful and faithfully decomposable.
\end{itemize}
\end{definition}

\begin{proposition} \label{propexistsfaithfulrep}
Every CP-tuple has a faithful representation.
\end{proposition}

\begin{proof}
We begin by letting $(H, \Omega, \dss{\pi}{R})$ be the GNS
construction for $(B, \omega)$; if $\omega$ is not
faithful, we replace $(H, \dss{\pi}{R})$ by its direct
sum with some faithful representation of $B$.  This guarantees
that our representation will be right-faithful.

Next, let $(K, V, \dss{\pi}{L})$ be the minimal
Stinespring dilation of $\dss{\pi}{R} \circ \phi$, and
let $L' = K \ominus \overline{\dss{\pi}{L}(A) VH}$
and $L'' = K \ominus \overline{\dss{\pi}{L}(A) H^-}$.
If $\dss{\pi}{L} \Big|_{L'}$ is not faithful,
we replace $(K, \dss{\pi}{L})$ by its direct sum with
some faithful representation of $A$, thereby
guaranteeing faithful decomposability.
\end{proof}

When some of these additional hypotheses are satisfied,
we can define a retraction from $A \star B$ to $A$ which
has properties analogous to the retraction $A \star B
\sa{\theta} B$ already discussed.  We continue our standard
notation for a CP-tuple, a right-faithful representation,
and the corresponding realization of the Sauvageot product,
and now fix a decomposition $(L', L'')$ as well (not assumed faithful unless specified).  We introduce additional notation:
\begin{itemize}
    \item $E_0'$ is the subspace $L' \subset L$
    \item For $n \geq 1$, $E_n'$ is the subspace
    $L'' \otimes L^{+ \otimes (n-1)} \otimes L
    \subset E_n$
    \item For all $n \geq 0$, $p_n$ is the projection
    from $\Hh$ onto $E_n'$
    \item For all $n \geq 0$, $F_n = H \otimes E_n$
    and $F_n' = H' \otimes E_n$; recall
    that $q_n$ is the projection onto $F_n'$
    \item $E_{-1} = \com \Omega$ and $F_{-1} = H^-$.
\end{itemize}

\begin{definition} \label{defleftcornermap}
The \textbf{left corner map}
for the given realization and decomposition is
the non-unital conditional expectation
$\Cc'$ on $B(\Hh)$ defined by
\[
\Cc'(T) = \sum_{n=0}^\infty p_n T p_n.
\]
\end{definition}

\begin{lemma} \label{lemintersectkernels}
Let $J \subset A \star B$ be an ideal contained
in the intersection of the kernels of $\Cc$ and $\Cc'$.
Then $J = \{0\}$.
\end{lemma}

\begin{proof}
Let $J^0$ denote the annihilator of $J$ in $\Hh$, i.e. the largest (necessarily closed)
subspace such that $J J^0 = \{0\}$.
\begin{enumerate}
    \item For any $\alpha \in J$, note that $\alpha^* \alpha \in J$; then
    \[
    (\alpha \dss{P}{H})^* (\alpha \dss{P}{H}) = \dss{P}{H} \alpha^* \alpha \dss{P}{H} \leq \Cc[\alpha^* \alpha] = 0,
    \]
    so that $\alpha \dss{P}{H} = 0$.  Similarly, $\alpha q_n = \alpha p_n = 0$ for all $n$.  Hence $H \subseteq J^0,$ and for each
    $n \geq 0$, $E_n'  \subseteq J^0$ and $H' \otimes E_n\subseteq J^0$.

    \item We will prove by induction
    that $E_n \subseteq J^0$ and $F_n \subseteq J^0$;
    the base case $n = -1$ was just established,
    since $E_{-1} \oplus F_{-1} = H \subset J^0$.

    \item Suppose $E_n \subseteq J^0$ and $F_n \subseteq J^0$.
    \begin{enumerate}
        \item Since $J \psi_L(A) F_n \subseteq J F_n = \{0\}$, we have $\psi_L(A) F_n \subseteq J^0$.
        \item By Proposition \ref{prop3UEreps}, $[\psi_L(A) H^-] \otimes E_n \subseteq \psi_L(A) F_n$,
        so that $[\psi_L(A) H^-] \otimes E_n \subseteq J^0$, for the case $n \geq 0$; for $n = -1$,
        we have $\psi_L(A) H \subseteq J^0$.
        \item Since we already know $L' \subseteq J^0$ (in the case $n = -1$) or
        $L'' \otimes E_n \subseteq J^0$ (in the case $n \geq 0$), it follows that
        $E_{n+1} = L \subset \overline{\psi_L(A) H} + L' \subseteq J^0$ for the case $n = -1$,
        or $E_{n+1} = L^+ \otimes E_n \subset [\overline{\pi_L(A) H^-} + L''] \otimes E_n \subseteq J_0$
        for the case $n \geq 0$.
        \item Since $J \psi_R(B) E_{n+1} \subseteq J E_{n+1} = \{0\}$, we have $\psi_R(B) E_{n+1} \subseteq J^0$.
        \item By Proposition \ref{prop3UEreps}, this
        implies
        $\overline{\pi_R(B) \Omega} \otimes E_{n+1} \subseteq \overline{\psi_R(B) E_{n+1}} \subseteq J^0$.
        \item Since we also have $F_n' \subseteq J^0$,
        \begin{align*}
        F_{n+1} &= H^- \otimes E_{n+1}
        \subseteq H \otimes E_{n+1} \\
        &= (\overline{\pi_R(B) \Omega} + H') \otimes E_{n+1} \subseteq \overline{\psi_R(B) E_{n+1}} + F_{n+1}' \subseteq J^0.
        \end{align*}
    \end{enumerate}
\end{enumerate}
\end{proof}

\begin{remark} \label{remideals}
We note that if $L'$ and $L''$ are both zero (for instance,
when $K$ is given by a minimal Stinespring dilation),
then $\Cc'$ is the zero map; in this case, the lemma
says that $\Cc$ has no nontrivial ideals in its kernel.
This corresponds to the fact that $A$ and $B$ together move $H$ around to all the other components of $\Hh$, in the sense that
  $\overline{(A \star B) H} = \Hh$.  On the other extreme,
if $(L', L'')$ is a faithful decomposition, one has
instead that $\overline{(A \star B) H
+ (A \star B) L' + (A \star B) L''} = \Hh$, but none
of $H, L', L''$ by itself is enough to reach all of
$\Hh$.  As a result, $\Cc$ and $\Cc'$ may each contain
ideals in their kernel, but these ideals are ``orthogonal''
in the sense of the lemma.
\end{remark}

%\begin{proposition} \label{propleftcornercovariant}
%Let $\hat{\omega}: B \to A$ be the unital completely
%positive map $\hat{\omega}(b) = \omega(b) \one$.
%Then $\Cc' \circ \dss{\psi}{L} \circ \hat{\omega}
%= \Cc' \circ \dss{\psi}{R}$.
%\end{proposition}
%
%\begin{proof}
%Given any $b \in B$, we have $\dss{\pi}{R}(b) \Omega
%= \omega(b) \Omega + h_0$ for some $h_0 \in H^-$.
%By Proposition \ref{prop3UEreps}, it follows that, for
%any $\xi \in E_n'$, $\dss{\psi}{R}(b) \xi =
%\omega(b) \xi + h_0 \otimes \xi$ and therefore that
%$p_n \dss{\psi}{R}(b) \xi = \omega(b) \xi$; hence,
%\[
%p_n \dss{\psi}{R}(b) p_n
%= \omega(b) \dss{\one}{E_n}
%= p_n \omega(b) \one p_n = p_n \dss{\psi}{L}(\hat{\omega}(b)) p_n.
%\]
%Summing over $n$ yields the result.
%\end{proof}

\begin{proposition} \label{propleftcornerliberated}
$(B(\Hh), \dss{\psi}{L}, \dss{\psi}{R}, \Cc')$ is
a left-liberating representation of $(A, B, \phi, \omega)$.
\end{proposition}

\begin{proof}
By Proposition \ref{prop3UEreps}, all of the
    $E_n'$ are $\dss{\psi}{L}$-invariant, so their
    projections commute with $\dss{\psi}{L}$; hence
    $\Cc'$ is a $\dss{\psi}{L}(A)$-bimodule map.  Now
    let $a_1, \dots, a_n \in A$
and $b_0, \dots, b_n \in B$ with $\omega(b_i) = 0$.  Let
$m \geq 0$ and $\xi \in E_m'$.  By Lemma \ref{lemhorriblyunmotivated},
\[
\prod_{k=1}^n \left[ \dss{\psi}{L}(a_k) - \dss{\psi}{R}(\phi(a_k))
\right] \xi \in \bigoplus_{k=0}^m E_{n+k}.
\]
Since $\omega(b_0) = 0$, it follows that $\dss{\pi}{R}
(b_0) \Omega \in H^-$, so that by Proposition \ref{prop3UEreps}
we obtain
\[
\dss{\psi}{R}(b_0) \prod_{k=1}^n \left[ \dss{\psi}{L}(a_k) - \dss{\psi}{R}(\phi(a_k))
\right] \xi \in \bigoplus_{k=0}^m H^- \otimes E_{n+k}.
\]
From this we see that
\[
p_n \dss{\psi}{R}(b_0) \prod_{k=1}^n \left[ \dss{\psi}{L}(a_k) - \dss{\psi}{R}(\phi(a_k))
\right] p_n = 0,
\]
and summing over $n$ finishes the proof.
\end{proof}

\begin{corollary} \label{corleftcornermapsinto}
$\Cc'(A \star B) = \Cc'(A)$.
\end{corollary}

\begin{proof}
This is an immediate consequence of Proposition \ref{propleftcornerliberated}, Corollary \ref{corleftlibgenerated},
and the contractivity (and, in case $(A, B, \phi, \omega)$
is a CPW$^*$-tuple, the normality) of $\Cc'$.
\end{proof}

In the case of a faithful decomposition, $\Cc' \circ \dss{\psi}{L}$ is injective, which allows us to make the following definition:

\begin{definition} \label{defleftretraction}
Given a CP-tuple, a faithful representation, and a choice
of faithful decomposition, the \textbf{left retraction} for the given tuple
and representation is the map $\theta: A \star B \to A$
given by
\[
\theta = (\Cc' \circ \dss{\psi}{L})\inv \circ \Cc'.
\]
This is well-defined by Corollary
 \ref{corleftcornermapsinto}, and is
 a retraction with respect to $\dss{\psi}{L}$.
\end{definition}

We come now to the main result of this section.

\begin{theorem} \label{thmSPhoms}
Let $(A_1, B_1, \phi_1, \omega_1)$ and $(A_2, B_2, \phi_2, \omega_2)$ be CPC$^*$-tuples (resp.\ CPW$^*$-tuples),
$(H_1, \Omega_1, \dss{\pi}{R}^{(1)},
K_1, V_1, \dss{\pi}{L}^{(1)})$ a faithful representation of the former, \\
$(H_2, \Omega_2, \dss{\pi}{R}^{(2)}, K_2, V_2, \dss{\pi}{L}^{(2)})$ a right-faithful representation of the latter, and\\
$(A_1 \star B_1, \dss{\psi}{L}^{(1)}, \dss{\psi}{R}^{(1)},\theta_1)$ and $(A_2 \star B_2,
\dss{\psi}{L}^{(2)}, \dss{\psi}{R}^{(2)}, \theta_2)$ the Sauvageot products realized by these
representations.  Let $f: A_1 \to A_2$ and $g: B_1 \to B_2$
be unital (normal) *-homomorphisms
satisfying $\phi_2 \circ f = g \circ \phi_1$ and
$\omega_2 \circ g = \omega_1$.  Then there is a unique (normal) unital *-homomorphism
$f \star g: A_1 \star B_1 \to A_2 \star B_2$
with the properties that
\begin{enumerate}
    \item $(f \star g) \circ \dss{\psi}{L}^{(1)}
    = \dss{\psi}{L}^{(2)} \circ f$
    \item $(f \star g) \circ \dss{\psi}{R}^{(1)} =
    \dss{\psi}{R}^{(2)} \circ g$
    \item $\theta_2 \circ (f \star g) = g \circ \theta_1$
\end{enumerate}
If $f$ and $g$ are both injective and
$(H_2, \Omega_2, \dss{\pi}{R}^{(2)},K_2, V_2, \dss{\pi}{L}^{(2)})$
is faithful, then $f \star g$ is injective.
\end{theorem}

\begin{proof}
Let $H = H_1 \oplus H_2$, $\Omega = \Omega_1$, $\dss{\pi}{R} = \dss{\pi}{R}^{(1)} \oplus (\dss{\pi}{R}^{(2)} \circ g)$,
$K = K_1 \oplus K_2$,
$V = V_1 \oplus V_2$, $\dss{\pi}{L} =
\dss{\pi}{L}^{(1)} \oplus (\dss{\pi}{L}^{(2)} \circ f)$.  Then
$(H, \Omega, \dss{\pi}{R}, K, V, \dss{\pi}{L})$ is another right-faithful representation of $(A_1, B_1, \phi_1, \omega_1)$.  Moreover, if $(L_1', L_1'')$ is a decomposition for $(H_1, \dots, \dss{\pi}{L}^{(1)})$, then
$L' = L_1' \oplus L_2$, $L'' = L_1'' \oplus L_2$
defines a decomposition $(L', L'')$
of $(H, \dots, \dss{\pi}{L})$, and the faithfulness
of $\dss{\pi}{L}^{(1)}$ on $L_1'$ implies the faithfulness
of $\dss{\pi}{L}$ on $L'$.

Let $(A_1 \tstar B_1, \dss{\psi}{L}, \dss{\psi}{R}, 
\theta)$ be the Sauvageot product realized
by $(H, \dots, \dss{\pi}{L})$, on the Hilbert space
$\Mm = H^- \star L$.

The inclusions of $H_1$ into $H$ and of $L_1$ into $L$ induce
an isometry $W: \Hh_1 \to \Hh$ as in Remark (\ref{remSPsubspaces}).  Moreover, by Equation (\ref{eqnSPopsintertwines}), this
isometry satisfies
\begin{equation} \label{eqnWintertwines}
W \circ \dss{\psi}{L}^{(1)}(\cdot) = \psi_L(\cdot) \circ W, \qquad
W \circ \dss{\psi}{R}^{(1)}(\cdot) = \psi_R(\cdot) \circ W.
\end{equation}
Let $\Psi$ be the restriction to $A_1 \tstar B_1$
of the (normal) unital CP map $T \mapsto W^* T W$, which
maps $B(\Hh)$ to $B(\Hh_1)$.  It follows from
Equation (\ref{eqnWintertwines}) that the image
of $W$ is invariant under $\dss{\psi}{L}$
and $\dss{\psi}{R}$, so that $W W^*$ commutes
with $A_1 \tstar B_1$.  Then
    \[
    \Psi(XY) = W^* XY W = W^*
    W W^* XY W =
    W^* X W W^* Y W
    = \Psi(X) \Psi(Y)
    \]
    for $X,Y \in A_1 \tstar B_1$, so that
    $\Psi$ is a *-homomorphism.

    Next, we show that
    $\Psi$ intertwines the representations,
    states, and retractions:
\begin{itemize}
    \item $\Psi \circ \dss{\psi}{L} = \dss{\psi}{L}^{(1)}$
    \item $\Psi \circ \dss{\psi}{R} = \dss{\psi}{R}^{(1)}$
    \item $\theta_1 \circ \Psi = \theta$
    \item $\theta_1' \circ \Psi = \theta'$
\end{itemize}

The first three are immediate consequences of Equation (\ref{eqnWintertwines}).  For the fourth and fifth, we have
by Equation (\ref{eqnSPopsintertwines}) that
\[
\Psi \circ \Cc(T) = W^*\dss{P}{H} T \dss{P}{H} W^* + \sum_n W^* p_n T p_n W = \dss{P}{H_1} W^* T W \dss{P}{H_1} + \sum_n p_{n,1} W^* T W p_{n,1} = \Cc_1 \circ \Psi(T)
\]
so that $\Psi \circ \Cc = \Cc_1 \circ \Psi$,
and similarly for $\Cc'$ and $\Cc_1'$.  Now
\[
\Psi \circ \Cc \circ \dss{\psi}{R} =
\Cc_1 \circ \Psi \circ \dss{\psi}{R} = \Cc_1 \circ \dss{\psi}{R}^{(1)},
\]
which is invertible; then
\[
(\Cc \circ \dss{\psi}{R})\inv =
(\Psi \circ \Cc \circ \dss{\psi}{R})\inv\circ \Psi
\circ \Cc \circ \dss{\psi}{R} \circ (\Cc \circ \dss{\psi}{R})\inv
= (\Psi \circ \Cc \circ \dss{\psi}{R})\inv \circ \Psi
\]
from which it follows that
\begin{align} \label{eqnintertwineretractions}
\theta &= (\Cc \circ \dss{\psi}{R})\inv \circ \Cc \nonumber\\
&= (\Psi \circ \Cc \circ \dss{\psi}{R})\inv \circ \Psi \circ \Cc\nonumber\\
&= (\Cc_1 \circ \dss{\psi}{R}^{(1)})\inv \circ \Cc_1 \circ \Psi \nonumber\\
&= \theta_1 \circ \Psi
\end{align}
and similarly
\begin{equation} \label{eqnintertwinesretractions2}
\theta' = \theta_1' \circ \Psi.
\end{equation}

Note that $\Psi$ maps into $A_1 \star B_1$,
as it maps both $\dss{\psi}{L}(A_1)$ and $\dss{\psi}{R}(B_1)$
into $A_1 \star B_1$, hence also the C$^*$-algebra (resp.\ von Neumann algebra) that they generate; moreover, it is onto $A_1 \star B_1$, since its range is a C$^*$-algebra (resp.\ von Neumann algebra)
which includes
both $\dss{\psi}{L}^{(1)}(A_1)$ and $\dss{\psi}{R}^{(1)}(B_1)$.

Next, $\Psi$ is injective, because
its kernel is an ideal in $A_1  \tstar B_1$ which, by equations
\ref{eqnintertwineretractions} and \ref{eqnintertwinesretractions2},
is contained in the kernels of both $\theta$ and $\theta'$, therefore also in the kernels of both $\Cc$ and $\Cc'$, and
hence is the zero ideal by Lemma \ref{lemintersectkernels}.
So $\Psi$ is an isomorphism from $A_1 \tstar B_1$
to $A_1 \star B_1$.

We repeat the above analysis for the inclusions of $H_2$ into $H$
and $K_2$ into $K$ to obtain a unital *-homomorphism
$\Xi: A_1 \tstar B_1 \to A_2 \star B_2$ such that
\begin{itemize}
    \item $\Xi \circ \dss{\psi}{L} = \dss{\psi}{L}^{(2)} \circ f$
    \item $\Xi \circ \dss{\psi}{R} = \dss{\psi}{R}^{(2)} \circ g$
    \item $\theta_2 \circ \Xi = g \circ \theta$
   % \item $\Omega_2' \circ \Xi = f \circ \hat{\Omega}'_2$.
\end{itemize}

We can now define $f \star g = \Xi \circ \Psi\inv:
A_1 \star B_1 \to A_2 \star B_2$.  Then we
obtain the enumerated properties of $f \star g$
by combining the lists of properties for $\Psi$
and $\Xi$, as
\[
(f \star g) \circ \dss{\psi}{L}^{(1)} = \Xi \circ \Psi\inv
\circ \dss{\psi}{L}^{(1)} = \Xi \circ \dss{\psi}{L} = \dss{\psi}{L}^{(2)} \circ f
\]
and similarly.

The uniqueness of $f \star g$ follows from the fact that it is contractive (resp.\ normal)
and is determined on the dense subalgebra of $A_1 \star B_1$
generated by $\dss{\psi}{L}^{(1)}(A_1)$ and $\dss{\psi}{R}^{(1)}(B_1)$.

Finally, if $f$ and $g$ are both injective and
$(H_2, \dots, \dss{\pi}{L}^{(2)})$ is faithful,
we can prove the additional property
$\theta_2' \circ \Xi = \Xi \circ f$,
 after which we prove $\Xi$ to be injective exactly as we did with
 $\Psi$.  Hence $f \star g$ is a composition
of injective maps.
\end{proof}

\begin{corollary} \label{corunique}
Let $(A, B, \phi, \omega)$ be a CP-tuple.  Then the realizations
of the Sauvageot product by any two faithful representations
are isomorphic.
\end{corollary}

Here ``isomorphic''  refers to an isomorphism which
intertwines the appropriate maps.  The proof is
simply to take the map $\dss{\text{id}}{A} \star
\dss{\text{id}}{B}$ constructed in the theorem.  Based on this corollary, we may now speak of \emph{the}
Sauvageot product of a CP-tuple.

Another special case of interest occurs when one, but
not both, of the initial maps is the identity.  The
results are summarized as follows.

\begin{corollary} \label{corcommutingsquare}
Let $A, B$ be unital C$^*$-algebras (resp.\ W$^*$-algebras),
$A \sa{f} B$ a unital (normal) *-homomorphism
and $C$ another unital C$^*$-algebra (resp.\ W$^*$-algebra).

\begin{enumerate}
    \item Let $B \sa{\phi} C$ be a (normal) unital CP map
    and $\omega$ a (normal) state on $C$.  Then,
    for the CP-tuples $(A, C, \phi \circ f, \omega)$
    and $(B, C, \phi, \omega)$ with Sauvageot retractions
    $A \star C \sa{\theta} C$ and $B \star C \sa{\eta} C$, the diagrams
    \[ \xymatrix{
    A \ar[r]^f \ar[d] & B \ar[d]\\
    A \star C \ar[r]_{f \star \text{id}} & B \star C
    } \qquad \qquad \qquad \xymatrix{
    A \star C \ar[r]^{f \star \text{id}} \ar[rd]_{\theta}
    & B \star C \ar[d]^{\eta}\\
    & C
    }\]
    commute.

    \item Let $C \sa{\phi} A$ be a (normal) unital CP
    map and $\omega$ a (normal) state on $B$.  Then, for
    the CP-tuples $(C, A, \phi, \omega \circ f)$ and
    $(C, B, f \circ \phi, \omega)$, the square
    \[ \xymatrix{
    A \ar[r]^f \ar[d] & B \ar[d]\\
    C \star A \ar[r]_{\text{id} \star f} & C \star B
    } \]
    commutes.
\end{enumerate}
\end{corollary}

The composition of Sauvageot products of maps obeys
the obvious functorial property:

\begin{proposition} \label{propfunctorial}
For $i = 1,2,3$ let $(A_i, B_i, \phi_i, \omega_i)$ be
CP-tuples, and for $i = 1,2$ let $A_i \sa{f_i} A_{i+1}$
and $B_i \sa{g_i} B_{i+1}$ be (normal) unital *-homomorphisms,
such that the diagram
\[ \xymatrix{
A_1 \ar[r]^{f_1} \ar[d]_{\phi_1} &A_2 \ar[r]^{f_2} \ar[d]_{\phi_2} &A_3 \ar[d]_{\phi_3} \\
B_1 \ar[r]^{g_1} \ar[rd]_{\omega_1} &B_2 \ar[r]^{g_2} \ar[d]_{\omega_2} &B_3 \ar[ld]^{\omega_3} \\
& \com &
} \]
commutes.  Then
\[
(f_2 \circ f_1) \star (g_2 \circ g_1)
= (f_2 \star g_2) \circ (f_1 \star g_1).
\]
\end{proposition}

Next, we note that the Sauvageot retraction possesses a certain
universal property.

\begin{proposition} \label{propuniqueretraction}
Let $(A, B, \phi, \omega)$ be a CP-tuple, with Sauvageot product
$A \star B$ and retraction $\theta: A \star B \to B$.
Suppose $\hat{\theta}: A \star B \to B$ is another (normal) retraction
with respect to $\dss{\psi}{R}$ such that
$(A \star B, \dss{\psi}{L}, \dss{\psi}{R},
\hat{\theta} \circ \dss{\psi}{R})$ is a right-liberating representation for $(A, B, \phi, \omega)$.  Then
$\hat{\theta} = \theta$.
\end{proposition}

\begin{proof}
Applying Theorem \ref{thmrightlibmoments} to
the conditional expectations $\dss{\psi}{R} \circ \theta$
and $\dss{\psi}{R} \circ \hat{\theta}$, we see that
they agree on a dense *-subalgebra of $A \star B$, hence
on the whole by continuity.  Since $\dss{\psi}{R}$ is
injective, this implies $\theta = \hat{\theta}$.
\end{proof}
%AGENDUM: Universal ``coproduct-like'' property more generally
% AGENDUM: (A \star B) \star_\theta B \simeq A \star B?  More generally, if there's an embedding of B into A wrt which \phi is a retraction, and with the liberation property, is A \star B isomorphic to A?
\section{Trivial Cases of the Sauvageot Product}
Tensor products have the property that $A \otimes \com
\simeq A \simeq \com \otimes A$ for any commutative unital
C$^*$-algebra $A$; similarly, unital free products have
the property that $A * \com \simeq A \simeq \com * A$
for any unital C$^*$-algebra $A$.  Moreover, amalgamated
free products satisfy $A *_A A \simeq A$.  We
now consider analogues of these properties for
the Sauvageot product.  These are of interest not only
for their own sake, but also as the base cases in the inductive
system of the next chapter.

\begin{proposition}[$\com \star \AAA \simeq \AAA$] \label{propcomstarA}
Let $\AAA$ be any unital C$^*$-algebra (resp.\ W$^*$-algebra),
$\upsilon: \com \to \AAA$ the embedding of $\com$,
and $\omega$ any (normal) state on $\AAA$.
Then the Sauvageot product $\com \star \AAA$ of the CP-tuple
$(\com, \AAA, \upsilon, \omega)$ is isomorphic
to $\AAA$; modulo this identification,
the embedding $\dss{\psi}{L}:
\com \to \com \star \AAA$ is $\upsilon$,
and $\dss{\psi}{R}: \AAA \to \com \star \AAA$
and $\E: \com \star \AAA \to \AAA$ are both the
identity map.
\end{proposition}

\begin{proof}
One can prove this by constructing a representation of
this CP-tuple; on the space $\Hh$, one
has $\dss{\psi}{L}$ mapping into $\dss{\psi}{R}(\AAA)$,
so that the algebra generated by both the images together
is isomorphic to $\AAA$.  Alternatively, right-liberation becomes trivial when one of the algebras involved is $\com$, so that $\E$ is multiplicative and hence is a *-homomorphic inverse for
$\dss{\psi}{R}$.
\end{proof}

\begin{remark} \label{remproductwithsubalgebra}
One might conjecture that, more generally, the Sauvageot product with respect to an embedding is trivial; that is, if
$A \sa{\iota} B$ is an embedding, or equivalently
if $A \subset B$ is an embedding, that $A \star B \simeq B$.

This turns out not to be the case.  We are interested in
whether $\dss{\psi}{L} = \dss{\psi}{R} \circ \iota$;
but on the subspace $L'$ in a faithful decomposition,
$\dss{\psi}{L}$ acts faithfully, whereas $\dss{\psi}{R}
\circ \iota$ acts in a trivial fashion (in particular,
the component in $L'$ of $\dss{\psi}{R}(\iota(a)) \xi$
for $\xi \in L'$ must be a scalar multiple of $\xi$).

This illustrates an important feature of the Sauvageot
product.  If we were to start by representing
$B$ on some $H$ through the GNS construction, then
use Stinespring dilation to obtain a representation of
$A$ on $K$, then
in the special case that the map from $A$ to $B$ is an embedding
(indeed, any homomorphism) one would have $K = H$ and therefore
$L = \{0\}$, from which it would follow that $\Hh \simeq H$
as well, and $A \star B \simeq B$.  But the Sauvageot product
is defined with respect to a \emph{faithful} representation, which involves taking direct sums at various points in the process so as
to avoid collapsing into triviality.
\end{remark}

\begin{proposition}[$\AAA \star \com \simeq \com$]
Let $\AAA$ be any unital C$^*$-algebra (resp.\ W$^*$-algebra),
and $\omega$ any (normal) state on $\AAA$.
Then the Sauvageot product $\AAA \star \com$ of the
CP-tuple $(\AAA, \com, \omega, \dss{\text{id}}{\com})$ is
isomorphic to $\AAA$; modulo this identification,
the left embedding $\dss{\psi}{L}: \AAA \to \AAA \star \com$
is the identity map, the right embedding
$\dss{\psi}{R}: \com \to \AAA \star \com$ is $\upsilon$,
and the retraction $\E: \AAA \star \com \to \com$
is $\omega$.
\end{proposition}

\begin{proof}
As with the previous proposition.
\end{proof}

\begin{remark} \label{remstarwithcom}
Now given a CP-tuple $(A, B, \phi, \omega)$, one
can identify $A$ with $\com \star A$ (resp.\ $A \star \com$)
and $B$ with $\com \star B$ (resp.\ $B \star \com$); it
is then natural to ask whether $\phi$ is thereby identified
with $\dss{\text{id}}{\com} \star \phi$ (resp.\
$\phi \star \dss{\text{id}}{\com}$).  The answer is yes;
indeed, this is a special case of Corollary \ref{corcommutingsquare}.
\end{remark}

% Thesis Chapter 4 by Dave Gaebler

\chapter{Algebraic C$^*$-Dilations through Iterated Products} \label{chapiteratedproducts}

\section{Introduction}
Having shown how to construct the Sauvageot product of a CP-tuple,
we now broach the question of how to iterate this product in order to construct dilations.  For motivation, we return again to the Daniell-Kolmogorov construction as viewed through the lens of the tensor product (Example \ref{exmarkovdilation}).

Recall that we begin with a compact Hausdorff space $S$ (the state
space of a Markov process), with corresponding path space
$\SSSS = S^{[0,\infty)}$; we use $\AAA$ to denote $C(S)$ and
$\Aa$ to denote $C(\SSSS)$, though we seek here to
construct $\Aa$ only through C$^*$-algebraic means, without
reference to $\SSSS$.  For
each finite subset $\gamma \subset [0,\infty)$, we let
$\AAA_\gamma$ denote a tensor product of $|\gamma|$ copies
of $C(S)$ with itself.  When we have constructed $\Aa$,
we will embed such an $\AAA_\gamma$ into it,
corresponding to those functions on the path space which only depend on times in $\gamma$.

For $\beta \leq \gamma$ we can embed $\AAA_\beta$ into
$\AAA_\gamma$ by tensoring with $\one$'s in all the missing coordinates. It is difficult to find notation which makes this more precise while maintaining the basic simplicity of the concept, but here are two attempts.  First, an example:  If $\gamma = \{t_1, \dots, t_7\}$ with
the times listed in increasing order, and $\beta = \{t_2, t_5, t_6\}$, then one embeds $\AAA_\beta$ into $\AAA_\gamma$ via
\[
f \otimes g \otimes h \longmapsto \one \otimes f \otimes \one
\otimes \one \otimes g \otimes h \otimes \one.
\]
Second, a general observation: Such an embedding can be built from repeated embeddings corresponding to adding a single time, so we
may reduce to the case $\beta = \{t_1, \dots, t_n\}$ and
$\gamma = \{t_1, \dots, t_k, \tau, t_{k+1}, \dots, t_n\}$ where
again we assume the times are in increasing order.  In this
case the embedding is
\[
f_1 \otimes \dots \otimes f_n
\longmapsto f_1 \otimes \dots \otimes f_k
\otimes \one \otimes f_{k+1} \otimes \dots \otimes f_n.
\]
It is easy to see that the family of embeddings under consideration form an inductive system, so that we may take the limit to obtain a
C$^*$-algebra $\Aa$ generated by copies of the $\AAA_\gamma$.

%We note in passing that the limit construction becomes even
%simpler when viewed through the lens of the Gelfand functor.
%Since $\AAA_\gamma$ may be identified with $C(S^\gamma)$, one
%can consider the inverse system of compact Hausdorff spaces
%$\{S^\gamma\}$ equipped with the canonical projections.  The
%projective (aka inverse) limit is the path space $\SSSS$,
%so applying the
%contravariant equivalence of categories, the inductive
%(aka direct) limit
%of the corresponding embeddings is isomorphic to $C(\SSSS)$.  While elegant, however, this point of view will be of little use in our noncommutative generalizations, since there is no underlying path space to work with.

Having constructed the limit algebra $\Aa$, with the embedding
$\AAA \hookrightarrow \Aa$ corresponding to the identification of $\AAA$ with $\AAA_{\{0\}}$, we are left with the task of constructing the retraction $\E: \Aa \to \AAA$.  We do this by first constructing a consistent family of retractions $\AAA_\gamma \to \AAA_\beta$
for $\beta \leq \gamma$, then showing how to use a limiting process
to induce the retraction $\Aa \to \AAA$.  First, we reduce as
before to the case where $\gamma$ contains one more point
than $\beta$, then retract
\[
f_1 \otimes \dots \otimes f_k \otimes g
\otimes f_{k+1} \otimes \dots \otimes f_n
\longmapsto f_1 \otimes \dots \otimes (f_k P_{\tau-t_k} g)
\otimes f_{k+1} \otimes \dots \otimes f_n.
\]
Note that in particular, when $\gamma$ contains 0 and
one identifies $\AAA$ with $\AAA_{\{0\}}$, repeated application
of this rule yields the retraction $\AAA_\gamma \to \AAA$ given
on simple tensors by
\[
f_1 \otimes \dots \otimes f_n
\longmapsto f_1 P_{t_2-t_1} \Big( f_2 P_{t_3-t_2}
\big(f_3 \cdots P_{t_n-t_{n-1}}(f_n)\big) \cdots \Big).
\]
Again, one can check that this family of retractions is consistent with the inductive system, so that it yields a well-defined and contractive map onto $\AAA$ from the dense subalgebra of $\Aa$ generated by the images of all the $\AAA_\gamma$; as this map is contractive, it
extends to a retraction on all of $\Aa$.

When seeking to carry this method across to the Sauvageot product,
one runs into several hurdles.  First, one does not form the Sauvageot product merely of two C$^*$-algebras, but rather of a CP-tuple; hence, one cannot begin by defining $\AAA_\gamma = \AAA \star \dots \star \AAA$
without specifying what maps are used between the various copies of $\AAA$.  Related, but more profound, is the failure of associativity; even when the relevant maps have been selected to make the notation well-defined, in general one does not have $(A \star A) \star (A \star A)$ isomorphic to $((A \star A) \star A) \star A$.
%AGENDUM: Come up with example, figure out where to put it (easiest seems after moment polynomials, so can refer to appendices)
Hence, we are led to adopt a more laborious inductive construction, though we follow the same high-level strategy as in the commutative case.

For the remainder of the chapter, we fix a unital C$^*$-algebra (resp. W$^*$-algebra) $\AAA$, a faithful (normal) state
$\omega$ on $\AAA$, and a cp$_0$-semigroup $\{\phi_t\}$ on $\AAA$.
We use $\FF$ to denote the set of finite subsets of $[0,\infty)$.
Throughout, we assume unless otherwise indicated that times within
such sets are listed in increasing order; hence, writing
$\gamma = \{t_1, \dots, t_n\}$ implies $t_1 < \dots < t_n$.

\section{Construction of the Inductive System and Limit}
\subsection{Objects and Immediate-Tail Morphisms}
\begin{definition} \label{deftail}
Let $\beta, \gamma \in \FF$ with
$\gamma = \{t_1,  \dots, t_n\}$.  We call $\beta$ an
\textbf{initial segment} of $\gamma$ if
$\beta = \{t_1, \dots, t_m\}$ for some $1 \leq m \leq n$,
and a \textbf{tail}
of $\gamma$ if $\beta = \{t_\ell, \dots t_n\}$
for some $1 \leq \ell \leq n$.
If $\ell = 2$ we call $\beta$ an \textbf{immediate tail}
with \textbf{distance} $t_2 - t_1$.
\end{definition}

We are now able to define the objects of our inductive system,
as well as some of the morphisms.

\begin{definition} \label{definductiveobjects}
For nonempty $\gamma \in \FF$ we define inductively
\begin{enumerate}
    \item a unital C$^*$-algebra (resp. W$^*$-algebra)
    $\AAA_\gamma$
    \item a unital embedding $\iota_\gamma: \AAA \to \AAA_\gamma$
    \item a retraction $\epsilon_\gamma: \AAA_\gamma \to \AAA$
\end{enumerate}
as follows:
\begin{itemize}
    \item If $\gamma$ is a singleton,
    then $\AAA_\gamma = \AAA$ and both $\iota_\gamma$
     and $\epsilon_\gamma$ are the identity.
    \item If $\beta$ is an immediate tail of $\gamma$
    with distance $\tau$, let $\Phi
    = \phi_\tau \circ \epsilon_\beta: \AAA_\beta \to \AAA$, and
    form the CP-tuple $(\AAA_\beta, \AAA, \Phi, \omega)$.
    Then $\AAA_\gamma$ is the Sauvageot product
    $\AAA_\beta \star \AAA$, $\iota_\gamma$ is the
    embedding of $\AAA$ into $\AAA_\beta \star \AAA$ (denoted
    $\dss{\psi}{R}$ in the previous chapter), and
    $\epsilon_\gamma$ is the Sauvageot retraction from
    $\AAA_\beta \star \AAA$ onto $\AAA$ (denoted $\theta$
    in the previous chapter).
\end{itemize}
We also define $\AAA_\emptyset = \com$.
\end{definition}
% AGENDUM: This embeds $\AAA$ into, say, $\AAA_{3,17}$ by identifying $\AAA$ with $\AAA_\{3\}$, not $\AAA_{3,17}$.  Go through and make sure this is consistent throughout.

Note that this definition also implicitly gives us
embeddings $\AAA_\beta \hookrightarrow \AAA_\gamma$ in the special
case where $\beta$ is an immediate tail of $\gamma$; this is
just the canonical embedding of $\AAA_\beta$ into $\AAA_\beta
\star \AAA$, the map denoted in the previous chapter
by $\dss{\psi}{L}$.

We turn
next to the question of how to embed $\AAA_\beta$ into
$\AAA_\gamma$ when $\beta \leq \gamma$ more generally.

\subsection{General Morphisms}
Consider now any inclusion $\beta \leq \gamma$
of nonempty elements of $\FF$.
Let $\gamma = \{t_1, \dots, t_n\}$ and for each $\ell \in \{1,
\dots, n\}$ define subsets $\gamma(\ell) \leq \gamma$
and $\beta(\ell) \leq \beta$ by
\[
\gamma(\ell) = \gamma \cap \{t_\ell, \dots, t_n\},
\quad \beta(\ell) = \beta \cap \{t_\ell, \dots, t_n\}.
\]
Then each $\gamma(\ell)$ is a tail of $\gamma$, with
$\gamma(1) = \gamma$, and similarly for $\beta$.
(Note that some of the $\beta(\ell)$ may be empty, if
$t_n \notin \beta$.)

\begin{definition} \label{definductivemorphisms}
For $\beta, \gamma$ as above, we define an embedding
$\AAA_\beta \sa{f} \AAA_\gamma$ by recursively
defining embeddings $\AAA_{\beta(\ell)} \sa{f_\ell} \AAA_{\gamma(\ell)}$
and letting $f = f_1$.  The embeddings are as follows:
\begin{itemize}
    \item In the base case $\ell = n$, the embedding $f_n$ is
    the identity map in case $t_n \in \beta$, or the canonical
    embedding $\com \hookrightarrow \AAA$ otherwise.
    \item Given $f_{\ell+1}$,
    %we define $f_{\ell-1}
%    = f_\ell \star \dss{\text{id}}{\AAA}$ in case
%    $t_{\ell-1} \in \beta$, or $f_{\ell-1} = i
%    \circ f_\ell$ in case $t_{\ell-1} \notin \beta$,
%    where $\AAA_{\gamma(n)} \sa{i} \AAA_{\gamma(n-1)}$
%    is the embedding given by Definition \ref{definductiveobjects}.
    let $\BB$ denote either $\AAA$ in the case that
    $t_\ell \in \beta$, or $\com$ otherwise;
    more succinctly, $\BB = \AAA_{\beta \cap \{t_\ell\}}$.  Let
    $\BB \sa{\psi} \AAA$ be either the identity map
    or the embedding of $\com$, accordingly.  Then
    \[
    f_{\ell} = f_{\ell+1} \star \psi.
    \]
\end{itemize}
\end{definition}

\begin{proposition} \label{propinductivesystem}
The family of embeddings $\AAA_\beta \hookrightarrow \AAA_\gamma$
in Definition \ref{definductivemorphisms} is an inductive system.
\end{proposition}

\begin{proof}
Let $\beta \leq \gamma \leq \delta$ be nonempty sets in $\FF$.
Write $\delta = \{t_1, \dots, t_n\}$.  We first prove
 that the embedding $\AAA_\delta \hookrightarrow \AAA_\delta$
 is the identity map.  We prove this for the embeddings
 $\AAA_{\delta(\ell)} \hookrightarrow \AAA_{\delta(\ell)}$
 by reverse induction; the base case $\ell=n$ is trivial,
 and the inductive step is just Corollary \ref{corunique}.

Now for each $\ell = 1, \dots, n$ let
\begin{align*}
\AAA_{\beta(\ell)} &\longsa{g_\ell} \AAA_{\gamma(\ell)}\\
\AAA_{\gamma(\ell)} &\longsa{f_\ell} \AAA_{\delta(\ell)}\\
\AAA_{\beta(\ell)} &\longsa{h_\ell} \AAA_{\delta(\ell)}
\end{align*}
be the embeddings from Definition \ref{definductivemorphisms}.
We will prove by reverse induction for $\ell = n, \dots, 1$
that $f_\ell \circ g_\ell = h_\ell$.  The base case $\ell=n$
is trivial, as each of the three maps
in question is either the identity map or the embedding
$\com \hookrightarrow \AAA$.  Supposing now the result
to be established for $\ell+1$, let $\BB = \AAA_{\beta \cap \{t_\ell\}}$
and $\CC = \AAA_{\gamma \cap \{t_\ell\}}$, and let
$\BB \sa{\psi} \CC \sa{\eta} \AAA$ be the corresponding
embeddings.  Then  by Proposition \ref{propfunctorial},
\[
    h_\ell = h_{\ell+1} \star (\eta \circ \psi)
    = (f_{\ell+1} \circ g_{\ell+1}) \star(\eta \circ \psi)
    = (f_{\ell+1} \star \eta)
    \circ (g_{\ell+1} \star \psi)
    = f_\ell \circ g_\ell.
    \]
\end{proof}

\section{Endomorphisms of the Limit Algebra}
We have constructed unital C$^*$-algebras (resp.\
W$^*$-algebras) $\AAA_\gamma$ for each $\gamma \in \FF$, together
with (normal) embeddings $\AAA_\beta \hookrightarrow \AAA_\gamma$
for $\beta \leq \gamma$, which we now denote $\dss{f}{\gamma, \beta}$,
satisfying the inductive properties
\begin{align*}
\dss{f}{\gamma, \gamma} &= \text{id}_{\AAA_\gamma}\\
\dss{f}{\delta, \beta} &= \dss{f}{\delta,\gamma} \circ
\dss{f}{\gamma, \beta} \qquad \text{ for }
\beta \leq \gamma \leq \delta.
\end{align*}
By a standard construction (see for instance section 1.23 of \cite{Sakai}, Proposition 11.4.1 of \cite{KadisonRingrose2},
or section II.8.2 of \cite{Blackadar}) we obtain an inductive
limit, that is, a unital C$^*$-algebra $\Aa$ and
embeddings $\dss{f}{\infty,\gamma}: \AAA_\gamma \to \Aa$
such that $\dss{f}{\infty,\gamma} \circ \dss{f}{\gamma, \beta}
= \dss{f}{\infty,\beta}$ for all $\beta \leq \gamma$, and
with the universal property that, given any other unital C$^*$-algebra
$\Bb$ and *-homomorphisms (not necessarily embeddings)  $\dss{g}{\infty,\gamma}: \AAA_\gamma \to \Bb$ satisfying $\dss{g}{\infty,\gamma}
\circ \dss{f}{\gamma, \beta} = \dss{g}{\infty,\beta}$, there
is a unique unital *-homomorphism $\Phi: \Aa \to \Bb$
satisfying $\dss{g}{\infty,\gamma} = \Phi \circ \dss{f}{\infty,\gamma}$
for all $\gamma$.  We denote by $i$ the distinguished embedding 
$f_{\infty, \{0\}}: A \to \Aa$.

We note that inductive limits do not always exist in the category
of W$^*$-algebras and normal *-homomorphisms; hence, $\Aa$
will not in general be a W$^*$-algebra even when $\AAA$ is.
We postpone until the next chapter the question of how to
adapt our construction to the W$^*$-category, and continue
for the time being with a purely C$^*$-construction.

Our next task is to define a semigroup of unital *-endomorphisms
of $\Aa$.  For this we note that for any $\gamma \in \FF$
and any $\tau \geq 0$, if $\gamma + \tau$ denotes
the set $\{t + \tau \mid t \in \gamma\}$, then
$\AAA_{\gamma+\tau} = \AAA_\gamma$.  (This
is an \emph{equality}, not just an isomorphism.)  This
is immediate from Definition \ref{definductiveobjects} by induction
on $|\gamma|$.  Similarly,
$\dss{f}{\gamma+t, \beta+t} = \dss{f}{\gamma, \beta}$.  But
this latter equation implies that $\dss{f}{\infty, \gamma+t}
\circ \dss{f}{\gamma, \beta} = \dss{f}{\infty, \beta+t}$ for
any $\beta \leq \gamma$, allowing us to make the following definition.

\begin{definition} \label{defsigmat}
For each $t \geq 0$ let $\sigma_t: \Aa \to \Aa$ denote the
unital *-endomorphism obtained through the inductive limit
as the unique map for which all the diagrams
\[ \xymatrix{
\AAA_\gamma \ar[rd]_{f_{\infty,\gamma+t}} \ar[r]^{f_{\infty,\gamma}} &\Aa \ar[d]^{\sigma_t}\\
& \Aa
} \]
commute.
\end{definition}

The universal property of the inductive limit then immediately
implies:

\begin{proposition} \label{prope0semigroup}
The maps $\{\sigma_t\}_{t \geq 0}$ form an e$_0$-semigroup on $\Aa$.
That is, $\sigma_0 = \dss{\text{id}}{\Aa}$, and for all
$s,t \geq 0$,
\[
\sigma_t \circ \sigma_s = \sigma_{s+t}.
\]
\end{proposition}

\section{The Limit Retraction}
We now turn to the construction of our retraction.
In the commutative analogue, for a set $\gamma$ with
minimum time $\tau$, the retraction $\epsilon_\gamma$
would (when composed with the embedding $\AAA \hookrightarrow
\Aa$) correspond to a conditional expectation onto the
subalgebra of $\Aa$ consisting of functions which depend only
on the location of a path at time $\tau$.  This does not
form a consistent system with respect to the embeddings
$\dss{f}{\gamma, \beta}$, because for $\beta \leq \gamma$
one could have times in $\gamma$ earlier than any in $\beta$.
However, the restriction to time sets which contain 0 is
consistent, which we now show in the noncommutative case.
We first consider how to relate the retraction for a given
set to the retractions for its tails.

\begin{lemma} \label{lemtailretractions}
Let $\gamma = \{t_1, \dots, t_n\} \in \FF$ and
$1 \leq \ell \leq n$.  Then
\[
\epsilon_\gamma \circ \dss{f}{\gamma, \gamma(\ell)}
= \phi_{t_\ell-t_1} \circ \epsilon_{\gamma(\ell)}.
\]
\end{lemma}

\begin{proof}
%PROOF STRATEGY:
%\begin{enumerate}
%    \item Prove by induction that $\phi_{t_\ell-t_1}
%    \circ \epsilon_{\gamma(\ell)} \circ \iota_{\gamma(\ell)}
%    = \epsilon_\gamma \circ \dss{f}{\gamma,\gamma(\ell)}
%    \circ \iota_{\gamma(\ell)}$.
%    \item Prove by induction that both
%    $\phi_{t_\ell-t_1} \circ \epsilon_{\gamma(\ell)}
%    \circ \dss{f}{\gamma(\ell,\gamma(\ell+1)} =
%    \epsilon_\gamma \circ \dss{f}{\gamma,\gamma(\ell)} \circ \dss{f}{\gamma(\ell),\gamma(\ell+1)}$ and
%    $\phi_{t_\ell-t_1} \circ \epsilon_{\gamma(\ell)}
%    \circ \iota_{\gamma(\ell)} =
%    \epsilon_\gamma \circ \dss{f}{\gamma,\gamma(\ell)} \circ
%    \iota_{\gamma(\ell)}$.
%    \item Combine these two, with the liberation properties
%    enjoyed by both, and conclude from Proposition \ref{propuniqueretraction} that
%    $\phi_{t_\ell-t_1} \circ \epsilon_{\gamma(\ell)}
%    = \epsilon_\gamma \circ \dss{f}{\gamma,\gamma(\ell)}$.
%\end{enumerate}

We proceed by (forward!) induction on $\ell$.
The base case $\ell=1$ is trivial.  Now supposing
the result is true for $\ell$, recall that $\AAA_{\gamma(\ell)}$
is the product $\AAA_{\gamma(\ell+1)} \star \AAA$
with respect to the map $\phi_{t_{\ell+1}-t_\ell}
\circ \epsilon_{\gamma(\ell+1)}: \AAA_{\gamma(\ell+1)} \to \AAA$, that $\dss{f}{\gamma(\ell),\gamma(\ell+1)}$ is the
embedding of $\AAA_{\gamma(\ell+1)}$ into this product, and that $\epsilon_{\gamma(\ell)}$ is the Sauvageot retraction.  By Equation \ref{eqnthetafactorsphi} we therefore have
\[
\epsilon_{\gamma(\ell)} \circ \dss{f}{\gamma(\ell),\gamma(\ell+1)}
= \phi_{t_{\ell+1}-t_\ell} \circ \epsilon_{\gamma(\ell+1)}
\]
so that
\begin{align*}
\epsilon_\gamma \circ \dss{f}{\gamma,\gamma(\ell+1)}
&= \epsilon_\gamma \circ \dss{f}{\gamma,\gamma(\ell)}
\circ \dss{f}{\gamma(\ell),\gamma(\ell+1)}\\
&= \phi_{t_\ell-t_1} \circ \epsilon_{\gamma(\ell)}
\circ \dss{f}{\gamma(\ell),\gamma(\ell+1)}\\
&= \phi_{t_\ell-t_1} \circ \phi_{t_{\ell+1}-t_\ell} \circ
\epsilon_{\gamma(\ell+1)}\\
&= \phi_{t_{\ell+1}-t_1} \circ \epsilon_{\gamma(\ell+1)}.
\end{align*}
\end{proof}

\begin{proposition} \label{propconsistentwith0}
Let $\beta \leq \gamma \in \FF$ such that the
minimum time in $\gamma$ is also in $\beta$.  Then
\[
\epsilon_\gamma \circ \dss{f}{\gamma,\beta} = \epsilon_\beta.
\]
\end{proposition}

%\begin{proposition} \label{propconsistentretractions}
%The retractions $\{\epsilon_\gamma\}$ are consistent; that is,
%for all $\beta \leq \gamma$,
%\[
%\epsilon_\gamma \circ \dss{f}{\gamma, \beta} = \epsilon_\beta.
%\]
%\end{proposition}
%
%\begin{proof}
%Fix $\gamma = \{t_1, \dots, t_n\}$ and $\beta \leq \gamma$.
%We prove the result inductively for $\beta(\ell) \leq \gamma(\ell)$
%for $\ell = n, \dots, 1$.  In the base case $\ell = n$, all
%three maps in question are either $\dss{\text{id}}{\AAA}$ or
%$\upsilon$, and the result is trivial.  Inductively, suppose
%$t_\ell \in \beta$.  Let $\Phi = \phi_{t_{\ell+1}-t_\ell}
%\circ \epsilon_{\beta(\ell+1)}: \AAA_{\beta(\ell+1)} \to \AAA$
%and $\Psi = \phi_{t_{\ell+1}-t_\ell}
%\circ \epsilon_{\gamma(\ell+1)}: \AAA_{\gamma(\ell+1)} \to \AAA$,
%so that $\AAA_{\beta(\ell)} = \AAA_{\beta(\ell+1)}
%\star_\Phi \AAA$ and $\AAA_{\gamma(\ell)} = \AAA_{\gamma(\ell+1)}
%\star_\Psi \AAA$.  Then
%\[
%\Psi \circ f_{\gamma(\ell+1), \beta(\ell+1)}
%= \phi_{t_{\ell+1}-t_\ell} \circ \epsilon_{\gamma(\ell+1)} \circ f_{\gamma(\ell+1), \beta(\ell+1)}
%= \phi_{t_{\ell+1} - t_\ell} \circ \epsilon_{\beta(\ell+1)}
%= \Psi.
%\]
%It follows by Corollary \ref{corcommutingsquare} that
%$\epsilon_{\beta(\ell)} = \epsilon_{\gamma(\ell)}
%\circ \dss{f}{\gamma(\ell),\beta(\ell)}$.
%\end{proof}

\begin{proof}
Let $\gamma = \{t_1, \dots, t_n\}$.
%We first note
%that, for all $\ell = 1, \dots, n-1$ we have
%\[
%\epsilon_{\gamma(\ell)} \circ \dss{f}{\gamma(\ell),\gamma(\ell+1)}
%= \phi_{t_{\ell+1} - t_\ell} \circ \epsilon_{\gamma(\ell+1)}
%\]
%by the definition of $\epsilon_{\gamma(\ell)}$ and
%Equation \ref{eqnthetafactorsphi}.
We will prove
that
\[
\epsilon_{\gamma(\ell)}
\circ \dss{f}{\gamma(\ell),\beta(\ell)} = \epsilon_{\beta(\ell)}
\]
for all $\ell$ such that $t_\ell \in \beta$.
For the base case with the maximal such $\ell$,
 $\dss{f}{\gamma(\ell), \beta(\ell)}$ is equal
to $\iota_{\gamma(\ell)}$, and since $\epsilon_{\gamma(\ell)}$
is a corresponding retraction, their composition is
$\dss{\text{id}}{\AAA} = \epsilon_{\beta(\ell)}$.
Inductively, suppose $t_\ell \in \beta$ and $t_{\ell+k}$ is
the next time in $\beta$, so that $\beta(\ell+1) =
\beta(\ell+k)$; then
\begin{align*}
%\epsilon_{\gamma(\ell)} \circ \dss{f}{\gamma(\ell), \gamma(\ell+1)} \circ \dss{f}{\gamma(\ell+1),\beta(\ell+1)} &=\phi_{t_{\ell+1}-t_\ell} \circ \epsilon_{\gamma(\ell+1)}
%\circ \dss{f}{\gamma(\ell+1),\beta(\ell+1)}\\
\phi_{t_{\ell+1}-t_\ell} \circ \epsilon_{\gamma(\ell+1)}
\circ \dss{f}{\gamma(\ell+1),\beta(\ell+1)} &=
\phi_{t_{\ell+1}-t_\ell} \circ \epsilon_{\gamma(\ell+1)}
\circ \dss{f}{\gamma(\ell+1),\beta(\ell+k)} \\
&=\phi_{t_{\ell+1}-t_\ell} \circ \epsilon_{\gamma(\ell+1)}
\circ \dss{f}{\gamma(\ell+1),\gamma(\ell+k)} \circ
\dss{f}{\gamma(\ell+k),\beta(\ell+k)} \\
&= \phi_{t_{\ell+1}-t_\ell} \circ \phi_{t_{\ell+k}-t_{\ell+1}}
\circ \epsilon_{\gamma(\ell+k)} \circ \dss{f}{\gamma(\ell+k),\beta(\ell+k)}\\
&= \phi_{t_{\ell+k}-t_\ell} \circ \epsilon_{\beta(\ell+k)}
\end{align*}
where the equalities follow respectively from the assumption that
$\beta(\ell+1) = \beta(\ell+k)$, the consistency of
the $f$'s, Lemma \ref{lemtailretractions}, and
induction.  It then follows from Corollary \ref{corcommutingsquare} that
\[
\epsilon_{\gamma(\ell)} \circ \dss{f}{\gamma(\ell),
\beta(\ell)} = \epsilon_{\gamma(\ell)} \circ
(\dss{f}{\gamma(\ell+1),\beta(\ell+1)} \star \text{id})
= \epsilon_{\beta(\ell)}
\]
as desired.  The case $\ell=1$ gives us the result.
\end{proof}

\begin{corollary} \label{corconsistentretractions}
The restriction of the family of retractions $\{\epsilon_\gamma\}$ to the subset $\FF_0 \subset \FF$ of sets containing 0
is consistent.
\end{corollary}

Since $\FF_0$ is a tail of $\FF$, the limit $\Aa$ is generated
by images of $\AAA_\gamma$ with $\gamma \in \FF_0$.
Hence, Corollary \ref{corconsistentretractions} implies
the existence of a retraction (with respect to $i$) \\
$\E: \Aa \to \AAA$
with the property that $\E \circ \dss{f}{\infty,\gamma}
= \epsilon_\gamma$ for all $\gamma \in \FF_0$.

\begin{definition} \label{defdilationretraction}
The \textbf{Sauvageot dilation retraction} for
$(\AAA, \{\phi_t\}, \omega)$ is the map
$\E: \Aa \to \AAA$ characterized by
\[
\E \circ \dss{f}{\infty, \gamma} = \epsilon_\gamma
\qquad \text{ for all } 0 \in \gamma \in \FF.
\]
\end{definition}

We now prove that $(\E, \{\sigma_t\})$ provides a strong dilation
of the semigroup $\{\phi_t\}$.

\begin{theorem} \label{thmstrongdilation}
For all $t \geq 0$,
\[
\E \circ \sigma_t = \phi_t \circ \E.
\]
\end{theorem}

\begin{proof}
The case $t = 0$ is trivial.  Now
let $\gamma \in \FF$ be nonempty and $t > 0$.  Let
$\delta = (\gamma+t) \cup \{0\}$; then $\AAA_\delta$ is
the Sauvageot product $\AAA_{\gamma+t} \star \AAA$ with
respect to the map $\phi_t \circ \epsilon_\gamma$.  By
Equation \ref{eqnthetafactorsphi}, it follows
that
\[
\epsilon_\delta \circ \dss{f}{\delta, \gamma+t}
= \phi_t \circ \epsilon_\gamma.
\]
Then
\begin{align*}
\E \circ \sigma_t \circ \dss{f}{\infty,\gamma}
&= \E \circ \dss{f}{\infty,\gamma+t}\\
&= \E \circ \dss{f}{\infty,\delta} \circ \dss{f}{\delta,\gamma+t}\\
&= \epsilon_\delta \circ \dss{f}{\delta,\gamma+t}\\
&= \phi_t \circ \epsilon_\gamma\\
&= \phi_t \circ \E \circ \dss{f}{\infty,\gamma}.
\end{align*}
So $\E \circ \sigma_t$ and $\phi_t \circ \E$ agree on
the dense subalgebra of $\Aa$ consisting of the images of
all the $\dss{f}{\infty,\gamma}$; as both are
contractive, they are equal.
\end{proof}

This concludes our construction of unital e$_0$-dilations
for cp$_0$-semigroups on C$^*$-algebras.  We summarize the
result in the following theorem.

\begin{theorem}
Let $\AAA$ be a unital C$^*$-algebra on which there exists
a faithful state.  Then every cp$_0$-semigroup on $\AAA$
has a strong unital e$_0$-dilation.
\end{theorem}

% Thesis Chapter 5 by Dave Gaebler
% Version 2: discovered that joint continuity of Sauvageot moment polynomials doesn't quite hold, and also got joint strong continuity of the semigroup via help from Orr.
% Version 3: Going back through just before thesis submission to clean up.  Decided to make a new version in case want to see pre-cleanup state.
\chapter{Continuous W$^*$-Dilations}\label{chapcontinuous}
In the previous chapter we saw how to construct a unital e$_0$-dilation of a cp$_0$-semigroup.  It remains to investigate whether such a construction dilates a continuous semigroup to a continuous semigroup (that is, whether it produces a unital E$_0$-dilation of a CP$_0$-semigroup), or, failing that, whether the construction can be modified to achieve this result.  Additionally, we have not yet resolved the question of how to adapt our C$^*$ construction to the W$^*$ setting.  To these issues we now turn our attention.

\section{Introduction: The Problem of Continuity}
The first question to consider is whether the existing dilation
may already be continuous.  It turns out that this is
\emph{never} the case unless $\AAA = \com$.  Consider a
nontrivial $\AAA$ with faithful state $\omega$, and
let $a$ be any nonzero element of $\ker \omega$.
Fixing some faithful representation $(H, \Omega, \dss{\pi}{R})$
of $(\AAA, \omega)$, let $h = \dss{\pi}{R}(a) \Omega$, which
is orthogonal to $\Omega$.  For each $t > 0$ there
is a faithful representation $(H, \Omega, \dss{\pi}{R},
K^{(t)}, V^{(t)}, \dss{\pi}{L}^{(t)})$ of $(A, A, \phi_t, \omega)$.
Form the Sauvageot product $\Hh^{(t)} = H^-
\star L^{(t)}$, and let $\xi$ be any unit vector in $L^{(t)}$.  By
Proposition \ref{prop3UEreps} we see that $\dss{\psi}{L}^{(t)}(a) \xi$ is a vector in $L^{(t)}$, whereas
$\dss{\psi}{R}^{(t)}(a) \xi = h \otimes \xi$ is
in $H^- \otimes L^{(t)}$.  Since these are orthogonal subspaces
of $\Hh^{(t)}$,
\[
\|\dss{\psi}{L}^{(t)}(a) \xi - \dss{\psi}{R}^{(t)}(a) \xi\|
\geq \|h \otimes \xi\| = \|h\| \|\xi\|
\]
which implies
\[
\|\dss{\psi}{L}^{(t)}(a) - \dss{\psi}{R}^{(t)}(a) \| \geq
\|h\|.
\]
Now letting $\gamma = \{0,t\}$, we have $\AAA_\gamma$ as the
Sauvageot product $\AAA \star \AAA$ with respect to $\phi_t$,
so that $\dss{\psi}{L}^{(t)}(a) - \dss{\psi}{R}^{(t)}(a)$
is the element $\dss{f}{\gamma, \{t\}}(a) -
\dss{f}{\gamma, \{0\}}(a)$ of $\AAA_\gamma$.  By
the above, this element has norm at least $\|h\|$.
Now because $\dss{f}{\infty, \gamma}$ is isometric,
\begin{align*}
\| \sigma_t(\iota(a)) - \iota(a) \|
&= \|\dss{f}{\infty, \{t\}}(a) - \dss{f}{\infty, \{0\}}(a)\|\\
&= \left\|\dss{f}{\infty,\gamma} \Big(
\dss{f}{\gamma, \{t\}}(a) - \dss{f}{\gamma, \{0\}}(a)\Big) \right\|\\
&= \|\dss{f}{\gamma, \{t\}}(a) - \dss{f}{\gamma, \{0\}}(a)\|
\geq \|h\|.
\end{align*}
It follows that $\|\sigma_t(\iota(a)) - \iota(a)\| \not \to 0$
as $t \to 0^+$.

Upon further reflection, the discontinuity of $\{\sigma_t\}$ is
not surprising, because it appears in the commutative dilation that the Sauvageot construction mimics.  Consider again the case
$\AAA = C(S)$, $\Aa = C(\mathscr{S})$ of Example
\ref{exmarkovdilation}. Given a regular Borel
probability measure $\mu_0$ on $S$, we obtain via Riesz representation
a regular Borel probability measure $\mu$ on $\SSSS$ characterized by
\[
\forall f \in \Aa: \qquad \int_\mathscr{S} f \, d\mu = \int_S (\E f) \, d\mu_0.
\]
Let us call the semigroup $\{\sigma_t\}$ ``point-pointwise'' continuous if for any fixed path $\pP \in \SSSS$ and any
$f \in \Aa$,
$(\sigma_t f - f)(\pP) \to 0$.  The failure of point-pointwise
continuity certainly implies the failure of point-norm continuity.  Now
let $\pP$ be any path not continuous at time 0, let
$\phi: S \to [0,1]$ be a continuous function such that
$\phi(\pP(t)) \not \to \phi(\pP(0))$ as $t \to 0^+$ (which
exists by Urysohn's lemma),
and let $f \in \Aa$ be defined by $f(p) = \phi(p(0))$.
Then
\[
\lim_{t \to 0^+} (\sigma_t f - f)(\pP)
= \lim_{t \to 0^+} \phi(\lambda_t \pP) - \phi(\pP)
= \lim_{t \to 0^+} \phi(\pP(t)) - \phi(\pP(0)) \neq 0.
\]

To remedy the problem, we move to the W$^*$ setting.  First, however, we need more information about the retraction constructed in the previous chapter.

\section{Moment Polynomials}
In the Sauvageot C$^*$-dilation of chapter \ref{chapiteratedproducts}, the inductive limit algebra $\Aa$
is norm-generated as an algebra by elements $\sigma_t(i(a))$
for $t \geq 0$ and $a \in \AAA$.  In studying the retraction
$\E$, therefore, one is naturally led to
consider expressions of the form
\begin{equation} \label{eqnkeyexpression}
\E \Big[ \sigma_{t_1}i((a_1)) \sigma_{t_2}(i(a_2))
\dots \sigma_{t_n}(i(a_n)) \Big], \quad
t_1, \dots, t_n \geq 0; \quad a_1, \dots, a_n \in \AAA.
\end{equation}
In particular, it would be desirable to have a formula
for the value of (\ref{eqnkeyexpression}) in terms of the
original semigroup $\{\phi_t\}$ and the state $\omega$
chosen for the dilation procedure.  From the construction
of $\E$ in previous chapters, we
see that (\ref{eqnkeyexpression}) can be evaluated as
follows:
\begin{enumerate}
    \item If all the $t_i$ are strictly positive, let
    $\tau$ denote the minimum; then, by Theorem
    \ref{thmstrongdilation},
    \[
    \E \Big[ \sigma_{t_1}i((a_1)) \sigma_{t_2}(i(a_2))
    \cdots \sigma_{t_n}(i(a_n)) \Big]
    = \phi_\tau \bigg( \E \big[ \sigma_{t_1-\tau}i((a_1)) \sigma_{t_2-\tau}(i(a_2))
\cdots \sigma_{t_n-\tau}(i(a_n)) \big] \bigg).
    \]

    \item If some of the $t_i$ are zero, let $\gamma = \{t_i\}$,
    $\gamma' = \gamma \setminus \{0\}$, and, disregarding the
    trivial case $\gamma = \{0\}$, $\tau = \min \gamma'$.
    By definitions
    \ref{defsigmat} and \ref{defdilationretraction},
    \begin{align*}
    \E \Big[ \sigma_{t_1}i((a_1)) \sigma_{t_2}(i(a_2))
    \cdots \sigma_{t_n}(i(a_n)) \Big] &=
    \E \Big[ f_{\infty, \{t_1\}} (a_1) \cdots f_{\infty, \{t_n\}}
    (a_n) \Big] \\
    &= \E \Big[ f_{\infty, \gamma} \big( f_{\gamma, \{t_1\}}(a_1)
    \cdots f_{\gamma, \{t_n\}}(a_n) \big) \Big]\\
    &= \epsilon_\gamma \Big[ f_{\gamma, \{t_1\}}(a_1)
    \cdots f_{\gamma, \{t_n\}}(a_n) \Big].
    \end{align*}
    Now $\AAA_\gamma$ is the Sauvageot product $\AAA_{\gamma'}
    \star \AAA$ with respect to the map $\phi_\tau \circ
    \epsilon_{\gamma'}$, so
    $f_{\gamma, \{t_1\}}(a_1) \cdots f_{\gamma, \{t_n\}}(a_n)$
    is a word in $\AAA_{\gamma'}$ (corresponding to nonzero $t_i$)
    and $\AAA$ (corresponding to those $t_i$ equal to zero); the
    value of $\epsilon_\gamma$ at this word can be computed using
    the moment function from chapter \ref{chapliberation}.
\end{enumerate}
In carrying out the second step, one
ends up applying the map $\epsilon_{\gamma'}$ to words
of the form $f_{\gamma', \{t_{i_1}\}} \cdots f_{\gamma', \{t_{i_m}\}}$,
yielding expressions similar to (\ref{eqnkeyexpression}).  It is therefore convenient, for both theoretical and practical purposes, to introduce a recursive definition for such expressions, rather than relying on appropriate evaluations of the moment function of chapter
\ref{chapliberation}.

We take as our basic object of study pairs of the form
$\lla t_1, \dots, t_n ; a_1, \dots, a_n \rra$,
where $t_1, \dots, t_n \geq 0$ and $a_1, \dots, a_n \in \AAA$.  The \textbf{length} of such
a pair is $n$.  This pair corresponds
to a word form $\sigma_{t_1}(i(a_1)) \cdots \sigma_{t_n}(i(a_n))$
in $\Aa$.
%The significance of $\gamma$ is that such a word
%may be viewed as an element of the copy of $\AAA_\gamma$ sitting
%inside $\Aa$.
We identify elements $a \in \AAA$ with pairs
$\lla 0; a \rra$.

We define a concatenation or ``multiplication'' operation on pairs
by
\begin{multline*}
\lla t_1, \dots, t_n ; a_1, \dots, a_n \rra
\vee \lla s_1, \dots, s_m ; b_1, \dots, b_m \rra\\
= \lla t_1, \dots, t_n, s_1, \dots, s_m ;
a_1, \dots, a_n, b_1, \dots, b_m \rra
\end{multline*}
or, more succinctly,
\[
\lla \vec{t} ; \vec{a} \rra
\vee \lla \vec{s} ; \vec{b}  \rra
= \lla \vec{t} \vee \vec{s} ; \vec{a} \vee \vec{b}
\rra.
\]

We make simultaneous recursive definitions of an $\AAA$-valued function $\Ss$
 on pairs (the \textbf{moment polynomial} function, corresponding to that defined in chapter 8 of \cite{ArvesonDynamics}), as well as a collapse function similar
 to that in chapter \ref{chapliberation},
as follows:

\begin{definition} \label{defmomentpolynomial}
 Given a pair $\lla \vec{t}; \vec{a} \rra$, let
    $\gamma = \cup \{t_i\}$.
\begin{itemize}
    \item If $\gamma = \{0\}$ then $\Ss \lla \vec{t}; \vec{a}; \gamma \rra = \Pi(\vec{a})$.
    \item If $0 < \tau = \min \gamma$,
    \[
    \Ss \lla \vec{t}; \vec{a} \rra
    = \phi_\tau \circ \Ss \lla \vec{t}-\tau; \vec{a} \rra.
    \]

    \item If $0 \in \gamma \neq \{0\}$, let $\tau
    = \min \gamma \setminus \{0\}$ and decompose
\[
\lla \vec{t}; \vec{a}\rra
= \lla \vec{n}_0; \vec{z}_0 \rra
\vee \lla \vec{s}_1; \vec{w}_1 \rra
\vee \lla \vec{n}_1; \vec{z}_1 \rra
\vee \cdots \vee \lla \vec{s}_\ell; \vec{w}_\ell \rra
\vee \lla \vec{n}_\ell; \vec{z}_\ell \rra
\]
where each $\vec{n}_i$ is a vector of zeros, each
$\vec{s}_i$ a vector of nonzero numbers, and
some of the pairs may be empty.  We refer to the pairs $\lla \vec{n}_i; \vec{z}_i \rra$ and $\lla \vec{s}_i; \vec{w}_i \rra$
as the \textbf{components} of this decomposition.   Given $S \subset [2\ell-1]$, let
\[
x_j = \begin{cases} \Ss \lla \vec{s}_{(j+1)/2};
 \vec{w}_{(j+1)/2} \rra & j \text{ odd }\\
 \Pi (\vec{z}_{j/2}) & j \text{ even} \end{cases}
\]
for $j \in S \cup \{0, 2\ell\}$ and
\[
y_k = \begin{cases} \lla \vec{s}_{(k+1)/2} ; \vec{w}_{(k+1)/2}\rra & k \text{ odd} \\
\lla \tau; \omega( \Pi(\vec{z}_{j/2}))\rra & k \text{ even}
\end{cases}
\]
for $k \in [2\ell-1] \setminus S$.  Let
$T_0, \dots, T_m$ denote the consecutive in-subsets
of $S$ and $U_1, \dots, U_m$ the consecutive out-subsets as in
chapter \ref{chapliberation}.  Then we define
\[
\text{Col}(\lla \vec{t}; \vec{a} \rra, S)
= \left(\bigvee_{j \in T_0} x_j \right)
\vee \left(\bigvee_{k \in U_1} y_k \right)
\vee \left(\bigvee_{j \in T_1} x_j \right)
\vee \cdots \vee \left(\bigvee_{j \in T_m} x_j \right)
\]
and
\[
\Ss\lla \vec{t}; \vec{a} \rra
= \sum_{\underset{S \neq 2 [\ell-1]}{S \subset [2\ell-1]} } (-1)^{\ell+|S|}
\Ss(\text{Col}(\lla \vec{t}; \vec{a} \rra, S)).
\]
\end{itemize}
\end{definition}

By the reasoning given above when introducing these pairs,
we arrive at the following:

\begin{proposition} \label{propmomentpolyasLM}
Let $\AAA$ be a unital C$^*$-algebra, $\{\phi_t\}$
a cp$_0$-semigroup on $\AAA$, $\omega$ a faithful
state on $\AAA$, and $(\Aa, i, \E, \{\sigma_t\})$
the Sauvageot dilation.  Then
for every $t_1, \dots, t_n \geq 0$
and $a_1, \dots, a_n \in \AAA$,
\[
\E \Big[ \sigma_{t_1}(i(a_1))
\cdots \sigma_{t_n}(i(a_n)) \Big]
= \Ss \lla\vec{t}; \vec{a} \rra.
\]
\end{proposition}

\section{Continuity Properties of Moment Polynomials}
The continuity properties of $\Ss(\vec{t}; \vec{a})$ in
the case where $\AAA$ is a W$^*$-algebra
will be important in what follows.
There are three types of continuity properties
to consider: continuity in $a_1, \dots, a_n$ with respect
to both the weak and the strong topologies, and
continuity in $t_1, \dots, t_n$.  It turns out
that weak continuity holds with respect
to $a_1, \dots, a_n$ separately (which is the best
we could hope for, as multiplication is not jointly
weakly continuous), whereas
strong continuity holds jointly in $a_1, \dots, a_n$,
and a restricted form of joint continuity
in $t_1, \dots, t_n$ holds as well.

\begin{proposition}
Let $\AAA$ be a W$^*$-algebra, $\{\phi_t\}$ a CP$_0$-semigroup on $\AAA$, $\omega$ a faithful normal state on $\AAA$.
Fix $n \geq 1$, $t_1, \dots, t_n \geq 0$, $j \in \{1, \dots,n\}$, and $a_k$ for $k \in \{1, \dots, n\} \setminus \{j\}$.  Then $\Ss \lla \vec{t}; \vec{a} \rra$, viewed as a
function of $a_j$, is a normal linear map from $\AAA$
to itself.
\end{proposition}

\begin{proof}
This follows from Definition
(\ref{defmomentpolynomial}) by induction on the length of the pair.  We also use
the normality of the state $\omega$ and the maps
$\phi_t$, as well as the normality of multiplication by
a fixed element of $\AAA$.
\end{proof}

\begin{definition} \label{defnoncrossingly}
For $n \geq 1$ and elements $\{\vec{s}_k\}$ and
$\vec{t}$ of $[0,\infty)^n$, we say that $\vec{s}_k$
\textbf{converges non-crossingly to} $\vec{t}$ if
$\vec{s}_k \to \vec{t}$ and, for all $k$, the
order relations among the entries of $\vec{s}_k$ are the
same as those in $\vec{t}$; that is, if
\[
\forall k: \ \forall i,j = 1, \dots, n: \
(s_k)_i \leq (s_k)_j \LRA t_i \leq t_j.
\]
\end{definition}

\begin{proposition} \label{propjointcontinuitymoments}
Let $\AAA$ be a separable W$^*$-algebra, $\{\phi_t\}$
a CP$_0$-semigroup on $\AAA$, $\omega$ a faithful
normal state on $\AAA$.  Let $n \geq 1$.  Let
$\vec{t}_k \to \vec{t}$ converge non-crossingly in $[0,\infty)^n$,
and let $\vec{a}_k \to \vec{a}$ be a strongly convergent
sequence of tuples in $(\AAA_1)^n$.  Then
$\Ss \lla \vec{t}_k; \vec{a}_k \rra \to \Ss \lla \vec{t}; \vec{a} \rra$
strongly.  That is, $\Ss\lla \vec{t}; \vec{a}\rra $ is jointly
strongly continuous in $\vec{t}$ and $\vec{a}$, subject
to the non-crossing restriction on $\vec{t}$.
\end{proposition}

\begin{proof}
It is convenient to write each $\vec{t}_i = \delta_i + \vec{u}_i$
and $\vec{t} = \delta + \vec{u}$, where $\delta_i, \delta \geq 0$
and each $\vec{u}_i$ and $\vec{u}$ contains at least one zero entry. We make the following observations:
\begin{enumerate}
    \item $\delta_i \to \delta$
    \item $\vec{u}_i \to \vec{u}$ non-crossingly
    \item $\Ss \lla \vec{t}_i; \vec{a}_i \rra
    = \phi_{\delta_i} \circ \Ss \lla \vec{u}_i; \vec{a}_i \rra$.
\end{enumerate}
By the joint strong continuity of the semigroup $\phi$  (Theorem \ref{thmjointcontinuityCP}), it therefore suffices to consider the case where all $\vec{t}_i$ and $\vec{t}$ have zero entries.  The non-crossing hypothesis implies that these occur at the same positions for all $i$, and hence that the components of the decompositions of all $\vec{t}_i$ in the definition of $\Ss$ all have constant length as $i$ varies.  The time vectors in each component pair converge non-crossingly to the time vector in the corresponding component of $\vec{t}$, and the result follows by induction on $n$.
\end{proof}

To illustrate the necessity of the non-crossing hypothesis, Appendix \ref{chapmomentpolynomials} contains an example where this hypothesis is violated and where discontinuity results.  The underlying reason is that a crossing creates a change in the lengths of the components of the decomposition, so that a different branch of the recursion is followed.

\section{The Continuous Theorem}
We now return to the question of how to obtain a
continuous W$^*$-dilation from an algebraic C$^*$-dilation.
The technique in this section is adapted
from the eighth chapter of \cite{ArvesonDynamics}.
Throughout, we let $\AAA$ denote a separable W$^*$-algebra,
$\{\phi_t\}$ a CP$_0$-semigroup
on $\AAA$, $(\Aa, i, \E, \{\sigma_t\})$ the Sauvageot
dilation from the previous chapter,
$\PP \subset \Aa$ the subset
\[
\PP = \{\sigma_{t_1}(i(a_1)) \dots \sigma_{t_k}(i(a_k))
\mid t_1, \dots, t_k \geq 0; \, a_1, \dots, a_k \in \AAA\},
\]
$\Aa_0 \subseteq \Aa$ the norm-dense linear span of $\PP$,
$(H, \pi)$ a faithful
normal representation of $\AAA$ on a separable Hilbert
space, $(\Hh, V, \psi)$ a minimal Stinespring
dilation of $\pi \circ \E$, $\widetilde{\Aa} = \psi(\Aa)''$,
and $\widetilde{\E}: \widetilde{\Aa} \to \AAA$ the
map $\widetilde{\E}[T] = \pi\inv(V^*T V)$, which
is well-defined because $T \mapsto V^* T V$ is normal and maps
the weakly dense subspace $\psi(\Aa) \subset \widetilde{\Aa}$ into the weakly closed set $\pi(\AAA)$, and
because $\pi$ is faithful; it satisfies
$\widetilde{\E} \circ \psi = \E$ and therefore is
a normal retraction with respect to $\psi \circ i$.

We begin with the observation that weak-operator continuity
of families of contractions can be checked
on a dense subset of Hilbert space.

\begin{lemma} \label{lemWOTcontinuitydense}
Let $\HH$ be a Hilbert space and $\{T_t\}_{t \geq 0}$
a family (not necessarily a semigroup) of contractions on $\HH$.  Let $\HH_0 \subseteq \HH$ be a dense linear
subspace such that for all $\xi, \eta \in \HH_0$,
the map $t \mapsto \la T_t \xi, \eta \ra$ is
continuous.  Then $t \mapsto T_t$ is WOT-continuous.
\end{lemma}

\begin{proof}
Let $\xi, \eta \in \HH$ and $t_0 \geq 0$.  Given
$\epsilon > 0$, choose $\xi_0, \eta_0 \in \HH_0$
with $\|\xi - \xi_0\| < \max(1,\epsilon)$ and
$\|\eta - \eta_0\| < \epsilon$.  Then
for any $t\geq 0$,
\begin{align*}
\la (T_t - T_{t_0}) \xi, \eta \ra &= \la (T_t - T_{t_0})
\xi_0, \eta_0 \ra \\
&+ \la (T_t - T_{t_0}) \xi_0, \eta - \eta_0 \ra\\
&+ \la (T_t - T_{t_0}) (\xi - \xi_0), \eta \ra.
\end{align*}
The first term tends to zero as $t \to t_0$ by hypothesis, so
that in particular it is less than $\epsilon$
for $t$ sufficiently close to $t_0$.  The second term
is at most $2 \|\xi_0\| \epsilon \leq
2 (\|\xi\|+1) \epsilon$ by Cauchy-Schwarz, and
the third term at most $2 \|\eta\| \epsilon$.
Hence
\[
|\la (T_t - T_{t_0}) \xi, \eta \ra|
\leq  \Big(3 + 2\|\eta\| + 2 \|\xi\| \Big) \epsilon
\]
for $t$ sufficiently near $t_0$.
\end{proof}

The next lemma is rather technical, but it
advances our study of how $\E$ interacts with time
translations, and in particular with translation
of the middle term of a threefold product.  It is based
on Lemma 8.6.3 in \cite{ArvesonDynamics}.

\begin{lemma} \label{lemexistsQ}
Let $y,z \in \PP$ and $t \geq 0$.
There exist $y_0,z_0 \in \PP$ and a normal
linear map $Q: \AAA \to \AAA$ such
that, for every $x \in \Aa$,
\begin{equation} \label{eqnQmap}
\E \big[ y \sigma_t(x) z \big] = Q \Big(
\E \big[ y_0 x z_0 \big] \Big).
\end{equation}
\end{lemma}

\begin{proof}
Let $y = \sigma_{v_1}(a_1) \cdots \sigma_{v_p}(a_p)$
and $z = \sigma_{t_1}(b_1) \cdots \sigma_{t_n}(b_n)$.  We
proceed by strong induction on $m+n$.  In the base case $m+n=0$
(meaning that $y = z = \one$) the requisite map is
$Q = \phi_t$, by Theorem \ref{thmstrongdilation}.
Inductively, letting
$\tau = \min(v_1, \dots, v_p, t_1, \dots, t_n)$, the result
is again trivial in case $t \leq \tau$, as then
one can use $Q = \phi_t$,
$y_0 = \sigma_{v_1-t}(a_1)
\cdots \sigma_{v_p-t}(a_p)$, and $z_0 = \sigma_{t_1-t}(b_1) \cdots \sigma_{t_n-t}(b_n)$.  Hence we assume $t > \tau$.
We further assume $\tau = 0$, as the case $\tau > 0$ reduces to
this by Theorem \ref{thmstrongdilation} again.

Let $(v_1', \dots, v_q')$ be the (possibly empty) final segment of nonzero entries from $(v_1, \dots, v_q)$,
and $(a_1', \dots, a_q')$ the corresponding entries
from $(a_1, \dots ,a_p)$.  Similarly, let $(t_1', \dots,
t_h')$ be the initial segment of nonzero entries
from $(t_1, \dots, t_n)$, and $(b_1', \dots, b_h')$
the corresponding entries from $(b_1, \dots, b_n)$.
Let $y_0 = \sigma_{v_1'}(a_1') \dots \sigma_{v_q'}(a_q')$
and $z_0 = \sigma_{t_1'}(b_1') \dots \sigma_{t_h'}(b_h')$.

For any $x \in \PP$, write $x = \sigma_{u_1}(c_1) \cdots \sigma_{u_m}(c_m)$, so that $\sigma_t(x)
= \sigma_{u_1+t}(c_1) \cdots \sigma_{u_m+t}(c_m)$.  Now $\E[y \sigma_t(x) z]
= \Ss \lla \vec{v} \vee (\vec{u}+t) \vee \vec{t};
\vec{a} \vee \vec{c} \vee \vec{b} \rra$ by Proposition \ref{propmomentpolyasLM}.
In the standard decomposition $\vec{v} \vee (\vec{u}+t)
\vee \vec{t}
= \vec{n}_0 \vee \vec{s}_1
\vee \dots \vee \vec{n}_{\ell}$, we must have $\vec{u}+t$
contained in a single one of the $\vec{s}_i$;
more specifically, for some $i$ we have
$\vec{s}_i = (v_1', \dots, v_q')
\vee (\vec{u}+t) \vee (t_1', \dots, t_h')$ and
$\vec{w}_i = (a_1', \dots, a_q') \vee
\vec{c} \vee (b_1', \dots, b_h')$.
Then Propositions \ref{propmomentpolyasLM} and \ref{propnormalmoments} imply that $\E[y \sigma_t(x) z]$
is the composition of $\E[y_0 x z_0]$ with some normal map $Q$, which is independent of $x$.  This gives us equation (\ref{eqnQmap})
for all $x \in \PP$, and since both sides are linear and
norm-continuous in $x$, it follows that (\ref{eqnQmap})
holds for all $x \in \Aa$.
\end{proof}

\begin{theorem} \label{thmexistsnormalextension}
There exists a (necessarily unique) semigroup of normal unital *-endomorphisms $\{\widetilde{\sigma}_t\}_{t \geq 0}$ of $\widetilde{\Aa}$
such that
\begin{equation} \label{eqncovarianceofextension}
\forall t \geq 0: \qquad \widetilde{\sigma}_t \circ \psi = \psi \circ \sigma_t.
\end{equation}
\end{theorem}

\begin{proof}
We construct $\{\widetilde{\sigma}_t\}$ and verify
its properties in many small steps.
\begin{enumerate}
    \item For each $t \geq 0$ and $\xi, \eta \in \psi(\PP) V H$,
    we construct a normal linear functional $\rho_{t, \xi, \eta}$
    on $\widetilde{\Aa}$ as  follows.
    Let $\xi = \psi(y) V \xi'$ and $\eta = \psi(z) V \eta'$
    for $y,z \in \PP$ and $\xi', \eta' \in H$.
    %For
%    every $x \in \Aa$ one has
%    \[
%    \la \psi(\sigma_t(x)) \xi, \eta \ra = \la V^* \psi(
%    z^* \sigma_t(x) y) V \xi', \eta' \ra =
%    \la \pi \circ \E[z^* \sigma_t(x) y] \xi', \eta' \ra.
%    \]
    By Lemma (\ref{lemexistsQ}), there exists a normal
    linear map $Q: \AAA \to \AAA$ and elements $y_0, z_0 \in \PP$
    such that $\E[z^* \sigma_t(x) y]
    = Q(\E[z_0^* x y_0])$ for all $x \in \Aa$.  We thus have
\[
\forall x \in \Aa: \qquad \la \psi(\sigma_t(x)) \xi, \eta \ra =
\la \pi \circ Q \circ \E[z_0^* x y_0] \xi', \eta' \ra.
\]
We now define $\rho_{t, \xi, \eta}$ by
\[
\rho_{t, \xi, \eta}(T) = \la \pi \circ Q \circ \widetilde{\E}
[\psi(z_0)^* T \psi(y_0)] \xi', \eta'\ra, \qquad T \in \widetilde{\Aa}.
\]
Then the restriction to $\psi(\Aa)$ satisfies
\begin{align}
\forall x \in \Aa: \qquad \rho_{t, \xi, \eta}(\psi(x))
&= \la \pi \circ Q \circ \widetilde{\E} \circ
\psi(z_0^* x y_0) \xi', \eta' \ra \nonumber\\
&= \la \pi \circ Q \circ \E[z_0^* x y_0] \xi', \eta' \ra \nonumber\\
&= \la \pi \circ \E[z^* \sigma_t(x) y] \xi', \eta' \ra \nonumber\\
&= \la V^* \psi(z^* \sigma_t(x) y) V \xi', \eta' \ra \nonumber\\
&= \la \psi(\sigma_t(x)) \xi, \eta \ra. \label{eqnrhofunctional}
\end{align}

    \item We extend the definition to $\xi, \eta$
    in the linear span of $\psi(\PP) V H$ in the natural way; for
$\xi = \sum_i c_i \xi_i$ and $\eta = \sum_j d_j \eta_j$ with
$\xi_i, \eta_j \in \psi(\PP) VH$, we define
$\rho_{t, \xi, \eta} = \sum_{i,j} c_i \overline{d}_j \rho_{t, \xi_i, \eta_j}$.  This is well-defined because, if
$\sum_i c_i \xi_i = \sum_k \tilde{c}_k \tilde{\xi}_k$
and $\sum_j d_j \eta_j = \sum_\ell \tilde{d}_\ell \tilde{\eta}_\ell$
then equation (\ref{eqnrhofunctional}) implies that, for $x$ in
the ultraweakly dense subspace $\psi(\Aa)$ of $\widetilde{\Aa}$,
\begin{align*}
\dss{\rho}{t, \sum c_i \xi_i, \sum d_j \eta_j}
(\psi(x)) &= \left\la \psi(\sigma_t(x))\sum c_i \xi_i,
\sum d_j \eta_j \right\ra\\
&= \left\la \psi(\sigma_t(x)) \sum \tilde{c}_k \tilde{\xi}_k,
\sum \tilde{d}_\ell \tilde{\eta}_\ell \right\ra
= \dss{\rho}{t, \sum \tilde{c}_k \tilde{\xi}_k,
\sum \tilde{d}_\ell \tilde{\eta}_\ell}(\psi(x)).
\end{align*}

    \item Next, we note that equation (\ref{eqnrhofunctional}) also
    implies that $\|\rho_{t, \xi, \eta}\| \leq \|\xi\| \|\eta\|$.
    This allows us to extend the definition to $\xi, \eta$ in
    the norm closure of the linear span of $\psi(\PP) VH$,
    which is all of $\Hh$.

    \item Having defined the family of functionals $\{\rho_{t, \xi, \eta}\}$, we now use them to define the family
        of endomorphisms $\{\widetilde{\sigma}_t\}$.
        Equation (\ref{eqnrhofunctional})
      implies that, for fixed $t \geq 0$ and $x \in \Aa$, $\rho_{\xi,\eta}(\psi(x))$
        is a bounded sesquilinear function of $\xi$ and $\eta$, so that it
corresponds to a unique operator in $B(\Hh)$, which we call
$S_t(\psi(x))$, characterized by the property
\begin{equation} \label{eqndefSt}
\forall \xi, \eta \in \Hh: \quad
\rho_{t,\xi, \eta}(\psi(x)) = \la S_t(\psi(x))
\xi, \eta \ra.
\end{equation}

    \item Equations (\ref{eqnrhofunctional}) and (\ref{eqndefSt})
together imply that
\begin{equation} \label{eqnStcovariance}
\forall x \in \Aa: \quad S_t(\psi(x)) = \psi(\sigma_t(x)).
\end{equation}

    \item Because $\psi$ and $\sigma_t$
    are unital *-homomorphisms, equation (\ref{eqnStcovariance})
    implies that $S_t$ is as well.

    \item Because $S_t$ is a unital *-homomorphism of a C$^*$-algebra, it is contractive.  This implies
    \begin{equation} \label{eqncontractivity}
    \forall x \in \Aa: \quad \|\psi(\sigma_t(x))\|
    \leq \|\psi(x)\|.
    \end{equation}

    \item Given any $z \in \widetilde{\Aa}$, we can now
    show that $\rho_{t,\xi, \eta}(z)$ is a bounded sesquilinear
    function of $\xi$ and $\eta$.  For boundedness, we
    will show more precisely that
    \begin{equation} \label{eqnboundedforfixedz}
    |\rho_{t,\xi, \eta}(z)| \leq \|z\| \|\rho\|
    \|\eta\|.
    \end{equation}
    Let $z, \xi, \eta$ be given, and choose
    $\epsilon > 0$.  By the Kaplansky density theorem and
    the normality of $\rho_{t,\xi, \eta}$, there exists
    $x \in \Aa$ such that $\|\psi(x)\| \leq \|z\|$ and
    $|\rho_{t,\xi, \eta}(z - \psi(x))|  < \epsilon$.  Then
    \begin{align*}
    |\rho_{t,\xi, \eta}(z)| &\leq
    |\rho_{t,\xi, \eta}(\psi(x))| + |\rho_{\xi, \eta}(z - \psi(x))|\\
    &\leq \epsilon + |\la \psi(\sigma_t(x)) \xi, \eta \ra|\\
    &\leq \epsilon + \|\psi(\sigma_t(x))\| \|\xi\| \|\eta\|\\
    &\leq \epsilon + \|\psi(x)\| \|\xi\| \|\eta\| \\
    &\leq \epsilon + \|z\| \|\xi\| \|\eta\|.
    \end{align*}
    %where we have used equation (\ref{eqncontractivity})
%    in the penultimate step.
    Letting $\epsilon \to 0$,
    we have \ref{eqnboundedforfixedz}.

    To show linearity in $\xi$, let $c_1, c_2 \in \com$
    and $\xi_1, \xi_2, \eta \in \Hh$ be given, and choose
    $\epsilon > 0$.  By Kaplansky density and the normality
    of $\rho_{t, \xi_1, \eta}$, $\rho_{t, \xi_2, \eta}$,
    and $\rho_{t, c_1 \xi_1 + c_2 \xi_2, \eta}$, there
    exists $x \in \Aa$ such that $\|\psi(x)\| \leq \|z\|$
    and the three inequalities
    \begin{align*}
    \left| \rho_{t,c_1 \xi_1 + c_2 \xi_2,\eta}(z - \psi(x)) \right|
    &< \epsilon \\
    |c_1| \left|\rho_{t,\xi_1, \eta}(z - \psi(x)) \right| &<
    \epsilon \\
    |c_2| \left| \rho_{t, \xi_2, \eta}(z - \psi(x))\right|
    &< \epsilon
    \end{align*}
    all hold.  Then
    \begin{align*}
    &\left|\rho_{t, c_1 \xi_1 + c_2 \xi_2, \eta}(z)
    - c_1 \rho_{t, \xi_1, \eta}(z) - c_2 \rho_{t, \xi_2,
    \eta}(z) \right| \\
    \leq &\left|\rho_{t, c_1 \xi_1 + c_2 \xi_2, \eta}(z - \psi(x))\right|
    + |c_1| \left|\rho_{t, \xi_1, \eta}(z - \psi(x))\right|
    + |c_2| \left|\rho_{t, \xi_2, \eta}(z - \psi(x))\right|\\
    &+ \left|\rho_{t, c_1 \xi_1 + c_2 \xi_2, \eta}(\psi(x))
    - c_1 \rho_{t, \xi_1, \eta}(\psi(x)) - c_2 \rho_{t, \xi_2,
    \eta}(\psi(x)) \right| \\
    \leq &3\epsilon +
    \left| \la \psi(\sigma_t(x)) (c_1 \xi_1 + c_2 \xi_2),
    \eta \ra - c_1 \la \psi(\sigma_t(x)) \xi_1, \eta \ra
    - c_2 \la \psi(\sigma_t(x)) \xi_2, \eta \ra \right|
    = &3\epsilon
    \end{align*}
    and as this is true for all $\epsilon > 0$,
    we conclude that
    \[
    \rho_{t, c_1 \xi_1 + c_2 \xi_2, \eta}(z)
    = c_1 \rho_{t, \xi_1, \eta}(z) +
    c_2 \rho_{t, \xi_2, \eta}(z).
    \]
    Conjugate-linearity in $\eta$ is, of course, established
    in the same way.

    \item We therefore obtain an operator in $B(\Hh)$,
    which we call $\widetilde{\sigma}_t(z)$, characterized by
    the property
    \begin{equation} \label{eqndefsigma}
    \forall \xi, \eta \in \Hh: \quad
    \rho_{t,\xi, \eta}(z) = \la \widetilde{\sigma}_t(z) \xi, \eta \ra.
    \end{equation}
    We now have a function (not yet known to be linear,
    continuous, multiplicative, or self-adjoint)
    $\widetilde{\sigma}_t: \widetilde{\Aa}
    \to B(\Hh)$ which extends the unital *-endomorphism
    $S_t: \psi(\Aa) \to \psi(\Aa)$.

    \item The function $\widetilde{\sigma}_t$ is contractive, because
    \[
    \|\widetilde{\sigma}_t(z)\| = \sup_{\xi, \eta \in \Hh_1}
    |\la \widetilde{\sigma}_t(z) \xi, \eta \ra |
    = \sup_{\xi, \eta \in \Hh_1} |\rho_{t, \xi, \eta}(z)|
    \leq \sup_{\xi, \eta \in \Hh_1} \|z\| \|\xi\| \|\eta\|
    = \|z\|
    \]
    by equation (\ref{eqnboundedforfixedz}).

    \item Weak continuity of
    $\widetilde{\sigma}_t$ is a
     straightforward consequence of the normality of the
    $\rho_{t, \xi, \eta}$.  Indeed, if $z_\nu \to z$
    weakly in the unit ball $\widetilde{\Aa}_1$, then for all
    $\xi, \eta$ it follows that
    \[
    \la \widetilde{\sigma}_t(z_\nu) \xi, \eta \ra
    = \rho_{t, \xi, \eta}(z_\nu) \to \rho_{t, \xi, \eta}(z)
    = \la \widetilde{\sigma}_t(z) \xi, \eta \ra
    \]
    so that $\widetilde{\sigma}_t(z_\nu) \to \widetilde{\sigma}_t(z)$
    in the weak operator topology, which agrees
    with the weak topology of $\widetilde{\Aa}$
    on bounded subsets.

    \item Since $\widetilde{\sigma}_t$ maps the unit ball
    of $\psi(\Aa)$ into $\psi(\Aa)$, it follows
    from the previous step and the Kaplansky density
    theorem that it maps the unit ball of $\widetilde{\Aa}$
    into $\widetilde{\Aa}$.  Hence $\widetilde{\sigma}_t$,
    initially defined as a map from $\widetilde{\Aa}$
    into $B(\Hh)$, is actually a self-map of $\widetilde{\Aa}$.

    \item Next, we prove that $\widetilde{\sigma}_t$ is
    a *-endomorphism of $\widetilde{\Aa}$.  Let $x, y \in \widetilde{\Aa}$ and choose bounded
    nets $\{x_\nu\}, \{y_\nu\} \subset \Aa$ with
    $\psi(x_\nu) \to x$ and $\psi(y_\nu) \to y$ weakly.
    By the weak continuity of $\widetilde{\sigma}_t$ and its multiplicativity  on $\psi(\Aa)$,
    \begin{align*}
    \widetilde{\sigma}_t(x y) &= \widetilde{\sigma}_t \left(\lim_\mu \lim_\nu
    \psi(x_\nu) \psi(y_\mu) \right)\\
    &= \lim_\mu \lim_\nu \widetilde{\sigma}_t \Big(\psi(x_\nu) \psi(y_\mu) \Big)\\
    &= \lim_\mu \lim_\nu \widetilde{\sigma}_t \Big(\psi(x_\nu) \Big)
    \widetilde{\sigma}_t \Big(\psi(y_\mu) \Big) \\
    &= \widetilde{\sigma}_t(x) \widetilde{\sigma}_t(y).
    \end{align*}
    Linearity and self-adjointness are proved similarly.

    \item Finally, it is clear that
    $\widetilde{\sigma}_0 = \dss{\text{id}}{\AAA}$, and for all $s,t \geq 0$
    and all $x \in \Aa$,
    \[
    \widetilde{\sigma}_{s+t}(\psi(x))
    = \psi(\sigma_{s+t}(x))
    = \psi(\sigma_s (\sigma_t(x)))
    = \widetilde{\sigma}_s(\psi(\sigma_t(x)))
    = \widetilde{\sigma}_s(\widetilde{\sigma}_t(\psi(x)))
    \]
    so that $\widetilde{\sigma}_{s+t}$ and $\widetilde{\sigma}_s
    \circ \widetilde{\sigma}_t$ agree on the ultraweakly
    dense subset $\psi(\Aa) \subset \widetilde{\Aa}$;
    as both are normal, they are equal.
%This then implies that $\rho_{t, \xi, \eta}$ varies
%boundedly and sesquilinearly in $\xi, \eta$ for fixed $t$ and for
%$\xi, \eta$ drawn from the norm-dense subspace $\psi(\PP)
%V H$ of $\Hh$.  For each $y \in \tilde{\Aa}$, the Riesz
% representation theorem therefore produces a unique bounded
%linear map $\widetilde{\sigma}_t(y) \in B(\Hh)$ satisfying
%$\la \widetilde{\sigma}_t(y) \xi,\eta \ra = \rho_{t,\xi,\eta}(y)$. Moreover, for $x \in \AAA$ we also have $\widetilde{\sigma}_t(\psi(x)) = \psi(\sigma_t(x))$.  Note that
%this implies $\widetilde{\sigma}_t(y) \in \tilde{\Aa}$
%for $y \in \psi(\Aa)$, which is a weakly dense subset
%of $\tilde{\Aa}$, so that by normality we have
%$\widetilde{\sigma}_t(y) \in \tilde{\Aa}$.  It is also clear from the construction that $\widetilde{\sigma}_t(y)$ varies linearly
%in $y$.  Further, $\widetilde{\sigma}_t$ is multiplicative
%on the subset $\psi(\Aa) \subset \tilde{\Aa}$, so
%that (by Kaplansky density, and the joint ultraweak continuity of multiplication on bounded subsets of W$^*$-algebras)
%it is multiplicative on all of $\tilde{\Aa}$.  The self-adjointness of $\widetilde{\sigma}_t$ is also clear.  Finally, for $x \in \Aa$ and $s,t \geq 0$ we have
%\[
%\widetilde{\sigma}_{s+t}(\psi(x)) =
%\psi(\sigma_{s+t}(x)) =
%\psi(\sigma_s(\sigma_t(x))) =
%\widetilde{\sigma}_s(\psi(\sigma_t(x)))=
%\widetilde{\sigma}_s(\widetilde{\sigma}_t(\psi(x)))
%\]
%so that $\widetilde{\sigma}_{s+t}$ and $\widetilde{\sigma}_s \circ
%\widetilde{\sigma}_t$ agree on the dense subset $\psi(\Aa)
%\subset \tilde{\Aa}$ and hence are equal.
\end{enumerate}
\end{proof}

As one corollary, we can now find many dense
subspaces of $\Hh$.  Recall that $\psi(\PP) VH$ is
dense by the standard properties of the minimal Stinespring
dilation plus the fact that $\PP$ is norm-dense
in $\Aa$.

\begin{lemma} \label{lemfamilydense}
For any finite set $F \subset [0,\infty)$ let
$\PP^{(F)}$ denote those elements of $\PP$ which
do not use any time indices from $F$.  Then for
all finite $F \subset [0,\infty)$, $\psi\big(P^{(F)} \big) VH$
is dense in $\Hh$.
\end{lemma}

\begin{proof}
Consider a general vector of the form
$\sigma_{t_1}(i(a_1)) \cdots \sigma_{t_n}(i(a_n)) V h$,
which we already know to be total in $\Hh$.  We proceed
by induction on $n$.  In the case $n = 1$ we have for
any $\tilde{t}$ that
\[
\|(\sigma_t(i(a)) - \sigma_{\tilde{t}}(i(a)) ) Vh\|^2
= \Ss \lla t; a^*a \rra - \Ss\lla \tilde{t}, t; a^*, a \rra
- \Ss \lla t, \tilde{t}; a^*, a \rra + \Ss \lla \tilde{t}; a^* a\rra.
\]
As $\tilde{t} \to t$, this approaches zero by
the continuity properties of $\Ss$.  Inductively,
we can approximate $\sigma_{t_2}(i(a_2))
\cdots \sigma_{t_n}(i(a_n)) Vh$ by a vector in
$\psi(\PP^{(F)}) VH$, which we then use as our
$h$ and proceed as before.
\end{proof}

Before establishing our main continuity result, one more preliminary is needed.

\begin{proposition} \label{propseparabilityofHh}
The Hilbert space $\Hh$ is separable.
\end{proposition}

\begin{proof}
Let $H_0$ be a countable dense subset of $H$,
and $\AAA_0$ a countable ultraweakly dense subset
of $\AAA$.  We may assume WLOG that $\AAA_0$ is
a self-adjoint $\Q$-subalgebra, so that its unit ball
is strongly dense in the unit ball of $\AAA$
by Kaplansky's theorem.
% Directly, Kaplansky density would apply to the $\com$-subalgebra generated by $\AAA_0$, rather than $\AAA_0$ itself.  It remains, then, to establish that $\AAA_0$ is ultrastrongly dense in the $\com$-algebra it generates.  In fact it is norm dense, because for a finite linear combination of a fixed family of finite products of elements from $\AAA_0$, the norm of the sum is jointly continuous in the coefficients.

We will show that the countable set
\[
\Big\{\psi\big(\sigma_{t_1}(i(x_1)) \dots \sigma_{t_n}(i(x_n))\big)Vh \mid 0 \leq t_1, \dots, t_n \in \Q;
\, x_1, \dots, x_n \in \AAA_0; \, h \in H_0 \Big\}
\]
spans a dense subset of $\Hh$.  We already know that $\psi(\PP) VH$ has dense span, so it suffices to show that vectors in
$\psi(\PP) VH$ can be norm-approximated by vectors
of the prescribed form.  Let $\tau_1, \dots, \tau_n \geq
0$, $y_1, \dots, y_n \in \AAA$, and
$k \in H$.  By the triangle inequality, we have for any $h \in H_0$, any $t_1, \dots, t_n \in \Q_+$, and any $x_1, \dots, x_n \in \AAA_0$ that
\begin{align*}
&\Big\| \psi\big(\sigma_{\tau_1}(i(y_1)) \dots \sigma_{\tau_n}(i(y_n))
\big) Vk - \psi\big(\sigma_{t_1}(i(x_1)) \dots \sigma_{t_n}(i(x_n))\big) Vh \Big\|\\
&\leq \Big\|\psi\big(\sigma_{\tau_1}(i(y_1))\dots \sigma_{\tau_n}(i(y_n))\big)\Big\| \|h-k\|\\
&\qquad + \Big\| \psi \Big[ \sigma_{\tau_1}(i(y_1)) \cdots
\sigma_{\tau_n}(i(y_n)) - \sigma_{\tau_1}(i(x_1)) \cdots
\sigma_{\tau_n}(i(x_n)) \Big] V h \Big\|\\
&\qquad + \Big\| \psi \Big[ \sigma_{\tau_1}(i(x_1)) \cdots
\sigma_{\tau_n}(i(x_n)) - \sigma_{t_1}(i(x_1)) \cdots
\sigma_{t_n}(i(x_n)) \Big] V h \Big\|.
%&+ \Big\|\psi\big(\sigma_{\tau_1}(i(y_1))\dots
%\sigma_{\tau_n}(i(y_n)) - \sigma_{t_1}(i(x_1))
%\dots \sigma_{t_n}(i(x_n)) \big) V h \Big\|.
%+\Big\| \psi\big( \sigma_{\tau_1}(i(y_1-x_1))\dots
%\sigma_{\tau_n}(i(y_n-x_n)) \big) Vh \Big\|\\
%+\Big\| \psi\big( \sigma_{\tau_1}(i(x_1))\dots
%\sigma_{\tau_n}(i(x_n)) - \sigma_{t_1}(i(x_1))
%\dots \sigma_{t_n}(i(x_n)) \big) Vh \Big\|.
\end{align*}
The first term can be made small by choosing $h$
sufficiently close to $k$.  For the second,
 note that each composition $\psi \circ \sigma_t \circ i$ is normal, since it equals the composition
$\widetilde{\sigma}_t \circ \psi \circ i$; hence
$\psi(\sigma_t(i(\AAA_0)))$ is weakly
dense in $\psi(\sigma_t(i(\AAA)))$.  By Kaplansky's theorem,
it follows that the unit ball of $\psi(\sigma_t(i(\AAA_0)))$
is strongly dense in the unit ball of $\psi(\sigma_t(i(\AAA)))$; this plus the
joint strong continuity of multiplication implies that \\
${\Big\{\psi\big(\sigma_{t_1}(i(x_1))
\dots \sigma_{s_n}(i(x_n)) \big) \mid  s_1, \dots, s_n \geq 0; \ x_1, \dots, x_n \in \AAA_0\Big\}}$ is
strongly dense in \\ ${\Big\{\psi\big(\sigma_{t_1}(i(y_1))
\dots \sigma_{t_n}(i(y_n)) \big) \mid  t_1, \dots, t_n \geq 0; \ y_1, \dots, y_n \in \AAA\Big\}}$.  Hence, once $h$ has
  been fixed, an appropriate choice of $x_1, \dots, x_n$ makes the second term arbitrarily small.  So far we have shown that vectors of the form
\begin{equation} \label{eqnvectorform}
\psi \big(\sigma_{\tau_1}(i(x_1)) \cdots \sigma_{\tau_n}(i(x_n))
\big) Vh \qquad \tau_1, \dots, \tau_n \geq 0; \,
x_1, \dots, x_n \in \AAA_0; \, h \in H_0
\end{equation}
are total in $\Hh$.  It remains to prove
that such vectors remain total under the added restriction
that the $\tau_i$ be rational.  Let $\xi \in \psi(\PP) VH$
be orthogonal to all vectors of the form (\ref{eqnvectorform}).
That is, we let $z_1, \dots, z_m \in \AAA$,
$\eta \in H$, and $s_1, \dots, s_m \geq 0$ such that,
for all $x_1, \dots, x_n \in \AAA_0$, all
$0 \leq t_1, \dots, t_n \in \Q$, and all $h \in H_0$,
\begin{align*}
0 &= \left\la
\psi \big(\sigma_{t_1}(i(x_1)) \cdots \sigma_{t_n}(i(x_n))\big)
Vh, \psi \big(\sigma_{s_1}(i(z_1)) \cdots \sigma_{s_m}(i(z_m))
\big) V \eta \right\ra\\
&= \left\la V^* \psi \big( \sigma_{s_m}(i(z_m^*))
\cdots \sigma_{s_1}(i(z_1^*)) \sigma_{t_1}(i(x_1)) \cdots \sigma_{t_n}(i(x_n)) \big) Vh, \eta \right\ra\\
&= \la \pi (\Ss \lla \vec{s}^{\, *} \vee \vec{t}; \vec{z}^{\, *} \vee
\vec{x} \rra) h, \eta \ra
\end{align*}
where we introduce the
notation $(s_1, \dots, s_m)^* = (s_m, \dots, s_1)$
for $s_1, \dots, s_m \geq 0$
and $(z_1, \dots, z_m)^* = (z_m^*, \dots, z_1^*)$
for $z_1, \dots, z_m \in \AAA$.  Now for any $\vec{t} \in [0,\infty)^n$, let $\{\vec{t}_k \} \subset \Q_+^n$
such that $\vec{s}^{\, *} \vee \vec{t}_k \to \vec{s}^{\, *} \vee \vec{t}$ non-crossingly; then
by Proposition
(\ref{propjointcontinuitymoments}),
\[
\la \pi(\Ss \lla \vec{s}^{\, *} \vee \vec{t}; \vec{z}^{\, *} \vee
\vec{x} \rra) h, \eta \ra
= \lim_{k \to \infty}
\la \pi(\Ss \lla \vec{s}^{\, *} \vee \vec{t}_k; \vec{z}^{\, *} \vee
\vec{x}\rra ) h, \eta \ra = 0.
\]

We thus see that $\xi$ must be orthogonal to a known total
set and hence zero.
\end{proof}

\begin{theorem}  \label{thmultraweakcontinuity}
For any $a \in \widetilde{\Aa}$,
$t \mapsto \widetilde{\sigma}_t(a)$  is ultraweakly continuous
for all $t > 0$.
\end{theorem}

\begin{proof}
We establish this in a series of steps.
\begin{enumerate}
    \item For any $a \in \AAA_0$ and $\xi, \eta \in \psi(\PP) VH$,
    the value of $\la \widetilde{\sigma}_t(\psi(a))
    \xi, \eta \ra = \la \psi(\sigma_t(a)) \xi, \eta \ra$ is
    given by a certain Sauvageot moment polynomial; explicitly, if\\ $a = \sigma_{\tau_1}(i(x_1))
    \cdots \sigma_{\tau_n}(i(x_n))$,
    $\xi = \sigma_{s_1}(i(y_1)) \cdots \sigma_{s_m}
    (i(y_n)) V \xi_0$, and\\
    $\eta = \sigma_{u_1}(i(z_1)) \cdots \sigma_{u_\ell}
    (i(z_\ell)) V \eta_0$, then
    \[
    \la \widetilde{\sigma}_t(\psi(a)) \xi, \eta \ra
    = \la \pi(\Ss \lla \vec{u}^{\, *} \vee (\vec{\tau} +t)
    \vee \vec{s}; \vec{z}^{\, *} \vee
    \vec{x} \vee \vec{y} \rra)  \xi_0, \eta_0 \ra.
    \]

    \item Given $\vec{\tau}$ and a time $t_0\geq 0$, let
    $F$ be the set of times in $\vec{\tau} + t_0$.
    Taking any $\xi_0, \eta_0 \in \psi(P^{(F)}) V H$,
    which is dense by Lemma \ref{lemfamilydense},
    we see by Proposition \ref{propjointcontinuitymoments}
    that the above expression is continuous at $t_0$,
    since if $t \to t_0$ within a sufficiently small
    neighborhood of $t_0$ then $\vec{u}^{\, *} \vee (\vec{\tau}
    + t) \vee \vec{s} \to \vec{u}^{\, *} \vee (\vec{\tau} + t_0)
    \vee \vec{s}$ non-crossingly.
    We therefore have that $t \mapsto \la \widetilde{\sigma}_t(\psi(a))
    \xi, \eta \ra$ is continuous at $t_0$
    for all $\xi, \eta \in
    \psi(\PP^{(F)}) VH$ and all $a \in \AAA_0$.

    \item By Lemma \ref{lemWOTcontinuitydense},
    this implies that $t \mapsto \la \widetilde{\sigma}_t(\psi(a))
    \xi, \eta \ra$ is continuous at $t_0$ for all
    $\xi, \eta \in \Hh$ and all $a \in \AAA_0$.

    \item Now let $a \in \widetilde{\Aa}$.  By Kaplansky
    density, there is a sequence $\{a_n\} \subset \Aa_0$
    such that $\psi(a_n) \to a$ in SOT.  We can use
    a sequence rather than a net because the separability
    of $\Hh$, established
     in Proposition \ref{propseparabilityofHh}, implies the SOT-metrizability of $B(\Hh)$ (\cite{Blackadar} III.2.2.27).  Then for any $\xi, \eta \in \Hh$,
     \[
     \la \widetilde{\sigma}_t(a) \xi, \eta
     \ra = \lim_n \la \widetilde{\sigma}_t(\psi(a_n))
     \xi, \eta \ra
     \]
     so that the left-hand side, as a function
      of $t$, is a pointwise limit of
     a sequence of continuous functions, hence
     measurable.  That is,
     $t \mapsto \widetilde{\sigma}_t(a)$ is WOT-measurable;
     as the $\widetilde{\sigma}$ are contractions and the WOT
     agrees with the ultraweak topology on bounded subsets,
     $t \mapsto \widetilde{\sigma}_t(a)$ is ultraweakly
     measurable at all $t \geq 0$.

     \item Since each $\widetilde{\sigma}_t$ is normal,
     there is a corresponding preadjoint semigroup
     $\{\rho_t\}$ on $\widetilde{\Aa}_*$ given
     by $\rho_t f = f \circ \widetilde{\sigma}_t$,
     as discussed in section \ref{secC0semigroups},
     such that for each $f \in \widetilde{\Aa}_*$,
     $t \mapsto \rho_t(f)$ is weakly measurable
     at all $t \geq 0$.
     %or more explicitly,
%     \begin{equation} \label{eqnpreadjoint}
%     \forall f \in \widetilde{\Aa}_*: \
%     \forall a \in \widetilde{\Aa}: \
%     \forall t \geq 0: \
%     (\rho_t f)(a) = f(\widetilde{\sigma}_t(a)).
%     \end{equation}
%     Both sides of this equation are functions from
%     $[0,\infty)$ to $\com$ for fixed $f$ and $a$,
%     the measurability of which corresponds to the ultraweak measurability
%     of $\widetilde{\sigma}_t$ and to the weak measurability
%     of $\rho_t$.  Hence, $\{\rho_t\}$ is weakly
%     measurable.

     \item Since $\Hh$ is separable and
     $\widetilde{\Aa} \subset B(\Hh)$, it follows
     that $\widetilde{\Aa}_*$ is a separable Banach
     space.  By section \ref{secC0semigroups},
     the weak measurability
     of $\{\rho_t\}$ is therefore equivalent to its
     weak continuity at times $t > 0$.
     This is then equivalent to the
     ultraweak continuity of $t \mapsto \widetilde{\sigma}_t$.
 \end{enumerate}
\end{proof}

\begin{theorem} \label{thmextensionstrongdilation}
$(\widetilde{\Aa}, \psi \circ i, \widetilde{\E},
\{\widetilde{\sigma}_t\})$ is a strong dilation
of $(\AAA, \{\phi_t\})$.
\end{theorem}

\begin{proof}
By the definition of
$\widetilde{\E}$, equation \ref{eqncovarianceofextension},
 and theorem \ref{thmstrongdilation},
\begin{align*}
\widetilde{\E} \circ \widetilde{\sigma}_t \circ \psi
&=\widetilde{\E} \circ \psi \circ \sigma_t
= \E \circ \sigma_t \\
&= \phi_t \circ \E
= \phi_t \circ \widetilde{\E} \circ \psi.
\end{align*}
Since both $\phi_t \circ \widetilde{\E}$ and
$\widetilde{\E} \circ \widetilde{\sigma}_t$ are normal, and
since they are equal on the ultraweakly dense subset $\psi(\Aa)
\subset \widetilde{\Aa}$, they must be equal.
\end{proof}

So far, theorem \ref{thmultraweakcontinuity}
leaves open the question whether
$\{\widetilde{\sigma}_t\}$ is point-weakly continuous
at $t = 0$.  Without resolving that question,
we can still remedy the situation by taking a suitable quotient.

\begin{lemma} \label{lemcontinuityat0}
Let $A$ be a separable W$^*$-algebra and $\{\alpha_t\}$
an e$_0$-semigroup on $A$ which is point-weakly
continuous at all $t > 0$.  Then $\alpha_t$
is point-weakly continuous at 0 iff
\[
\bigcap_{t > 0} \ker \alpha_t = \{0\}.
\]
\end{lemma}

\begin{proof}
The point-weak continuity of $\alpha_t$ at $t=0$ is equivalent
to the weak (equivalently, strong) continuity at $t=0$ of
the preadjoint semigroup $\{\rho_t\}$ on $A_*$ defined
by $\rho_t f = f \circ \alpha_t$.  As mentioned
in section \ref{secC0semigroups}, this is equivalent to
the condition
\[
\overline{ \bigcup_{t > 0} \rho_t A_*} = A_*
\]
since $A_*$ is assumed separable.  Now
the annihilator of the left-hand side is
\begin{align*}
\overline{ \bigcup_{t > 0} \rho_t A_*}^\perp
&= \{a \in A \mid \forall t > 0: \ \forall f \in A_*:\
(\rho_t f)(a) = 0\}\\
&= \{a \in A \mid \forall t > 0: \ \forall f \in A_*: \
f(\alpha_t(a)) = 0\}\\
&= \{a \in A \mid \forall t > 0: \alpha_t(a) = 0\}\\
&= \bigcap_{t > 0 } \ker \alpha_t
\end{align*}
because $A_*$ separates points on $A$.
\end{proof}

\newpage
\begin{theorem} \label{thmkahuna}
Let $\AAA$ be a separable W$^*$-algebra and $\{\phi_t\}$
a CP$_0$-semigroup on $\AAA$.  Then there exists a
unital strong dilation of $\{\phi_t\}$ to an E$_0$-semigroup on a
separable W$^*$-algebra.
\end{theorem}

\begin{proof}
The dilation $(\widetilde{\Aa}, \psi \circ i,
\widetilde{\E}, \{\widetilde{\sigma}_t\})$ constructed
in this chapter satisfies all the requirements
except possibly point-ultraweak continuity at $t=0$.

We now let
\[
\RR = \bigcap_{t > 0} \ker \widetilde{\sigma}_t.
\]
This is an ultraweakly closed ideal in $\widetilde{\Aa}$;
we use $\widehat{\Aa}$ for the quotient
$\widetilde{\Aa}/\RR$, which is another
separable W$^*$-algebra.  Because
$\widetilde{\sigma}_t(\RR) \subset
\RR$ for each $t > 0$, we obtain for each
$t > 0$ a map $\widehat{\sigma}_t: \widehat{\Aa} \to\widehat{\Aa}$
characterized by the commutative diagram
\[ \xymatrix{
\widetilde{\Aa} \ar[d] \ar[r]^{\widetilde{\sigma}_t} &
\widetilde{\Aa} \ar[d] \\
\widehat{\Aa} \ar[r]_{\widehat{\sigma}_t} &
\widehat{\Aa}
}\]
Defining also $\widehat{\sigma}_0 = \dss{\text{id}}{\widehat{\Aa}}$, we see that $\{\widehat{\sigma}_t\}$ inherits from
$\{\widetilde{\sigma}_t\}$ the properties of being an
e$_0$-semigroup and of point-ultraweak continuity
at $t > 0$.  Furthermore, \\
$\displaystyle\bigcap_{t > 0}
\ker \widehat{\sigma}_t = \{0\}$ by construction,
so that $\{\widehat{\sigma}_t\}$ is point-weakly continuous at $t=0$ and hence is an E$_0$-semigroup.  Our embedding
of $\AAA$ into $\widehat{\Aa}$ is given by
$q \circ \psi \circ i$, where $\widetilde{\Aa} \sa{q}
\widehat{\Aa}$ is the quotient map; this is injective because,
if $a \in \AAA$ is such that $q(\psi(i(a)))=0$, then
$\psi(i(a)) \in \RR$, so that for all $t > 0$ one has
\begin{align*}
\widetilde{\sigma}_t(\psi(i(a))) &= 0\\
\psi(\sigma_t(i(a))) &= 0\\
\sigma_t(i(a)) &= 0\\
\E[ \sigma_t(i(a))] &= 0\\
\phi_t(\E[i(a)]) &= 0\\
\phi_t(a) &= 0
\end{align*}
and since $\phi_t(a) \to a$ as $t \to 0^+$ this implies
$a = 0$.  To construct our retraction, we first note that
$\RR \subset \ker \widetilde{\E}$; indeed, if
$a \in \RR$ then for all $t > 0$ we have
\begin{align*}
\widetilde{\sigma}_t(a) &= 0\\
\widetilde{\E} \circ \widetilde{\sigma}_t(a) &= 0\\
\phi_t \circ \widetilde{\E}(a) &= 0
\end{align*}
and by letting $t \to 0^+$ we conclude $\widetilde{\E}(a) = 0$.
Hence, $\ker q \subset \ker \widetilde{\E}$, so there is
a unique map $\widehat{\E}: \widehat{\Aa} \to \AAA$
with $\widetilde{\E} = \widehat{\E} \circ q$.
% AGENDUM: Does this only get you linearity of \widehat{\E}?  Prove it's a normal retraction, yada yada
This map satisfies $\widehat{\E} \circ q \circ \psi \circ i
= \widetilde{\E} \circ \psi \circ i = \dss{\text{id}}{\AAA}$,
so it is a retraction with respect to the given embedding.
Finally,
\[
\widehat{\E} \circ \widehat{\sigma}_t \circ
q = \widehat{\E} \circ q \circ \widetilde{\sigma}_t
= \widetilde{\E} \circ \widetilde{\sigma}_t
= \phi_t \circ \widetilde{\E}
= \phi_t \circ \widehat{\E} \circ q,
\]
and since the image of $q$ generates $\widehat{\Aa}$ this
implies $\widehat{\E} \circ \widehat{\sigma}_t =
\phi_t \circ \widehat{\E}$.  We therefore have
a strong dilation of the original semigroup.
\end{proof}

\appendix
%    Include appendix "chapters" here.
% Thesis Appendix: Moment Tables by Dave Gaebler
% Some of the details behind these calculations are in the ``Knell'' document
% Version 2 changes: no longer assuming \nu \circ \rho = \nu, so more complicated expressions in the moment tables
\chapter{Table of Moments} \label{appendixmomenttables}

\noindent Here $w_\ell$ denotes
the tuple $(b_0, a_1, b_1, \dots, a_\ell, b_\ell)$.  The function $\MM_0$ in the tables is related to $\MM$  by the
relation $\MM(w_\ell) = b_0 \MM_0(w_\ell)
b_\ell$.  We also let
\begin{align*}
\rho(a_{i,j}) &= \rho(a_i a_j) - \rho(a_i) \rho(a_j)\\
\rho(a_{i,j,k}) &= \rho(a_i a_j a_k)
- \rho(a_i) \rho(a_j a_k) - \rho(a_i a_j) \rho(a_k)
+ \rho(a_i) \rho(a_j) \rho(a_k)\\
\rho(a_{i,j,k,l}) &= \rho(a_i a_j a_k a_l)
- \rho(a_i) \rho(a_j a_k a_l) - \rho(a_i a_j) \rho(a_k a_l)
-\rho(a_i a_j a_k) \rho(a_l) + \rho(a_i a_j)
\rho(a_k) \rho(a_l)\\
&\quad + \rho(a_i) \rho(a_j a_k) \rho(a_l)
+ \rho(a_i) \rho(a_j) \rho(a_k a_l) -
\rho(a_i) \rho(a_j) \rho(a_k) \rho(a_l)
\end{align*}

\section*{General Case}
%\newpage

\begin{align*}
\MM_0(w_1) &= \rho(a_1)\\
\MM_0(w_2) &= \rho(a_1) b_1 \rho(a_2) + \rho \big[ a_1 \psi(b_1) a_2 \big] - \rho(a_1) \rho \big[ \psi(b_1) a_2 \big]
- \rho \big[ a_1 \psi(b_1) \big] \rho(a_2)
+ \rho(a_1) \rho \big[ \psi(b_1) \big] \rho(a_2)\\
\MM_0(w_3) &=
% These are grouped, at the highest level, according
% to which a's appear ``by themselves'' inside of a rho.
\rho(a_1) b_1 \rho(a_2) b_2 \rho(a_3) %Begin group "123"
+ \rho(a_1) \rho \big[ \psi(b_1) \big] \rho(a_2) b_2 \rho(a_3)
+ \rho(a_1) b_1 \rho(a_2) \rho \big[ \psi(b_2) \big] \rho(a_3)\\
&\qquad+ \rho(a_1) \rho \big[ \psi(b_1) \big] \rho(a_2)
\rho \big[\psi(b_2) \big] \rho(a_3) % End group "123"
- \rho(a_1) b_1 \rho(a_2) \rho \big[ \psi(b_2) a_3 \big]\\ %Begin group "12"
&\qquad- \rho(a_1) \rho \big[ \psi(b_1) \big] \rho(a_2)
\rho \big[\psi(b_2) a_3 \big] %End group "12"
- \rho \big[ a_1 \psi(b_1) \big] \rho(a_2) b_2 \rho(a_3)\\ %Begin group "23"
&\qquad - \rho \big[ a_1 \psi(b_1) \big] \rho(a_2)
\rho \big[ \psi(b_2) \big] \rho(a_3) % End group "23"
 + \rho(a_1) \rho \big[ \psi(b_1) a_2 \psi(b_2) \big]
\rho(a_3) \\% Begin group "13both".  Terms of type 13 are secondarily grouped according to which of b_1 and b_3 appear in the same rho-expression as a_2.
&\qquad- \rho(a_1) \rho \Big( \psi(b_1) \psi \big[ \rho(a_2) b_2 \big] \Big) \rho(a_3)
+ \rho(a_1) \rho \Big( \psi(b_1) \psi \big[
\rho(a_2) \big] \psi(b_2)\Big) \rho(a_3)\\
&\qquad- \rho(a_1) \rho \Big( \psi \big[ b_1 \rho(a_2) \big]
\psi(b_2) \Big) \rho(a_3)
+ \rho(a_1) \rho \Big( \psi \big[ b_1 \rho(a_2)
b_2 \big] \Big) \rho(a_3) \\% End group "13both"
&\qquad+ \rho(a_1) \rho \big[ \psi(b_1) \big] \rho \Big( \psi \big[
\rho(a_2) \big] \Big) \rho \big[ \psi(b_2) \big] \rho(a_3) %Group "13neither" is this single term
- \rho(a_1) b_1 \rho \big[ a_2 \psi(b_2) \big] \rho(a_3)\\ %Begin group "13right"
&\qquad- \rho(a_1) \rho \big[ \psi(b_1) \big] \rho \Big( \psi \big[
\rho(a_2) \big] \psi(b_2) \Big) \rho(a_3)
+ \rho(a_1) \rho \big[ \psi(b_1) \big] \rho \Big( \psi \big[
\rho(a_2) b_2 \big] \Big) \rho(a_3)\\
&\qquad- \rho(a_1) \rho \big[ \psi(b_1) \big] \rho \big[ a_2
\psi(b_2) \big] \rho(a_3) %End group "13right"
+ \rho(a_1) \rho \Big( \psi \big[ b_1 \rho(a_2)
\big] \Big) \rho \big[ \psi(b_2) \big] \rho(a_3) \\%Begin group "13left"
&\qquad- \rho(a_1) \rho \big[ \psi(b_1) a_2 \big] b_2 \rho(a_3)
- \rho(a_1) \rho \big[ \psi(b_1) a_2 \big] \rho \big[ \psi(b_2)
\big] \rho(a_3)\\
&\qquad- \rho(a_1) \rho \Big( \psi(b_1) \psi \big[ \rho(a_2) \big] \Big) \rho \big[ \psi(b_2) \big] \rho(a_3) %End group "13left"
+ \rho \big[ a_1 \psi(b_1) \big] \rho(a_2) \rho \big[ \psi(b_2) a_3 \big] \\ 
&\qquad -\rho(a_1) \rho \big[ \psi(b_1) a_2 \psi(b_2) a_3 \big] %Begin group "1one"
+ \rho(a_1) b_1 \rho \big[ a_2 \psi(b_2) a_3 \big]\\
&\qquad - \rho(a_1) \rho \Big( \psi(b_1) \psi
\big[ \rho(a_2) \big]
\psi(b_2) a_3 \Big) + \rho(a_1) \rho \Big( \psi \big[
b_1 \rho(a_2) \big] \psi(b_2) a_3 \Big)\\
&\qquad - \rho(a_1) \rho \Big( \psi \big[ b_1 \rho(a_2) b_2
\big] a_3 \Big)
+ \rho(a_1) \rho \Big( \psi(b_1) \psi \big[ \rho(a_2)
_2 \big] a_3 \Big)+ \cdots
\end{align*} %End group "1one"

\begin{align*}
\MM_0(w_3) &=\cdots + \rho(a_1) \rho \big[ \psi(b_1) \big] \rho \big[
a_2 \psi(b_2) a_3 \big] %Begin group "1two"
+ \rho(a_1) \rho \big[ \psi(b_1) \big] \rho \Big( \psi \big[
\rho(a_2) \big] \psi(b_2) a_3 \Big)\\
&\qquad + \rho(a_1) \rho \Big( \psi(b_1) \psi \big[ \rho
(a_2) \big] \Big) \rho \big[ \psi(b_2) a_3 \big]
+ \rho(a_1) \rho \big[ \psi(b_1) a_2 \big] \rho \big[ \psi
(b_2) a_3 \big]\\
&\qquad - \rho(a_1) \rho \Big( \psi \big[ b_1 \rho(a_2)
\big] \Big) \rho \big[ \psi(b_2) a_3 \big]
- \rho(a_1) \rho \big[ \psi(b_1) \big] \rho \Big( \psi \big[
\rho(a_2) b_2 \big] a_3 \Big)\\ %End group "1two"
&\qquad - \rho(a_1) \rho \big[ \psi(b_1) \big] \rho \Big( \psi \big[ \rho(a_2) \big] \Big) \rho \big[ \psi(b_2) a_3 \big] %Group "1three" is this single term
+ \rho \big[ a_1 \psi(b_1) a_2 \big] b_2 \rho(a_3) \\%Begin group "3one"
&\qquad- \rho \Big( a_1 \psi \big[ b_1 \rho(a_2) b_2 \big] \Big) \rho(a_3) + \rho \Big( a_1 \psi \big[ b_1 \rho(a_2) \big] \Big) \rho(a_3)\\
&\qquad - \rho \big[ a_1 \psi(b_1) a_2 \psi(b_2) \big] \rho(a_3)
+ \rho \Big( a_1 \psi(b_1) \psi \big[ \rho(a_2) b_2 \big] \Big) \rho(a_3)\\
&\qquad - \rho \Big( a_1 \psi(b_1) \psi \big[ \rho(a_2) \big]
\psi(b_2) \Big) \rho(a_3) %End group "3one"
+ \rho \big[ a_1 \psi(b_1) a_2 \big] \rho \big[ \psi(b_2) \big]
\rho(a_3)\\ %Begin group "3two"
&\qquad + \rho \big[ a_1 \psi(b_1) \big] \rho \big[ a_2 \psi(b_2) \big] \rho(a_3) - \rho \big[ a_1 \psi(b_1) \big] \rho \Big(
\psi \big[ \rho(a_2) b_2 \big] \Big) \rho(a_3)\\
&\qquad - \rho \Big( a_1 \psi \big[ b_1 \rho(a_2) \big] \Big)
\rho \big[ \psi(b_2) \big] \rho(a_3)
+ \rho \Big( a_1 \psi(b_1) \psi \big[ \rho(a_2) \big] \Big)
\rho \big[ \psi(b_2) \big] \rho(a_3)\\
&\qquad + \rho \big[ a_1 \psi(b_1) \big]
\rho \Big( \psi \big[ \rho(a_2) \big] \psi(b_2) \Big) \rho(a_3) %End group "3two"
- \rho \big[ a_1 \psi(b_1) \big] \rho \Big( \psi \big[
\rho(a_2) \big] \Big) \rho \big[ \psi(b_2) \big] \rho(a_3)\\
%Group "3three" is this single term
&\qquad + \rho \big[ a_1 \psi(b_1) a_2 \psi(b_2) a_3 \big] %Begin group "0one".  The secondary grouping within group 0 is based on how many rho-expressions there are.
+ \rho \Big( a_1 \psi(b_1) \psi \big[ \rho(a_2) \big] \psi(b_2)
a_3 \Big)\\
&\qquad - \rho \Big( a_1 \psi \big[ b_1 \rho(a_2) \big] \psi(b_2) a_3
\Big) - \rho \Big( a_1 \psi(b_1) \psi \big[ \rho(a_2) b_2 \big] a_3 \Big)\\
&\qquad + \rho \Big( a_1 \psi \big[ b_1 \rho(a_2)
b_2 \big] a_3 \Big) %End group "0one"
+ \rho \big[ a_1 \psi(b_1) \big] \rho \Big( \psi \big[
\rho(a_2) b_2 \big] a_3 \Big)\\ %Begin group "0two"
&\qquad + \rho \Big( a_1 \psi \big[ b_1 \rho(a_2) \big] \Big)
\rho \big[ \psi(b_2) a_3  \big]
- \rho \big[a_1 \psi(b_1) \big] \rho \big[ a_2 \psi(b_2) a_3 \big]\\
&\qquad - \rho \big[ a_1 \psi(b_1) \big]
\rho \Big( \psi \big[ \rho(a_2) \big] \psi(b_2) a_3 \Big)
- \rho \Big( a_1 \psi(b_1) \psi \big[ \rho(a_2) \big] a_3\Big)\\
&\qquad - \rho \big[ a_1 \psi(b_1) a_2 \big] \rho \big[ \psi(b_2) a_3 \big] + \rho \big[ a_1 \psi(b_1) \big] \rho \Big( \psi \big[ \rho(a_2) \big] \Big) \rho \big[ \psi(b_2) a_3 \big]
\end{align*}

\newpage
\section*{Special Case: Scalar-Valued $\psi$}
We list here values of the moment function in the special
case where the map from $B$ to $A$ is scalar-valued.  To reflect
this extra assumption, we denote that map by $\nu$ rather than $\psi$.  Then

\begin{align*}
\MM_0(w_1) &= \rho(a_1)\\
\MM_0(w_2) &= \nu(b_1)  \rho(a_{1,2})+ \rho(a_1) b_1 \rho(a_2)\\
\MM_0(w_3) &=   \rho(a_1) b_1 \rho(a_2) b_2 \rho(a_3)
 +\nu(b_1) \nu(b_2)   \rho(a_{1,2,3}) + \nu(b_1)   \rho(a_{1,2})
b_2 \rho(a_3)   + \nu(b_2)   \rho(a_1) b_1 \rho(a_{2,3})\\
&\quad + \Big[ \nu(b_1) \nu(b_2) \nu\big(\rho(a_2)\big) -
\nu(b_1) \nu \big( \rho(a_2) b_2 \big)
- \nu\big(b_1 \rho(a_2)\big) \nu(b_2) + \nu \big( b_1 \rho(a_2) b_2 \big) \Big]  \rho(a_{1,3}) \\
\MM_0(w_4) &=   \rho(a_1) b_1 \rho(a_2) b_2 \rho(a_3) b_3 \rho(a_4) + \nu(b_1)\nu(b_2) \nu(b_3)   \rho(a_{1,2,3,4})  \\
&+ \nu(b_1) \nu(b_2) \nu(b_3) \nu(\rho(a_2))   \rho(a_{1,3,4})
+ \nu(b_1) \nu(b_2) \nu(b_3) \nu(\rho(a_3))   \rho(a_{1,2,4}) \\
&+ \nu(b_1) \nu(b_3) \nu(b_3) \Big[ \nu(\rho(a_2 a_3)) + 2 \nu(\rho(a_2))
\nu(\rho(a_3)) - \nu \big( \rho(a_2) \rho(a_3) \big) \Big]
\rho(a_{1,4})  \\
&+ \nu(b_1) \nu(b_2) \nu(\rho(a_2))   \rho(a_{1,3}) b_3
\rho(a_4)  - \nu(b_1) \nu(b_2) \nu \big( \rho(a_3) b_3 \big)
\dss{\rho}{[1,2,4]}(a)  \\
&- \nu(b_1) \nu(b_2)  \Big[ \nu \big( \rho(a_{2,3}) b_3
\big) - 2 \nu(\rho(a_2)) \nu \big( \rho(a_3) b_3 \big) \Big]
  \rho(a_{1,4})  \\
&+\nu(b_1) \nu(b_3)   \rho(a_{1,2}) b_2
\rho(a_{3,4})   - \nu(b_1) \nu(b_3) \nu \big( \rho(a_2) b_2 \big) \rho(a_{1,3,4})  \\
&- \nu(b_1) \nu(b_3) \nu \big(b_2 \rho(a_3) \big)
\rho(a_{1,2,4})  \\
&+ \nu(b_1) \nu(b_3) \Big[ \nu \big(\rho(a_2) b_2 \rho(a_3) \big)
- \nu \big( \rho(a_2) b_2 \big) \nu(\rho(a_3))
- \nu(\rho(a_2)) \nu \big( b_2 \rho(a_3) \big) \Big]
  \rho(a_{1,4})  \\
&+ \nu(b_2) \nu(b_3)   \rho(a_1) b_1 \rho(a_{2,3,4})
- \nu(b_2) \nu(b_3) \nu \big( b_1 \rho(a_2) \big)
\rho(a_{1,3,4})  \\
&- \nu(b_2) \nu(b_3) \Big[ \nu \big( b_1 \rho(a_{2,3})
\big) + 2 \nu \big( b_1 \rho(a_2) \big) \nu(\rho(a_3)) \Big]
  \rho(a_{1,4})  \\
&+ \nu(b_1)   \rho(a_{1,2}) b_2 \rho(a_3)
b_3 \rho(a_4)
- \nu(b_1) \nu \big( \rho(a_2) b_2 \big)
\rho(a_{1,3}) b_3 \rho(a_4)  \\
&+ \nu(b_1) \nu \big( b_2 \rho(a_3) b_3 \big)
\rho(a_{1,2,4})  \\
&- \nu(b_1) \Big[ \nu \big( \rho(a_2) b_2 \rho(a_3) b_3
\big) - \nu(\rho(a_2)) \nu \big( b_2 \rho(a_3) b_3 \big)
- \nu \big( \rho(a_2) b_2 \big)
\nu \big( \rho(a_3) b_3 \big) \Big]
\rho(a_{1,4})  \\
&+ \nu(b_2)   \rho(a_1) b_1 \rho(a_{2,3}) b_3
\rho(a_4)   - \nu(b_2) \nu \big( b_1 \rho(a_2) \big)
\rho(a_{1,3}) b_3 \rho(a_4)  \\
&- \nu(b_2) \nu \big( \rho(a_3) b_3 \big)
\rho(a_1) b_1 \rho(a_{2,4})  \\
&+ \nu(b_2) \Big[ \nu \big( b_1 \rho(a_{2,3}) b_3 \big)
+ 2 \nu \big( b_1 \rho(a_2) \big)
\nu \big( \rho(a_3) b_3 \big) \Big]
\rho(a_{1,4})  \\
&+ \nu(b_3)   \rho(a_1) b_1 \rho(a_2) b_2
\rho(a_{3,4})
+ \nu(b_3) \nu \big( b_1 \rho(a_2) b_2 \big)
\rho(a_{1,3,4})  \\
&- \nu(b_3) \nu \big( b_2 \rho(a_3) \big)   \rho(a_1)
b_1 \rho(a_{2,4})  \\
&- \nu(b_3) \Big[ \nu \big( b_1 \rho(a_2) b_2 \rho(a_3)
\big) - \nu \big( b_1 \rho(a_2) b_2 \big) \nu(\rho(a_3))
- \nu \big( b_1 \rho(a_2) \big) \nu \big(b_2 \rho(a_3)\big)
\Big]   \rho(a_{1,4})  \\
&+ \Big[ \nu \big( b_1 \rho(a_2) b_2 \rho(a_3) b_3 \big)
- \nu \big( b_1 \rho(a_2) b_2 \big) \nu \big(
\rho(a_3) b_3 \big) - \nu \big( b_1 \rho(a_2) \big)
\nu \big( b_2 \rho(a_3) b_3 \big) \Big]
\rho(a_{1,4})  \\
&+ \nu \big( b_1 \rho(a_2) b_2 \big)
\rho(a_{1,3}) b_3 \rho(a_4)
+ \nu \big( b_2 \rho(a_3) b_3 \big)  
\rho(a_1) b_1 \rho(a_{2,4})  
\end{align*}
\chapter{Table of Moment Polynomials} \label{chapmomentpolynomials}
For the sake of brevity, we use $1,2,3$ to denote
$t_1, t_2, t_3$, with the standing assumption
that $0 < t_1 < t_2 < t_3$, and omit listing $a_1, \dots, a_n$;
  hence $\Ss(1,0,3,2)$ is an abbreviation for
  $\Ss(t_1, 0, t_3, t_2; a_1, a_2, a_3, a_4)$,
  and $\phi_{2-1}$ for $\phi_{t_2-t_1}$.  After the first few, we omit
polynomials in which 0 appears as the first or last index,
since the bimodule property easily reduces these to others, viz. \begin{align*}
\Ss(0, s_1, \dots, s_k; a_0, a_1, \dots, a_k) &= a_0 \phi_\tau
\big(\Ss(s_1-\tau, \dots, s_k-\tau; a_1, \dots, a_k) \big)\\
\Ss(s_1, \dots, s_k, 0;  a_1, \dots, a_k, a_{k+1}) &= \phi_\tau
\big(\Ss(s_1-\tau, \dots, s_k-\tau; a_1, \dots, a_k) \big) a_{k+1}
\end{align*}
where $\tau = \min (s_1, \dots, s_k)$.   We also omit polynomials with consecutive time indices equal, since these can be reduced by multiplying consecutive terms with the same time index; for instance,
\[
\Ss(t_1, t_1, t_2, t_3, t_3; a_1, a_2, a_3, a_4, a_5) = \Ss(t_1, t_2, t_3; a_1 a_2, a_3, a_4 a_5).
\]
Then

\begin{align*}
\Ss(0) &= a_1  \\
\Ss(0,1) &= a_1 \phi_1(a_2)\\
\Ss(1,0) &= \phi_1(a_1) a_2\\
\Ss(0,1,0) &= a_1 \phi_1(a_2) a_3\\
\Ss(1,0,1) &= \phi_1(a_1) a_2 \phi_1(a_3) +\omega(a_2) \big[ \phi_1(a_1 a_3) - \phi_1(a_1) \phi_1(a_3) \big] \\
\Ss(0,1,2) &= a_1 \phi_1 \big( a_2 \phi_{2-1}( a_3) \big)\\
\Ss(0,2,1) &= a_1 \phi_1 \big( \phi_{2-1}(a_2) a_3 \big)\\
\Ss(1,0,2) &= \phi_1(a_1) a_2 \phi_2(a_3) + \omega(a_2) \big[ \phi_1 \big(a_1 \phi_{2-1}(a_3)
\big) - \phi_1(a_1) \phi_2(a_3) \big]\\
\Ss(1,2,0) &= \phi_1 \big( a_1 \phi_{2-1}(a_2)\big)a_3\\
\Ss(2,0,1) &= \phi_2(a_1) a_2 \phi_1(a_3)
+ \omega(a_2) \big[ \phi_1 \big( \phi_{2-1}(a_1) a_3 \big)
- \phi_2(a_1) \phi_1(a_3) \big]\\
\Ss(2,1,0) &= \phi_1 \big( \phi_{2-1}(a_1) a_2 \big) a_3\\
\Ss(1,0,1,2) &= \phi_1(a_1) a_2 \phi_1 \big( a_3 \phi_{3-1}(a_4) \big) + \omega(a_2) \big[ \phi_1 \big( a_1 a_3 \phi_{2-1}(a_4) \big) - \phi_1(a_1) \phi_1 \big( a_3 \phi_{2-1}(a_4) \big) \big]\\
\Ss(1,0,2,1) &= \phi_1(a_1) a_2 \phi_1 \big( \phi_{2-1}(a_3) a_4 \big) +\omega(a_2) \big[ \phi_1 \big( a_1 \phi_{2-1}(a_3) a_4 \big) - \phi_1(a_1) \phi_1 \big( \phi_{2-1}(a_3) a_4 \big)\big]\\
\Ss(1,2,0,1) &= \phi_1 \big( a_1 \phi_{2-1}(a_2)\big)
a_3 \phi_1(a_4) + \omega(a_3) \big[ \phi_1 \big(
a_1 \phi_{2-1}(a_2) a_4 \big) - \phi_1 \big( a_1 \phi_{2-1}(a_2)\big) \phi_1(a_4) \big]\\
\Ss(2,1,0,1) &= \phi_1 \big( \phi_{2-1}(a_1) a_2 \big) a_3 \phi_1(a_4) + \omega(a_3) \big[ \phi_1\big( \phi_{2-1}(a_1)
a_2 a_4 \big) - \phi_1 \big( \phi_{2-1}(a_1) a_2\big)
\phi_1(a_4) \big]\\
\Ss(1,2,0,2) &=\phi_1 \big( a_1 \phi_{2-1}(a_2) \big)
a_3 \phi_2(a_4) + \omega(a_3) \big[ \phi_1 \big( a_1 \phi_{2-1}(a_2a_4) \big) - \phi_1 \big( a_1 \phi_{2-1}(a_3)\big) \phi_2(a_4) \big]
\end{align*}
\begin{align*}
\Ss(2,0,1,2) &= \phi_2(a_1) a_2\phi_1 \big( a_3 \phi_{2-1}(a_4)\big) + \omega(a_2) \big[ \phi_1 \big(
\phi_{2-1}(a_1) a_3 \phi_{2-1}(a_4) \big) \\
&\qquad\qquad \qquad -
\phi_2(a_1) \phi_1 \big( a_3 \phi_{2-1}(a_4)\big)\big]
+ \omega(a_2) \omega(a_3) \big[
\phi_2(a_1 a_4) - \phi_2(a_1) \phi_2(a_4)\big]\\
\Ss(2,0,2,1) &= \phi_2(a_1) a_2 \phi_1 \big( \phi_{2-1}(a_3) a_4 \big) + \omega(a_2) \big[ \phi_1 \big( \phi_{2-1}(a_1 a_3) a_4\big) - \phi_2(a_1) \phi_1 \big( \phi_{2-1}(a_3) a_4\big)\big]\\
\Ss(2,1,0,2) &= \phi_1 \big( \phi_{2-1}(a_1) a_2 \big)
a_3 \phi_2(a_4) + \omega(a_3) \big[ \phi_1 \big( \phi_{2-1}(a_1) a_2 \phi_{2-1}(a_4) \big) \\
&\qquad \qquad \qquad - \phi_1 \big( \phi_{2-1}(a_1) a_2
\big) \phi_2(a_4) \big] + \omega(a_2) \omega(a_3)
\big[ \phi_2(a_1 a_4) - \phi_2(a_1) \phi_2(a_4)\big]\\
\Ss(1,0,2,3) &= \phi_1(a_1) a_2 \phi_2 \big( a_3
\phi_{3-2}(a_4) \big) \\
&\qquad  \qquad + \omega(a_2) \Big[ \phi_1 \big(
a_1 \phi_{2-1} \big( a_3 \phi_{3-2}(a_4)\big)\big)
- \phi_1(a_1) \phi_2 \big( a_3 \phi_{3-2}(a_4)\big)\Big]\\
\Ss(1,0,3,2) &= \phi_1(a_1) a_2 \phi_2 \big( \phi_{3-2}(a_3) a_4 \big) \\
&\qquad  \qquad+ \omega(a_2) \Big[ \phi_1\big(a_1 \phi_{2-1}\big(
\phi_{3-2}(a_3) a_4 \big) \big) - \phi_1(a_1)
\phi_2 \big(\phi_{3-2}(a_3)a_4\big) \Big]\\
\Ss(1,2,0,3) &= \phi_1 \big( a_1 \phi_{2-1}(a_2)\big)
a_3 \phi_3(a_4) \\
&\qquad\qquad+ \omega(a_3) \Big[ \phi_1 \big(
a_1 \phi_{2-1}\big(a_2 \phi_{3-2}(a_4)\big)\big)
- \phi_1 \big( a_1 \phi_{2-1}(a_2) \big) \phi_3(a_4)\Big]\\
\Ss(1,3,0,2) &= \phi_1 \big( a_1 \phi_{2-1}(a_3)\big)
a_3 \phi_2(a_4) \\
&\qquad \qquad+ \omega(a_3) \Big[ \phi_1 \big( a_1
\phi_{2-1} \big( \phi_{3-2}(a_2) a_4 \big) \big)
- \phi_1 \big( a_1 \phi_{3-1}(a_2)\big) \phi_2(a_4) \Big]\\
\Ss(2,0,1,3) &= \phi_2(a_1) a_2 \phi_1 \big( a_3 \phi_{3-1}(a_4)\big) + \omega(a_2) \Big[ \phi_1 \big( \phi_{2-1}(a_1) a_3 \phi_{3-1}(a_4) \big)\\
 &\qquad \qquad - \phi_2(a_1)
\phi_1 \big( a_3 \phi_{3-1}(a_4) \big) \Big] + \omega(a_2) \omega(a_3) \Big[
\phi_2\big(a_1 \phi_{3-2}(a_4)\big) - \phi_2(a_1)
\phi_3(a_4) \Big]\\
\Ss(2,0,3,1) &= \phi_2(a_1) a_2 \phi_1 \big(
\phi_{3-1}(a_3) a_4 \big)\\
&\qquad \qquad + \omega(a_2) \Big[
\phi_1 \big( \phi_{2-1} \big( a_1 \phi_{3-2}(a_3) \big)
a_4 \big) - \phi_2(a_1) \phi_1 \big( \phi_{3-1}(a_3) a_4\big)
\Big]\\
\Ss(2,1,0,3) &= \phi_1 \big( \phi_{2-1}(a_1) a_2 \big)
a_3 \phi_3(a_4) + \omega(a_3) \big[ \phi_1 \big(
\phi_{2-1}(a_1) a_2 \phi_{3-1}(a_2)\big) \\
&\qquad \qquad- \phi_1 \big(
\phi_{2-1}(a_1) a_2 \big) \phi_3(a_4) \Big] + \omega(a_2) \omega(a_3)
\Big[ \phi_2 \big( a_1 \phi_{3-2}(a_4)\big)
- \phi_2(a_1) \phi_3(a_4) \Big]\\
\Ss(2,3,0,1) &= \phi_2 \big( a_1 \phi_{3-2}(a_2)\big)
a_3 \phi_1(a_4) \\
&\qquad \qquad+ \omega(a_3) \Big[ \phi_1 \big(
\phi_{2-1}(a_1 \phi_{3-2}(a_3)\big) a_4 \big)
- \phi_2 \big( a_1 \phi_{3-2}(a_2)\big) \phi_1(a_4)\Big]\\
\Ss(3,0,1,2) &= \phi_3(a_1) \phi_1 \big( a-3 \phi_{2-1}(a_4)\big) + \omega(a_2) \Big[ \phi_1 \big( \phi_{3-1}(a_1) a_3 \phi_{2-1}(a_4)\big) \\
&\qquad \qquad - \phi_3(a_1) \phi_1 \big( a_3 \phi_{2-1}(a_4)\big) \Big] + \omega(a_2) \omega(a_3) \Big[
\phi_2 \big( \phi_{3-2}(a_1) a_4 \big) - \phi_3(a_1)
\phi_2(a_4) \Big]\\
\Ss(3,0,2,1) &= \phi_3(a_1) a_2 \phi_1 \big( \phi_{2-1}(a_3) a_4 \big) \\
&\qquad \qquad + \omega(a_2) \Big[ \phi_1 \big( \phi_{2-1} \big( \phi_{3-2} (a_1) a_3 \big) a_4 \big)
- \phi_3(a_1) \phi_1 \big( \phi_{2-1}(a_3) a_4 \big) \Big]\\
\Ss(3,1,0,2) &= \phi_1 \big( \phi_{3-1}(a_1) a_2 \big)
a_3 \phi_2(a_4) + \omega(a_3) \Big[ \phi_1 \big( \phi_{3-1}(a_1) a_2 \phi_{2-1}(a_4) \big)\\
 &\qquad \qquad- \phi_1 \big( \phi_{3-1}(a_1) a_2 \big) \phi_2(a_4) \Big] + \omega(a_2) \omega(a_3) \Big[
\phi_2 \big( \phi_{3-2}(a_1) a_4 \big) - \phi_3(a_1)
\phi_2(a_4) \Big]\\
\Ss(3,2,0,1) &= \phi_2 \big( \phi_{3-2}(a_1) a_2 \big) a_3
\phi_1(a_4) \\
&\qquad \qquad+ \omega(a_3) \Big[ \phi_1 \big( \phi_{2-1}
\big( \phi_{3-2}(a_1) a_2 \big) a_4 \big)
- \phi_2 \big( \phi_{3-2}(a_1) a_2 \big) \phi_1(a_4)\Big]
\end{align*}

\newpage
\noindent To illustrate possible discontinuity in the time parameters,
we consider the following \\
for $0 < \tau < t_1
< t_2 < t_3$:

\begin{align*}
\Ss(t_1, \tau, t_3, 0, t_2) &= \phi_\tau \big( \phi_{t_1-\tau}(a_1) a_2
\phi_{t_3-\tau}(a_3) \big) a_4 \phi_{t_2}(a_5)\\
&\qquad + \omega(a_2) \phi_\tau \Big( \phi_{t_1-\tau}
\big( a_1 \phi_{t_3-t_1}(a_3) \big) - \phi_{t_1-\tau}(a_1)
\phi_{t_3-\tau}(a_3) \Big) a_4 \phi_{t_2}(a_5) \\
&\qquad + \omega(a_4)\phi_\tau \Big( \phi_{t_1-\tau}(a_1) a_2
 \phi_{t_2-\tau} \big( \phi_{t_3-t_2}(a_3) a_5 \big)
 \Big) \\
&\qquad + \omega(a_2) \omega(a_4)
 \phi_\tau \Big[ \phi_{t_1-\tau} \big( a_1
 \phi_{t_2-\tau} \big( \phi_{t_2-t_2}(a_3) a_5 \big) \big)\\
 &\qquad \qquad \qquad
 - \phi_{t_1-\tau}(a_1) \phi_{t_2-\tau} \big(\phi_{t_3-t_2}(a_3) a_5\big) \Big]\\
&\qquad - \omega(a_4) \phi_\tau \big(\phi_{t_1-\tau}(a_1) a_2
\phi_{t_3-\tau}(a_3) \big) \phi_{t_2}(a_5)\\
&\qquad - \omega(a_2) \omega(a_4)\phi_\tau \Big( \phi_{t_1-\tau}
\big( a_1 \phi_{t_3-t_1}(a_3) \big) - \phi_{t_1-\tau}(a_1)
\phi_{t_3-\tau}(a_3) \Big) \phi_{t_2}(a_5)\\
\Ss(t_1, 0, t_3, 0, t_2) &= \phi_{t_1}(a_1)
a_2 \phi_{t_3}(a_3) a_4 \phi_{t_2}(a_5)\\
&\quad + \omega(a_2) \omega(a_4)
\Big[ \phi_{t_1} \big(a_1 \phi_{t_2-t_1} \big(\phi_{t_3-t_2}(a_3)
\big) a_5 \big) - \phi_{t_1}(a_1) \phi_{t_2} \big(
\phi_{t_3-t_2}(a_3) a_5 \big)\\
&\qquad \qquad \qquad - \phi_{t_1}\big(a_1 \phi_{t_3-t_1}(a_3) \big) \phi_{t_2}(a_5) +\phi_{t_1}(a_1) \phi_{t_3}(a_3) \phi_{t_2}(a_5) \Big]\\
&\quad + \omega(a_2) \Big[ \phi_{t_1}\big(a_1 \phi_{t_3-t_1}(a_3)\big) - \phi_{t_1}(a_1) \phi_{t_3}(a_3) \Big] a_4
\phi_{t_2}(a_5)\\
&\quad + \omega(a_4) \phi_{t_1}(a_1) a_2 \Big[ \phi_{t_2}\big(\phi_{t_3-t_2}(a_3) a_5 \big) - \phi_{t_3}(a_3) \phi_{t_2}(a_5) \Big]\\
&\quad + \Big[\omega(a_2) \omega(a_3)\omega(a_4)
- \omega(a_2) \omega \big(\phi_{t_3}(a_3) a_4 \big)
- \omega \big( a_2 \phi_{t_3}(a_3) \big) \omega(a_4)\\
&\qquad \qquad \qquad + \omega \big( a_2 \phi_{t_3}(a_3) a_4 \big) \Big] \big[ \phi_{t_1}\big(a_1 \phi_{t_2-t_1}(a_5) \big)
- \phi_{t_1}(a_1) \phi_{t_2}(a_5) \big]\\
\Ss(t_1, 0, t_3, 0, t_2) &- \lim_{\tau \to 0^+}
\Ss(t_1, \tau, t_3, 0, t_2) = \Big[\omega(a_2) \omega(a_3)\omega(a_4)
 - \\
 &\qquad \qquad \omega(a_2) \omega \big(\phi_{t_3}(a_3) a_4 \big)
 - \omega \big( a_2 \phi_{t_3}(a_3) \big) \omega(a_4)\\
&\qquad \qquad + \omega \big( a_2 \phi_{t_3}(a_3) a_4 \big) \Big]
\Big[ \phi_{t_1}\big(a_1 \phi_{t_2-t_1}(a_5) \big)
- \phi_{t_1}(a_1) \phi_{t_2}(a_5) \Big]
\end{align*}

\backmatter
%    Bibliography styles amsplain or harvard are also acceptable.
\bibliographystyle{amsalpha}
\bibliography{UnitalDilations}
%    See note above about multiple indexes.
\printindex

\end{document}